\documentclass[12pt]{amsart}

\usepackage{amssymb,amsmath,amsfonts,amsthm,mathrsfs}
\usepackage{fullpage}
\usepackage{thumbpdf}
\usepackage{ifthen}
\usepackage{ifpdf}
\ifpdf
    \usepackage[ pdftex, plainpages = false, pdfpagelabels, 
                 pdfpagelayout = useoutlines,
                 bookmarks,
                 bookmarksopen = true,
                 bookmarksnumbered = true,
                 breaklinks = true,
                 linktocpage,
                 pagebackref,
                 colorlinks = false,  
                 linkcolor = blue,
                 urlcolor  = blue,
                 citecolor = red,
                 anchorcolor = green,
                 hyperindex = true,
                 hyperfigures,
                 pdfauthor = Adam \ Topaz,
                 pdftitle =  Commuting-Liftable \ Subgroups \ of \ Galois \ Groups \ II
                 ]{hyperref} 
\else
    \usepackage[ dvips, 
                 bookmarks,
                 bookmarksopen = true,
                 bookmarksnumbered = true,
                 breaklinks = true,
                 linktocpage,
                 pagebackref,
                 colorlinks = true,
                 linkcolor = blue,
                 urlcolor  = blue,
                 citecolor = red,
                 anchorcolor = green,
                 hyperindex = true,
                 hyperfigures
                 ]{hyperref}
\fi
\usepackage[all]{xy}
\usepackage{color}
\usepackage[lite, backrefs, alphabetic]{amsrefs}

\DeclareMathOperator{\Hom}{Hom}
\DeclareMathOperator{\Gal}{Gal}

\DeclareMathOperator{\res}{res}
\DeclareMathOperator{\Spec}{Spec}

\DeclareMathOperator{\Char}{char}

\newtheorem{thm}{Theorem}

\newtheorem{prop}{Proposition}[section]
\newtheorem{lem}[prop]{Lemma}

\newtheorem{cor}[prop]{Corollary}
\newtheorem*{cor*}{Corollary}
\newtheorem{claim}[prop]{Claim}

\theoremstyle{definition}
\newtheorem{defn}[prop]{Definition}
\newtheorem{example}[prop]{Example}
\newtheorem{remark}[prop]{Remark}

\numberwithin{equation}{section}

\newcommand{\F}{\mathbb{F}}
\newcommand{\Z}{\mathbb{Z}}
\newcommand{\Pbb}{\mathbb{P}}
\newcommand{\Abb}{\mathbb{A}}

\newcommand{\Q}{\mathbb{Q}}
\newcommand{\C}{\mathbb{C}}

\newcommand{\Gc}{\mathcal{G}}

\newcommand{\Oc}{\mathcal{O}}

\newcommand{\Sc}{\mathcal{S}}
\newcommand{\Dc}{\mathcal{D}}

\newcommand{\Vc}{\mathcal{V}}
\newcommand{\Wc}{\mathcal{W}}

\newcommand{\mf}{\mathfrak{m}}

\newcommand{\Ibcl}{\mathbf{I}^{\rm CL}}
\newcommand{\Ibc}{\mathbf{I}^{\rm C}}

\newcommand{\Nfr}{\mathbf{N}}
\newcommand{\Mfr}{\mathbf{M}}

\newcommand{\Kln}{K^{\times \ell^n}}

\newcommand{\Nbar}{\overline{\mathbb{N}}}
\newcommand{\Nb}{{\mathbb{N}}}
\newcommand{\Rbf}{\mathbf{R}}

\newcounter{inlineequation}
\begin{document}

\title{Commuting-Liftable Subgroups of Galois Groups II}
\author{Adam Topaz}
\thanks{Research supported in part by a Benjamin Franklin fellowship from the University of Pennsylvania and in part by NSF postdoctoral fellowship DMS-1304114.}
\address{Department of Mathematics, \vskip0pt
University of California, Berkeley, \vskip0pt 
970 Evans Hall \#3840, \vskip0pt
Berkeley, CA 94720-3840 \vskip0pt
USA}
\date{\today}
\email{atopaz@math.berkeley.edu}
\urladdr{www.math.berkeley.edu/~atopaz}

\begin{abstract}
Let $n$ denote either a positive integer or $\infty$, let $\ell$ be a fixed prime and let $K$ be a field of characteristic different from $\ell$.
In the presence of sufficiently many roots of unity in $K$, we show how to recover some of the inertia/decomposition structure of valuations inside the maximal $\ell^n$-abelian Galois group of $K$ using the maximal $\ell^N$-abelian-by-central Galois group of $K$, whenever $N$ is sufficiently large relative to $n$.
\end{abstract}

\maketitle
\tableofcontents

\section{Introduction}
\label{sec:introduction}

The first key step in most strategies towards anabelian geometry is to develop a \emph{local theory}, by which one recovers inertia and/or decomposition groups of ``points'' using the given Galois theoretic information.
In the context of anabelian curves, one should eventually detect inertia/decomposition groups of closed points of the given curve within its \'etale fundamental group.
On the other hand, in the birational setting, this corresponds to detecting inertia/decomposition groups of arithmetically and/or geometrically meaningful places of the function field under discussion within its absolute Galois group. 
The first instance of such a local theory, which predates Grothendieck's anabelian geometry, is Neukirch's group-theoretical characterization of decomposition groups of finite places of global fields.
This was the basis for the celebrated Neukirch-Uchida theorem \cite{Neukirch1969}, \cite{Neukirch1969a}, \cite{Uchida1976}.
The Neukirch-Uchida theorem was expanded by Pop to all higher dimensional infinite finitely generated fields by developing a local theory based on his $q$-Lemma \cite{Pop1994}, \cite{Pop2000}.
The $q$-Lemma deals with the \emph{absolute} pro-$q$ Galois theory of fields of characteristic prime to $q$; as with Neukirch's result, the $q$-Lemma works only in arithmetical situations.

At about the same time, two non-arithmetically based methods were proposed which recover inertia and decomposition groups of valuations using Galois groups.
The first relies on the theory of \emph{rigid elements}, which was first introduced by Ware \cite{Ware1981} and further developed by several authors including \cite{Arason1987}, \cite{Efrat1999}, \cite{Koenigsmann1995}, \cite{Efrat2006}, \cite{Efrat2007} (see below for more details).
Rigid elements have since been extensively used to detect valuations in \emph{large} Galois groups.
For instance, using rigid elements one can recover inertia/decomposition using the full \emph{relative} pro-$\ell$ Galois theory of a field whose characteristic is prime to $\ell$ and which contains $\mu_\ell$ \cite{Engler1994}, \cite{Efrat1995}, \cite{Engler1998}.
Similar results also show how to recover inertia/decomposition in the absolute Galois group of an arbitrary field \cite{Koenigsmann2003}.
In both situations, however, the input is an extremely large Galois group: the maximal pro-$\ell$ Galois group resp. absolute Galois group.
Nevertheless, this method eventually led to the characterization of solvable absolute Galois groups of fields \cite{Koenigsmann2001}, and also the characterization of maximal pro-$\ell$ Galois groups of small rank \cite{Koenigsmann1998}, \cite{Efrat1998}.

The second method is Bogomolov's theory of \emph{commuting-liftable pairs} in Galois groups which was first introduced by Bogomolov in \cite{Bogomolov1991} then further developed together with Tschinkel in \cite{Bogomolov2007}.
Its input is the much smaller maximal pro-$\ell$ \emph{abelian-by-central} Galois group, but it requires that the base field contain an \emph{algebraically closed subfield}.
Nevertheless, this theory was a key technical tool in the local theory needed to settle Bogomolov's program in birational anabelian geometry for function fields over the algebraic closure of finite fields; see Bogomolov-Tschinkel \cite{Bogomolov2008a} in dimension $2$ and Pop \cite{Pop2011} in general. 

Until now, the two approaches -- that of rigid elements versus that of commuting-liftable pairs -- remained almost completely separate.
However, Pop suggested in his Oberwolfach report \cite{Pop2006a} that the two methods should be linked, even in the analogous $\ell^n$-abelian-by-central situation, but unfortunately never followed up with the details.
Also, the work done by Mah\'e, Min\'a\v{c} and Smith \cite{Mah'e2004} in the $2$-abelian-by-central situation, and Efrat-Min\'a\v{c} \cite{Efrat2011a} in special cases of the $\ell$-abelian-by-central situation suggest a connection between the two methods in this analogous context.

This paper provides an approach which unifies the two methods.
At the same time, we provide simpler arguments for the pro-$\ell$ abelian-by-central assertions of \cite{Bogomolov2007}, and prove more general versions of these assertions which assume only that the field contains $\mu_{\ell^\infty}$ and not necessarily an algebraically closed subfield.
We thereby generalize the main results of loc.cit. where the existence of an algebraically closed subfield is essential in the proof.
The following is a summary of the more detailed Theorems \ref{thm:main-thm-intro-maximal} and \ref{thm:main-thm-char-intro}.
\\  \\
\noindent{\bf Summary of Main Theorems. }{\it Let $n \geq 1$ or $n = \infty$ be given, then for all $N \gg n$ the following holds.
Let $K$ be a field such that $\Char K \neq \ell$ which contains $\mu_{2\ell^N}$.
Then there is a group-theoretical recipe which recovers (minimized) inertia and decomposition subgroups in the maximal $\ell^n$-elementary-abelian Galois group of $K$ using the group-theoretical structure of the $\ell^N$-abelian-by-central Galois group of $K$.
Moreover, if $n = 1$ then $N=1$ suffices and if $n \neq \infty$ then one can find (an explicit) $N \neq \infty$ as well.}
\\ \\
For readers' sake, we give a more detailed overview of some of the results mentioned above to see how the results of this paper fit into the larger context.

\subsection{Overview}
\label{sec:overview}

Let $K$ be a field such that $\Char K \neq \ell$ and $\mu_\ell \subset K$.
Denote by $K(\ell)$ the maximal pro-$\ell$ Galois extension of $K$ (inside a chosen separable closure of $K$) so that $\Gc_K := \Gal(K(\ell)|K)$ is the maximal pro-$\ell$ quotient of $G_K$, the absolute Galois group of $K$.
For a subset $\Sc$ of a profinite group, we will denote by $\langle \Sc \rangle$ the \emph{closed} subgroup generated by $\Sc$.

Let $w$ be a valuation of $K(\ell)$ and let $v = w|_K$ denote its restriction to $K$; denote by $k(w)$ the residue field of $w$ and $k(v)$ the residue field of $v$.
We denote the inertia resp. decomposition group of $w|v$ by $T_{w|v}$ resp. $Z_{w|v}$; these are subgroups of $\Gc_K$, and $T_{w|v}$ is a normal subgroup of $Z_{w|v}$.
Recall that $Z_{w|v} / T_{w|v} = \Gc_{k(v)}$ and that the following canonical short exact sequence splits: 
\[ 1 \rightarrow T_{w|v} \rightarrow Z_{w|v} \rightarrow \Gc_{k(v)} \rightarrow 1. \]
Moreover, if $\Char k(v) \neq \ell$, then $T_{w|v}$ is a free abelian pro-$\ell$ group of the same rank as $v(K^\times)/\ell$, and the action of $\Gc_{k(v)}$ on $T_{w|v}$ factors via the $\ell$-adic cyclotomic character.
Thus, if $\Char k(v) \neq \ell$, and $\sigma \in T_{w|v}$, $\tau \in Z_{w|v}$ are given non-torsion elements so that the closed subgroup $\langle \sigma,\tau \rangle$ is non-pro-cyclic, then ${\langle \sigma,\tau \rangle} = {\langle \sigma \rangle} \rtimes {\langle \tau \rangle} \cong \Z_\ell \rtimes \Z_\ell$ is a semi-direct product.

In a few words, the theory of rigid elements in the context of pro-$\ell$ Galois groups (\cite{Engler1994}, \cite{Efrat1995}, \cite{Engler1998}) asserts that the only way the situation above can arise is from valuation theory.
More precisely, let $K$ be a field such that $\Char K \neq \ell$ and $\mu_{\ell} \subset K$.
Suppose that $\sigma,\tau \in \Gc_K$ are non-torsion elements such that ${\langle \sigma,\tau \rangle} = {\langle \sigma \rangle} \rtimes {\langle \tau \rangle}$ is non-pro-cyclic.
Then there exist valuations $w|v$ of $K(\ell)|K$ such that $\Char k(v) \neq \ell$, $v(K^\times) \neq v(K^{\times \ell})$, $\sigma, \tau \in Z_{w|v}$ and $\langle \sigma,\tau \rangle /( \langle \sigma,\tau \rangle \cap T_{w|v})$ is cyclic.
The key technique in this situation is the explicit ``creation'' of valuation rings inside $K$ using rigid elements \cite{Ware1981}, \cite{Arason1987} and so-called ``$\ell$-rigid calculus'' developed, for instance, in \cite{Koenigsmann1995} and/or \cite{Efrat1999}.
Namely, under the assumption that $\Gc_K = {\langle \sigma,\tau \rangle} = {\langle \sigma \rangle} \rtimes {\langle \tau \rangle}$ as above, one shows that $K$ has sufficiently many ``strongly-rigid elements'' (cf. \cite{Koenigsmann1995}) to produce an $\ell$-Henselian valuation $v$ of $K$ with $v(K^\times) \neq v(K^{\times \ell})$ and $\Char k(v) \neq \ell$.

The history of rigid elements is quite rich.
They were first considered by Ware \cite{Ware1981} and others primarily in the context of quadratic forms.
The theory was further developed in the context of valuation theory and/or Galois theory by Arason-Elman-Jacob in \cite{Arason1987}, Engler-Nogueira in \cite{Engler1994}, Koenigsmann in \cite{Koenigsmann1995}, \cite{Koenigsmann2003}, Engler-Koenigsmann in \cite{Engler1998}, Efrat in \cite{Efrat1995}, \cite{Efrat1999}, \cite{Efrat2006}, \cite{Efrat2007} and also by others.

On the other hand, Bogomolov's theory of commuting-liftable pairs deals with a more specialized context.
Here we assume further that $\mu_{\ell^\infty} \subset K$, and choose an isomorphism of Galois modules $\Z_\ell(1) \cong \Z_\ell$.
In this case, we define 
\[ \Pi_K^{a} := \frac{\Gc_K}{[\Gc_K,\Gc_K]}, \ \text{ and } \ \Pi_K^c := \frac{\Gc_K}{[\Gc_K,[\Gc_K,\Gc_K]]}. \]
The group $\Pi_K^a$ is called the \emph{maximal pro-$\ell$ abelian} Galois group of $K$ while $\Pi_K^c$ is called the \emph{ maximal pro-$\ell$ abelian-by-central Galois group} of $K$ -- this terminology and notation was introduced by Pop \cite{Pop2010}.

In the above context, assume again that $\Char k(v) \neq \ell$, then the $\ell$-adic cyclotomic character of $K$ (and of $k(v)$) is trivial. 
Hence, $\Gc_{k(v)}$ acts trivially on $T_{w|v}$; we conclude that $Z_{w|v} \cong T_{w|v} \times \Gc_{k(v)}$ and recall that $T_{w|v}$ is abelian.
Denote by $K^{ab}$ the Galois extension of $K$ such that $\Gal(K^{ab}|K) = \Pi_K^a$, $v^{ab} := w|_{K^{ab}}$ the restriction of $w$ to $K^{ab}$, $T_v := T_{v^{ab}|v}$ and $Z_v := Z_{v^{ab}|v}$; since $\Pi_K^a$ is abelian, $T_v$ and $Z_v$ are independent of choice of $w$.
We deduce that for all $\sigma \in T_v$ and $\tau \in Z_v$, there exist lifts $\tilde\sigma,\tilde\tau \in \Pi_K^c$ of $\sigma,\tau \in \Pi_K^a$ which commute in $\Pi_K^c$.
Since $\Pi_K^c$ is a central extension of $\Pi_K^a$, we conclude that \emph{any lifts} $\tilde\sigma,\tilde\tau \in \Pi_K^c$ of $\sigma,\tau \in \Pi_K^a$ commute as well -- such a pair $\sigma,\tau \in \Pi_K^a$ is called \emph{commuting-liftable}.
The theory of commuting-liftable pairs \cite{Bogomolov2007} asserts that, under the added assumption that $K$ \emph{contains an algebraically closed subfield} $k = \bar k$, the only way a commuting-liftable pair can arise is via a valuation as described above.\footnote{It turns out that $\Char k(v) \neq \ell$ is not needed in order to produce a commuting-liftable pair, under a modified notion of decomposition and inertia. Valuations with residue characteristic equal to $\ell$ can and do arise from commuting-liftable pairs, as we prove in this paper.}

The method of \cite{Bogomolov2007} uses the notion of a ``flag function;'' in particular, this is a homomorphism $K^\times \rightarrow \Z_\ell$ which corresponds, via Kummer theory and our chosen isomorphism $\Z_\ell(1) \cong \Z_\ell$, to an element in $T_v$ for some valuation $v$.
One then considers $\sigma,\tau$ as elements of $\Hom(K^\times,\Z_\ell(1)) \cong \Hom(K^\times/k^\times,\Z_\ell)$, and produces the corresponding map:
\[ \Psi = (\sigma,\tau) : K^\times/k^\times \rightarrow \Z_\ell^2 \subset \Abb^2(\Q_\ell). \]
When one views $K^\times/k^\times = \Pbb_k(K)$ as an infinite dimensional projective space over $k$, the assumption that $\sigma,\tau$ are commuting liftable ensures that $\Psi$ sends \emph{projective lines} to \emph{affine lines}.
This severe restriction on $\Psi$ is then used to show that some $\Z_\ell$-linear combination of $\sigma$ and $\tau$ is a flag function.

As mentioned above, the theory of commuting-liftable pairs was originally outlined by Bogomolov in \cite{Bogomolov1991}, where he also introduced a program in birational anabelian geometry for fields of purely geometric nature -- i.e. function fields over an algebraically closed field of characteristic different from $\ell$ and dimension $\geq 2$ -- which aims to reconstruct such function fields $K$ from the Galois group $\Pi_K^c$.
If $\Char K > 0$, the above technical theorem eventually allows one to detect the decomposition and inertia subgroups of \emph{quasi-divisorial valuations} inside $\Pi_K^a$ using the group-theoretical structure of $\Pi_K^c$ (see Pop \cite{Pop2010}).
In particular, for function fields $K$ over the algebraic closure of a finite field, one can detect the decomposition/inertia structure of \emph{divisorial valuations} inside $\Pi_K^a$ using $\Pi_K^c$.
While Bogomolov's program in its full generality is far from being complete, it has been carried through for function fields $K$ over the algebraic closure of a finite field: by Bogomolov-Tschinkel \cite{Bogomolov2008a} in dimension 2, and by Pop \cite{Pop2011} in general.

In this paper, we obtain analogous results to those in the theory of commuting-liftable pairs, for the $\ell^n$-abelian-by-central and the pro-$\ell$-abelian-by-central situations, by elaborating on and using the theory of rigid elements, while working under far less restrictive assumptions than Bogomolov and Tschinkel's approach.
In particular, we reprove and generalize the main results of \cite{Bogomolov2007} using this method.
We begin by introducing some technical assumptions and notation.

\subsection{Notation}
\label{sec:notation}

For the remainder of the paper, $\ell$ will denote a fixed prime.
A ``subgroup'' in the context of profinite groups will always mean a closed subgroup, and all homomorphisms we consider will be continuous.
Also, the word ``cyclic'' in the context of profinite groups should be understood as ``procyclic.''
In a pro-$\ell$ group $\Gc$, we will use the following standard notation: for $\sigma \in \Gc$, one has $\sigma^{\ell^\infty} = 1$.
This is consistent with saying that $\ell^\infty = 0$ in $\Z_\ell$.
For an abelian group $A$, we will denote by $\widehat A$ the $\ell$-adic completion of $A$; namely:
\[ \widehat A := \varprojlim_n A/\ell^n. \]
Similarly, for a homomorphism of abelian groups $f : A \rightarrow B$, we denote by $\widehat f$ the induced homomorphism $\widehat f : \widehat A \rightarrow \widehat B$.
Also, the notation $\ell^\infty \cdot A$ will be used to denote the intersection $\bigcap_{n \in \Nb} \ell^n \cdot A$.
Namely, $\ell^\infty \cdot A$ is the kernel of the $\ell$-adic completion map $A \rightarrow \widehat A$.

Let $K$ be a field whose characteristic is different from $\ell$.
Let $n$ denote either a positive integer or $n=\infty$ and assume that $\mu_{2\ell^n} \subset K$.
In this case, we denote by $\Gc_K^{a,n}$ the maximal $\ell^n$-abelian Galois group and $\Gc_K^{c,n}$ the maximal $\ell^n$-abelian-by-central Galois groups of $K$.
More explicitly, define $\Gc_K^{(2,n)} := [\Gc_K,\Gc_K] \cdot (\Gc_K)^{\ell^n}$ and $\Gc_K^{(3,n)} = [\Gc_K,\Gc_K^{(2,n)}] \cdot (\Gc_K^{(2,n)})^{\ell^n}$; then one has 
\[ \Gc_K^{a,n} := \Gc_K/\Gc_K^{(2,n)}, \ \text{ and } \ \Gc_K^{c,n} := \Gc_K/\Gc_K^{(3,n)}. \]

Consider the canonical projection $\Pi : \Gc_K^{c,n} \twoheadrightarrow \Gc_K^{a,n}$; we will use additive notation for the abelian pro-$\ell$ groups $\Gc_K^{a,n}$ and $\ker\Pi$.
The map $\Pi$ induces certain operations on $\Gc_K^{a,n}$ with values in $\ker\Pi$, as follows.
First, $[\bullet,\bullet] : \Gc_K^{a,n} \times \Gc_K^{a,n} \rightarrow \ker \Pi$ defined by $[\sigma,\tau] = \tilde\sigma^{-1}\tilde\tau^{-1}\tilde\sigma\tilde\tau$ where $\tilde\sigma,\tilde\tau \in \Gc_K^{c,n}$ are some lifts of $\sigma,\tau \in \Gc_K^{a,n}$.
Since $\Pi$ is a central extension, the map $[\bullet,\bullet]$ is well-defined, and it is well known that $[\bullet,\bullet]$ is bilinear.
Second, $(\bullet)^\pi : \Gc_K^{a,n} \rightarrow \ker \Pi$ defined by $\sigma^\pi = \tilde\sigma^{\ell^n}$ where, again, $\tilde\sigma \in \Gc_K^{c,n}$ is some lift of $\sigma \in \Gc_K^{a,n}$.
Since $\Pi$ is a central extension with kernel killed by $\ell^n$, the map $(\bullet)^\pi$ is well defined.
It is well-known that this map is linear if $\ell \neq 2$, although it is generally not linear if $\ell = 2$.
We will furthermore define $\sigma^\beta := 2 \cdot \sigma^\pi$.
The map $(\bullet)^\beta : \Gc_K^{a,n} \rightarrow \ker \Pi$ is always linear.

In order to account for the possibility of a non-trivial $\ell$-adic cyclotomic character for $K$, we must work with a modified notion of ``commuting-liftable pairs.''
Henceforth, a pair of elements $\sigma,\tau \in \Gc_K^{a,n}$ will be called a {\bf commuting-liftable} pair (or a {\bf CL-pair} for short) provided that $[\sigma,\tau] \in \langle \sigma^\beta,\tau^\beta \rangle$.
Note, if $n = \infty$ (e.g. as in the context of \cite{Bogomolov2007}), then $\sigma,\tau$ form a CL-pair if and only if $[\sigma,\tau] = 0$, as expected.

For a (closed) subgroup $A \leq \Gc_K^{a,n}$, we define
\[ \Ibcl(A) := \{ \sigma \in A \ : \ \forall \tau \in A, \  [\sigma,\tau] \in \langle \sigma^\beta,\tau^\beta \rangle \}. \]
Then $\Ibcl(A)$ is a subgroup\footnote{This is not immediate if $n \neq \infty$, but follows from Theorem \ref{thm:cl-to-c-pairs}. See also Remark \ref{remark:ibcl-intro} and/or \ref{remark:ibcl} for the case $n=1$. See Remark \ref{rem:minimized-inertia-decomp} alongside the main results of the paper to see that this definition of $\Ibcl$ is indeed sufficient for the context of valuation theory.} of $A$; it should be thought of as the ``commuting-liftable center'' of $A$.
The subgroup $A$ will be called a {\bf commuting-liftable} group (or a {\bf CL-group} for short) provided that $A = \Ibcl(A)$.

\begin{remark}
\label{remark:ibcl-intro}
Let $K$ be a field such that $\Char K \neq \ell$ and $\mu_{2\ell} \subset K$, and let $A \leq \Gc_K^{a,1}$ be given.
In this case, we can give an alternative definition for $\Ibcl(A)$ which is easily seen to be a subgroup of $A$.
Namely, for $A \leq \Gc_K^{a,1}$ one has 
\[ \Ibcl(A) = \{ \sigma \in A \ : \ \forall\tau \in A , \ [\sigma,\tau] \in A^\beta\}. \]
This alternative definition also shows that our context extends the situation of \cite{Efrat2011a}.
See Remark \ref{remark:ibcl} for the proof of this equivalence.
\end{remark}

Suppose $v$ is a valuation of $K$.
We will denote by $\Gamma_v = v(K^\times)$ the value group, $\Oc_v$ the valuation ring with valuation ideal $\mf_v$, and $k(v) = \Oc_v/\mf_v$ the residue field of $v$.

We denote by $K^{a,n} = K(\sqrt[\ell^n]{K})$ the Galois extension of $K$ such that $\Gal(K^{a,n}|K) = \Gc_K^{a,n}$.
For a subgroup $A \leq \Gc_K^{a,n}$, we denote by $K_A$ the field $(K^{a,n})^A$.

Choose a prolongation $v'$ of $v$ to $K^{a,n}$.
We denote by $T_v^n := T_{v'|v}$ resp. $Z_{v}^n = Z_{v'|v}$ the inertia resp. decomposition subgroups of $\Gc_K^{a,n}$ associated to $v'|v$; since $\Gc_K^{a,n}$ is abelian, these groups are independent of choice of $v'$.

We now introduce the {\bf minimized inertia/decomposition} subgroups associated to $v$: 
\[ I_v^n := \Gal(K^{a,n}|K(\sqrt[\ell^n]{\Oc_v^\times})) \ \text{ and } \  D_v^n := \Gal(K^{a,n}|K(\sqrt[\ell^n]{1+\mf_v})). \]
Observe that $I_v^n \leq D_v^n$.
More importantly, however, $I_v^n \leq T_v^n$ and $D_v^n \leq Z_v^n$ with equality whenever $\Char k(v) \neq \ell$ (see Proposition \ref{prop:decomp-thy-prop}).
It turns out that the minimized inertia and decomposition groups, $I_v^n \leq D_v^n$, have an abelian-by-central Galois theoretical structure which resembles that of the usual inertia and decomposition, even for valuations whose residue characteristic is $\ell$ -- see Remark \ref{rem:minimized-inertia-decomp} for the details.
In particular, for any valuation $v$ of $K$, one has $I_v^n \leq \Ibcl(D_v^n)$.
This is analogous to the fact that $T_v^n \leq \Ibcl(Z_v^n)$ in case $\Char k(v) \neq \ell$.

Consider the following three conditions concerning valuations of $K$:
\begin{enumerate}
\item[(V1)] The value group $\Gamma_v$ contains no non-trivial $\ell$-divisible convex subgroups.
\item[(V2)] The valuation $v$ is maximal among all valuations $w$ such that (a) $D_v^n = D_w^n$ and (b) $\Gamma_w$ contains no non-trivial $\ell$-divisible convex subgroups.
Namely, for all refinements $w$ of $v$ such that $D_w^n = D_v^n$, one has $I_w^n = I_v^n$.
\item[(V3)] The group $k(v)^\times/\ell^n$ (resp. $\widehat{k(v)^\times}$ if $n = \infty$) is non-cyclic. 
\end{enumerate}
We denote by $\Vc_{K,n}$ the collection of valuations $v$ of $K$ which satisfy (V1),(V2) and (V3).
We also denote by $\Wc_{K,n}$ the collection of valuations $v$ of $K$ which only satisfy (V1) and (V2).
It turns out that many valuations of interest are contained in $\Wc_{K,n}$.
For instance, if $K$ is a function field over an algebraically closed field $k$, then all Parshin chains of divisors are contained in $\Wc_{K,n}$ and those Parshin chains of non-maximal length are contained in $\Vc_{K,n}$.
This is also true in more general situations; see Example \ref{ex:prime-divs} for more details.

We also introduce the ``prime-to-$\ell$'' analogue of $\Vc_{K,n}$, which we denote by $\Vc_{K,n}'$.
Denote by $\Vc_{K,n}'$ the collection of valuations $v$ of $K$ which satisfy the following conditions:
\begin{enumerate}
\item[(V0')] One has $\Char k(v) \neq \ell$.
\item[(V1')] The value group $\Gamma_v$ contains no non-trivial $\ell$-divisible convex subgroups.
\item[(V2')] The valuation $v$ is maximal among all valuations $w$ such that (a) $\Char k(w) \neq \ell$, (b) $D_v^n = D_w^n$ and (c) $\Gamma_w$ contains no non-trivial $\ell$-divisible convex subgroups.
Namely, for all refinements $w$ of $v$ such that $\Char k(w) \neq \ell$ and $D_w^n = D_v^n$ as subgroups of $\Gc_K^{a,n}$, one has $I_w^n = I_v^n$.
\item[(V3')] The group $\Gc_{k(v)}^{a,n}$ is non-cyclic.
\end{enumerate}

The relationship between $\Vc_{K,n}$ and $\Vc_{K,n}'$ is as follows.
If $v \in \Vc_{K,n}$ has residue characteristic different from $\ell$, then $v$ lies in $\Vc_{K,n}'$.
Also, if $\Char K > 0$, we have $\Vc_{K,n} = \Vc_{K,n}'$.
In general, however, the two sets are quite different.

\begin{remark}
\label{remark:alt-defn-of-Vc}
Using some technical results of this paper, we can give an alternative concise definition for $\Vc_{K,n}$ in the case where $\Char K \neq \ell$ and $\mu_{2\ell^n} \subset K$. 
Namely, $\Vc_{K,n}$ is precisely the collection of valuations $v$ of $K$ such that:
\begin{enumerate}
\item $\Gamma_v$ contains no non-trivial $\ell$-divisible convex subgroups.
\item $I_v^1 = \Ibcl(D_v^1) \neq D_v^1$.
\end{enumerate}
From this it is clear that $\Vc_{K,m} = \Vc_{K,n}$ for all $m \leq n$.
See Lemma \ref{lem:alt-defn-of-Vc} and Theorem \ref{thm:cl-to-c-pairs} for the proof of this equivalence.
\end{remark}

Denote by $\Nb$ the collection of positive integers and $\Nbar = \Nb \cup\{ \infty\}$; we declare that $\infty > n$ for all $n \in \Nb$.
If $N \geq n$ and $\mu_{\ell^N} \subset K$ (and thus $\mu_{\ell^n} \subset K$ as well), we will denote the canonical map $\Gc_K^{a,N} \rightarrow \Gc_K^{a,n}$ by $f \mapsto f_n$.
Furthermore, for an extension $L|K$ of fields, we will denote by $f \mapsto f_K$ the canonical restriction map $\Gc_L^{a,n} \rightarrow \Gc_K^{a,n}$.

\subsection{Main Results of the Paper}
\label{sec:main-results-manuscr}

The main goal of this paper is to produce an explicit function $\mathbf{R} : \Nbar \rightarrow \Nbar$, satisfying the following conditions:
\begin{itemize}
\item If $n \in \Nb$ then $\mathbf{R}(n) \in \Nb$.
\item One has $\mathbf{R}(1) = 1$ and $\mathbf{R}(\infty) = \infty$.
\item One has $\mathbf{R}(n) \geq n$ for all $n \in \Nbar$.
\end{itemize}
such that Theorems \ref{thm:main-thm-intro-maximal} and \ref{thm:main-thm-char-intro} hold true.
We succeed in constructing such a function: in the notation introduced in Part \ref{part:underlying-theory}, $\mathbf{R}(n) := \Nfr(\Mfr_2(\Mfr_1(n)))$ suffices.
However, we do not expect that the function which we construct is optimal.
Nevertheless, our conditions on $\Rbf$ ensure that Theorems \ref{thm:main-thm-intro-maximal} and \ref{thm:main-thm-char-intro} yield a non-trivial generalization of \cite{Bogomolov2007}.
Along with Remark \ref{remark:ibcl-intro}, our theorems can also be seen as generalizations of \cite{Engler1994}, \cite{Efrat1995}, \cite{Engler1998} and \cite{Efrat2011a} for fields which contain $\mu_{2\ell}$.

\begin{thm}
\label{thm:main-thm-intro-maximal}
Let $n \in \Nbar$ be given and let $N \geq \mathbf{R}(n)$.
Let $K$ be a field such that $\Char K \neq \ell$ and $\mu_{2\ell^N} \subset K$.
Then the following hold:
\begin{enumerate}
 \item Let $D \leq \Gc_K^{a,n}$ be given. There exists a valuation $v$ of $K$ such that $D \leq D_v^n$ and $D/(D \cap I_v^n)$ is cyclic if and only if there exists a CL-group $D' \leq \Gc_K^{a,N}$ such that $D'_n = D$.
 \item Let $I \leq D \leq \Gc_K^{a,n}$ be given. There exists a valuation $v \in \Vc_{K,n}$ such that $I = I_v^n$ and $D = D_v^n$ if and only if the following hold:
   \begin{enumerate}
   \item There exist $D' \leq \Gc_K^{a,N}$ such that $(\Ibcl(D'))_n = I$ and $D'_n = D$.
   \item The subgroups $I \leq D \leq \Gc_K^{a,n}$ are maximal with property (a).
   Namely, if $D \leq E \leq \Gc_K^{a,n}$ and $E' \leq \Gc_K^{a,N}$ is given such that $E'_n = E$ and $I \leq (\Ibcl(E'))_n$, then $D = E$ and $I = (\Ibcl(E'))_n$.
   \item One has $\Ibcl(D) \neq D$.
   Namely, $D$ is not a CL-group.
   \end{enumerate}
 \end{enumerate}
\end{thm}

The theorem above provides a group-theoretical recipe to recover valuations using abelian-by-central Galois groups.
More precisely, let $N = \Rbf(n)$, and $K$ a field with $\Char K \neq \ell$ and $\mu_{2\ell^N} \subset K$.
Thus, Theorem \ref{thm:main-thm-intro-maximal}(2) provides a group theoretical recipe to detect the subgroups $I_v^n$ and $D_v^n$ for $v \in \Vc_{K,n}$, using only the group-theoretical structure of $\Gc_K^{c,N}$.

Furthermore, it turns out that the ordered structure of $\Vc_{K,n}$ is also encoded group theoretically, as follows.
Let $v,w$ be two valuations of $K$ such that $\Gamma_v$ and $\Gamma_w$ contain no non-trivial $\ell$-divisible convex subgroups (e.g. $v,w \in \Vc_{K,n}$).
It follows from the results of this paper (Lemma \ref{lem:v_H}, in particular), that $v \leq w$ (i.e. $v$ is coarser than $w$) if and only if $I_v^n \leq I_w^n$.
In particular, each element $v$ of $\Vc_{K,n}$ is uniquely determined by the subgroup $I_v^n$ of $\Gc_K^{a,n}$.
Theorem \ref{thm:main-thm-intro-maximal}(2), in particular, gives a group-theoretical recipe to recover the subgroups $I_v^n$ of $\Gc_K^{a,n}$ for $v \in \Vc_{K,n}$.
Therefore, this theorem shows that the structure of $\Vc_{K,n}$, as a partially ordered set, is encoded group-theoretically using $\Gc_K^{c,N}$.
See Remark \ref{remark:main-detect-thm-mu} for a more detailed discussion.

By enlarging the group $\Gc_K^{c,N}$ we can also detect which of those valuations $v$ in the theorem above have residue characteristic different from $\ell$.
This therefore gives a group-theoretical recipe to detect the usual decomposition and inertia subgroups of valuations $v \in \Vc_{K,n}$ whose residue characteristic is different from $\ell$.
This is essentially the content of our next main theorem.
Before we state the theorem, we recall the notation $K_A := (K^{a,n})^A$ for a subgroup $A \leq \Gc_K^{a,n}$.

\begin{thm}
\label{thm:main-thm-char-intro}
Let $n \in \Nbar$ be given and let $N \geq \mathbf{R}(n)$.
Let $K$ be a field such that $\Char K \neq \ell$ and $\mu_{2\ell^N} \subset K$.
Then the following hold:
\begin{enumerate}
 \item Let $D \leq \Gc_K^{a,n}$ be given and let $L := K_D$. 
There exists a valuation $v$ of $K$ such that $\Char k(v) \neq \ell$, $D \leq Z_v^n$ and $D/(D \cap T_v^n)$ is cyclic if and only if there exists a CL-group $D' \leq \Gc_L^{a,N}$ such that $(D'_n)_K = D$.
\item Assume that $\Ibcl(\Gc_K^{a,n}) \neq \Gc_K^{a,n}$ and consider $(\Ibcl(\Gc_K^{a,N}))_n =: T$. 
Then there exists a valuation $v \in \Vc_{K,n}$ such that $\Char k(v) \neq \ell$, $T = T_v^n$ and $\Gc_K^{a,n} = Z_v^n$.
\item Let $v \in \Vc_{K,n}$ be given. 
Consider $I := I_v^n \leq D_v^n =: D$ and let $L := K_D$.
One has $\Char k(v) \neq \ell$ if and only if there exist $I' \leq D' \leq \Gc_L^{a,N}$ such that:
\begin{enumerate}
\item One has $I' \leq \Ibcl(D')$.
\item One has $(I'_n)_K = I$ and $(D'_n)_K = D$.
\end{enumerate}
Moreover, if these equivalent conditions hold then $I = I_v^n = T_v^n$ and $D = D_v^n = Z_v^n$.
\item Let $I \leq D \leq \Gc_K^{a,n}$ be given and let $L := K_D$.
 There exists a valuation $v \in \Vc_{K,n}'$ such that $I = T_v^n$ and $D = Z_v^n$ if and only if the following hold:
   \begin{enumerate}
   \item There exist $D' \leq \Gc_L^{a,N}$ such that $((\Ibcl(D'))_n)_K = I$ and $(D'_n)_K = D$.
   \item The subgroups $I \leq D \leq \Gc_K^{a,n}$ are maximal with property (a).
   Namely, if $D \leq E \leq \Gc_K^{a,n}$ and $E' \leq \Gc_{K_E}^{a,N}$ is given such that $(E'_n)_K = E$ and $I \leq ((\Ibcl(E'))_n)_K$, then $D = E$ and $I = ((\Ibcl(E'))_n)_K$.
   \item One has $\Ibcl(D) \neq D$.
   Namely, $D$ is not a CL-group.
   \end{enumerate}
 \end{enumerate}
\end{thm}

Let $N = \mathbf{R}(n)$, and $K$ a field with $\Char K \neq \ell$ and $\mu_{2\ell^N} \subset K$.
Denote by $\Gc_K^{M,n}$ the smallest quotient of $\Gc_K$ for which $\Gc_L^{c,N}$ is a subquotient for all $K \subset L \subset K^{a,n}$.
We note that Theorem \ref{thm:main-thm-intro-maximal}(2) along with Theorem \ref{thm:main-thm-char-intro}(3) provide a group-theoretical recipe to detect $T_v^n \leq Z_v^n$ for valuations $v \in \Vc_{K,n}$ such that $\Char k(v) \neq \ell$, using only the group-theoretical structure of $\Gc_K^{M,n}$.

Furthermore, Theorem \ref{thm:main-thm-char-intro}(4) provides a group-theoretical recipe to detect $T_v^n \leq Z_v^n$ for valuations $v \in \Vc_{K,n}'$ using only the group-theoretical structure of $\Gc_K^{M,n}$.
Arguing similarly to the discussion following Theorem \ref{thm:main-thm-intro-maximal}, and using the fact that $T_v^n = I_v^n$ for valuations $v$ with $\Char k(v) \neq \ell$, this shows that the partially ordered structure of the set $\Vc_{K,n}'$ is encoded group-theoretically in $\Gc_K^{M,n}$.

\subsection{A Guide Through the Paper and Corollaries}
\label{sec:guide-thro-manuscr}

In Part \ref{part:underlying-theory}, we develop the underlying theory which proves the main results of the paper.
This theory works for an arbitrary field $K$, and is based on an abstract notion of ``C-pairs'' which is related to a condition in the Milnor K-theory of the field (see Proposition \ref{prop:k-thy-c-pair}).

In Part \ref{part:galois-theory-milnor} we prove our K-theoretic condition which determines C-pairs.
Although this K-theoretic condition is primarily needed for the Galois-theoretical characterization of C-pairs, this also puts the results of Part \ref{part:underlying-theory} in a similar context as the results of Efrat \cite{Efrat1999}, \cite{Efrat2006b}, \cite{Efrat2007}.
In particular, this shows how to detect valuations using the Milnor K-theory groups of a field without the presence of any Galois theory.

The main theorem of Part \ref{part:galois-theory-milnor}, which is Theorem \ref{thm:cl-to-c-pairs}, shows that the two notions -- that of C-pairs and that of CL-pairs -- are identical in the situation where $\Char K \neq \ell$ and $\mu_{2\ell^n} \subset K$.
The proof of this theorem relies on the Merkurjev-Suslin theorem \cite{Merkurjev1982}.

In part \ref{part:proofs}, we prove Theorems \ref{thm:main-thm-intro-maximal} and \ref{thm:main-thm-char-intro}.
We also prove the following main corollary, which provides a sufficient condition to detect whether or not $\Char K = 0$ using the Galois group $\Gc_K^{M,n}$.

\begin{cor*}[Corollary \ref{cor:application-char}]
\label{cor:intro-main}
Let $n \in \Nbar$ be given and let $N := \mathbf{R}(n)$.
Let $K$ be a field such that $\Char K=0$ and $\mu_{2\ell^N} \subset K$.
Assume that there exists a field $F$ such that $\Char F > 0$, $\mu_{2\ell^N} \subset F$ and $\Gc_K^{M,n} \cong \Gc_F^{M,n}$.
Then for all $v \in \Vc_{K,n}$ one has $\Char k(v) \neq \ell$.
\end{cor*}

As a consequence of the Corollary above, we find many examples of fields $K$ of characteristic $0$ whose maximal pro-$\ell$ Galois group $\Gc_K$ is not isomorphic to $\Gc_F$ for any field $F$ of positive characteristic.
Strongly $\ell$-closed fields, which are mentioned in the following corollary, are defined in \S \ref{sec:set-vc_k-n}, but we mention here that algebraically closed fields in particular are strongly $\ell$-closed.

\begin{cor*}[Corollary \ref{cor:function-fields-galois-groups}]
\label{cor:intro-ex}
Suppose that $K$ is one of the following:
\begin{itemize}
\item a function field over a number field $k$ such that $\mu_{2\ell} \subset k$, and $\dim(K|k) \geq 1$, or
\item a function field over a strongly $\ell$-closed field $k$ of characteristic $0$ such that $\dim(K|k) \geq 2$.
\end{itemize}
Then there does not exist a field $F$ such that $\mu_{2\ell} \subset F$, $\Char F > 0$ and $\Gc_K \cong \Gc_F$.
\end{cor*}

\subsection*{Acknowledgments}
\label{sec:acknowledgments}
The author would like to thank all who expressed interest in this work and in particular Florian Pop, Jakob Stix, Jochen Koenigsmann, Moshe Jarden, Dan Haran, Lior Bary-Soroker, J\'{a}n Min\'{a}\v{c} and Ido Efrat.
The author also thanks the referee for very thoroughly reading the paper, and for his excellent comments which were very helpful in improving the paper.

\part{C-groups and Valuations}
\label{part:underlying-theory}

In this first part of this paper, we develop the underlying theory using an abstract notion of ``C-pairs.''
It turns out, as we will see in Part \ref{part:galois-theory-milnor}, that this notion is equivalent to that of CL-pairs as defined in the introduction.
Throughout, we will tacitly use the following trivial observation and dub it ``the Cancellation Principle.''

\begin{lem}[The Cancellation Principle]
\label{lem:cancellation-principle}
For positive integers $n$ and $r$, define $\Mfr_r(n) := (r+1)\cdot n-r$.
Assume that $R \geq (r+1)\cdot n-r = \Mfr_r(n)$.
Let $a,b,c_1,\ldots,c_r \in \Z/\ell^R$ be given such that $c_i \neq 0 \mod \ell^n$ for $i = 1,\ldots,r$.
Assume that the following equality holds in $\Z/\ell^R$:
\[ a \cdot (c_1\cdots c_r) = b\cdot (c_1\cdots c_r). \]
Then $a = b \mod \ell^n$.
\end{lem}
\begin{proof}
Let $s$ be the minimal positive integer such that $\ell^s \cdot (c_1 \cdots c_r) = 0$ as an element of $\Z/\ell^R$.
Then the map $\Z/\ell^s \rightarrow \Z/\ell^R$ defined by $x \mapsto x \cdot (c_1 \cdots c_r)$ is injective.
On the other hand, as $c_i \neq 0 \mod \ell^n$ for each $i$, we observe that $s \geq R - rn +r \geq n$ and this proves the claim.
\end{proof}

\section{Main Theorem of C-Pairs}
\label{sec:main-theorem-c}

Recall that $\Nbar = \{1,2,\ldots,\infty\}$.
For positive integers $n$ and $r$, we define the following three integers:
\begin{enumerate}
\item $\Mfr_r(n) := (r+1)\cdot n-r$, as used in Lemma \ref{lem:cancellation-principle}.
\item $\Nfr'(n) := (6\ell^{3n-2}-7)\cdot (n-1)+3n-2$.
\item $\Nfr(n) := \Mfr_1(\Nfr'(n))$.
\end{enumerate}
The precise formula for $\Nfr'$ and $\Nfr$ will not play an important role until we prove Theorem \ref{thm:main-c-pairs-thm} in \S\ref{sec:proof-theorem}.

To make the notation consistent, we define $\Mfr_r(\infty) = \Nfr(\infty) := \infty$.
In particular, one has $\Nfr(n) \geq \Mfr_1(n) \geq n$ for all $n \in \Nbar$, and $\Nfr(n),\Mfr_r(n) \in \Nb$ if and only if $n \in \Nb$.
Also, observe that $\Mfr_r(1) = \Nfr'(1) = \Nfr(1) = 1$, and $\Mfr_r(\infty) = \Nfr'(\infty) = \Nfr(\infty) = \infty$. 
We will work with a coefficient ring which depends on $n \in \Nbar$, which we define as follows:
\[R_n := \begin{cases} \Z/\ell^n, & n \neq \infty. \\ \Z_\ell, & n = \infty.
\end{cases} \]

Let $K$ be a field and $n \in \Nbar$ be given.
We define:
\[ \Gc_K^a(n) := \Hom(K^\times/\pm1,R_n). \]
Endowed with the point-wise convergence topology, we consider $\Gc_K^a(n)$ as a pro-$\ell$ Group.
We will always consider elements $f \in \Gc_K^a(n)$ as homomorphisms $f : K^\times \rightarrow R_n$ with $f(-1) = 0$.
Note furthermore, if $-1 \in K^{\times \ell^n}$ (e.g. if $\mu_{2\ell^n} \subset K$), then $\Gc_K^a(n) = \Hom(K^\times,R_n)$.

For $n,N \in \Nbar$ with $N \geq n$, we will denote the canonical projection $R_N \rightarrow R_n$ by $a \mapsto a_n$; we extend this notation to elements of $(R_N)^k$ in the obvious way: $(a_1,\ldots,a_k)_n := ((a_1)_n,\ldots,(a_k)_n)$.
Similarly, we denote by $f \mapsto f_n$ the canonical map $\Gc_K^a(N) \rightarrow \Gc_K^a(n)$ which is induced by the canonical projection $R_N \twoheadrightarrow R_n$; namely, for $f \in \Gc_K^a(N)$, one has $f_n(x) := f(x)_n$.

If $v$ is a valuation of $K$ we define:
\begin{enumerate}
\item $I_v(n) := \Hom(K^\times/\Oc_v^\times,R_n) \leq \Gc_K^a(n)$ and 
\item $D_v(n) := \Hom(K^\times/\pm(1+\mf_v),R_n) \leq \Gc_K^a(n)$.
\end{enumerate}
For a subgroup $A \leq \Gc_K^a(n)$, we denote by $A^\perp$ the subgroup of $K^\times$ annihilated by $A$:
\[ A^\perp = \bigcap_{f \in A} \ker f. \]
This is precisely the left kernel of the canonical pairing $K^\times \times A \rightarrow R_n$.
In particular, we note that $\Oc_v^\times \leq I_v(n)^\perp$ and $\pm(1+\mf_v) \leq D_v(n)^\perp$.

The following notion of C-pairs is motivated by Bogomolov and Tschinkel's notion under the same name \cite{Bogomolov2007}; we note, however, that our notion of C-pairs is \emph{a priori} much weaker than the one considered in loc.cit.

\begin{defn}
\label{defn:c-pair}
Let $f,g \in \Gc_K^a(n)$ be given.
We say that $f,g$ are a {\bf C-pair} provided that for all $x \in K \smallsetminus \{0,1\}$ one has:
\[ f(1-x)g(x) = f(x)g(1-x). \]
A subgroup $A \leq \Gc_K^a(n)$ will be called a {\bf C-group} provided that any pair of elements $f,g \in A$ form a C-pair.
If $A = \langle f_i \rangle_i$, we observe that $A$ is a C-group if and only if $f_i,f_j$ form a C-pair for all $i,j$.

For a subgroup $A \leq \Gc_K^a(n)$, we define $\Ibc(A)$ to be the subgroup:
\[ \Ibc(A) := \{f \in A \ : \ \forall g \in A, \ f,g \  \text{ form a C-pair} \}. \]
and call $\Ibc(A)$ the C-center of $A$. 
In particular, $A$ is a C-group if and only if $A = \Ibc(A)$.
Also, it is easy to see that $A$ is a C-group if and only if $A/\Ibc(A)$ is cyclic.
\end{defn}

The following lemma shows that valuations yield C-pairs.
The main technical theorem in the paper, The Main Theorem of C-pairs, is a weak converse to this lemma.

\begin{lem}
\label{lem:I-and-D-give-c-pairs}
Let $n \in \Nbar$ be given and let $(K,v)$ be a valued field.
Let $d \in D_v(n)$ and $i \in I_v(n)$ be given.
Denote the map $(i,d) : K^\times \rightarrow R_n \times R_n$ by $\Psi$.
Then for all $x \in K^\times \smallsetminus \{0,1\}$, the subgroup $\langle \Psi(1-x),\Psi(x) \rangle$ is cyclic.
In particular, $i,d$ form a C-pair.
\end{lem}
\begin{proof}
If $v(x) > 0$ then $\Psi(1-x) = 0$ since $1+\mf_v \leq \ker\Psi$ so we obtain the claim.
If $v(x) < 0$ then $1-x = x(1/x-1)$ with $v(1/x) > 0$ so that $\Psi(1-x) = \Psi(x)$, and this proves the claim.
By replacing $x$ with $1-x$ if needed, the last case to consider is where both $x,1-x \in \Oc_v^\times$.
But then $i(1-x) = i(x) = 0$, so the claim is trivial.
\end{proof}

\begin{thm}[The Main Theorem of C-pairs]
\label{thm:main-c-pairs-thm}
Let $n \in \Nbar$ be given and let $N := \Nfr(n)$.
Let $K$ be an arbitrary field and let $f,g \in \Gc_K^a(n)$ be given.
Assume that there exist $f'',g'' \in \Gc_K^a(N)$ such that 
\begin{itemize}
\item $f'',g''$ form a C-pair.
\item $f''_n = f$ and $g''_n = g$.
\end{itemize}
Then there exists a valuation $v$ of $K$ such that
\begin{itemize}
\item $f,g \in D_v(n)$
\item $\langle f,g \rangle/(\langle f,g \rangle \cap I_v(n))$ is cyclic (possibly trivial).
\end{itemize}
\end{thm}
\begin{proof}
The proof of this theorem is highly technical since it involves a lot of calculation and intermediate steps.
Thus, for the sake of exposition, we defer the proof to \S\ref{sec:proof-theorem} which is completely devoted to this task.
\end{proof}

\section{Valuative Subgroups and Comparable Valuations}
\label{sec:comp-valu}

In this section we prove the main theorems which allow us to detect valuations using C-pairs in a more precise way.
To begin, we introduce the notion of a ``valuative'' subgroup $I \leq \Gc_K^a(n)$ which generalizes the notion of a ``flag function'' from \cite{Bogomolov2007}.
To a valuative subgroup $I \leq \Gc_K^a(n)$ we associate a canonical valuation $v_I$ which is reminiscent of Pop's notion of a core of a valuation in a Galois extension, and was also considered in \cite{Arason1987}.
It turns out that the C-pair property is intimately related to the comparability of these canonical valuations $v_I$.
Namely, we will show that, in certain cases, these ``valuative'' subgroups can be ``glued'' together.

\subsection{Rigid Elements and Valuations}

To define a ``valuative'' subgroup, we will require a result from the theory of rigid elements. 
Rigid elements will also play an important role when we prove Theorem \ref{thm:main-c-pairs-thm} in \S\ref{sec:proof-theorem}.
While one can use many references in the subject to deduce these results (see e.g. the overview in the introduction), we will take \cite{Arason1987} to be our reference of choice.
The following theorem is a summary of the main results of loc.cit. which we will need in this paper.

\begin{thm}[Arason-Elman-Jacob \cite{Arason1987}]
\label{thm:rigid-summary}
Let $K$ be an arbitrary field and let $T \leq K^\times$ be a subgroup with $-1 \in T$.
Denote by $H$ the subgroup of $K^\times$ which is generated by $T$ and all $x \in K^\times \smallsetminus T$ such that $1+x \notin T \cup x \cdot T$.
Then the following hold:
\begin{enumerate}
\item Suppose that $H = T$. 
Then there exists a valuation $v$ of $K$ such that $1+\mf_v \leq T$ and $\#((\Oc_v^\times \cdot T )/ T) \leq 2$; in particular $(\Oc_v^\times \cdot T) / T$ is cyclic.
\item Suppose that $H = T$ and that for all $x,y \in K^\times \smallsetminus H$ such that $1+x,1+y \in H$, one has $1+x\cdot (1+y) \in H$. 
Then there exists a valuation $v$ of $K$ such that $\Oc_v^\times \leq H$.
\item Suppose that $H \neq T$. 
Then there exists a valuation $v$ of $K$ such that $1+\mf_v \leq T$ and $H = \Oc_v^\times \cdot T$.
\end{enumerate}
\end{thm}
\begin{proof}
We first make a simple observation which is needed to deduce claim (3).
If $v$ is a valuation of $K$ such that $1+\mf_v \leq T$, then for all $x \in K^\times \smallsetminus ( \Oc_v^\times \cdot T)$, one has $1+x \in (\Oc_v^\times \cdot T) \cup x \cdot (\Oc_v^\times \cdot T)$.
If such a valuation $v$ exists, the definition of $H$ ensures that $\Oc_v^\times \leq H$ if and only if $H = \Oc_v^\times \cdot T$.

Claims (1) and (3) follow from Theorem 2.16 of \cite{Arason1987}.
Claim (2) follows from Theorem 2.10 of loc.cit.; namely, the assumption of claim (2) is a reformulation of the ``preadditive'' condition from loc.cit.
\end{proof}

\subsection{Valuative Subgroups}
The non-trivial direction of the following lemma is a reformulation of Theorem \ref{thm:rigid-summary}(2).
We state this lemma explicitly so that we can conveniently cite it later.

\begin{lem}
\label{lem:valuative}
Let $K$ be a field and let $H \leq K^\times$ be given.
Then the following are equivalent:
\begin{enumerate}
\item There exists a valuation $v$ of $K$ such that $\Oc_v^\times \leq H$.
\item One has (a) $-1 \in H$, (b) for all $x \in K^\times \smallsetminus H$ one has $1+x \in H \cup x \cdot H$, and (c) for all $x,y \in K^\times \smallsetminus H$ such that $1+x,1+y \in H$, one has $1+x\cdot (1+y) \in H$.
\end{enumerate}
\end{lem}
\begin{proof}
The non-trivial direction, $(2) \Rightarrow (1)$, is Theorem \ref{thm:rigid-summary} claim (2) taking $T = H$.

We now prove $(1) \Rightarrow (2)$.
Assume (1), that there exists a valuation $v$ such that $\Oc_v^\times \leq H$; this clearly implies that $-1 \in H$.
Let $x \in K^\times \smallsetminus H$ be given.
Then, in particular, $v(x) \neq 0$.
One has $v(x) > 0$ if and only if $1+x \in \Oc_v^\times$.
Similarly, one has $v(x) < 0$ if and only if $1+x \in x \cdot \Oc_v^\times$.
Thus $1+x \in H \cup x \cdot H$ for all $x \in K^\times \smallsetminus H$.
Moreover, if $x,y \in K^\times \smallsetminus H$ with $1+x,1+y \in H$, then one has $v(x),v(y) > 0$.
Thus $v(x\cdot(1+y)) > 0$ and therefore $1+x\cdot (1+y) \in H$, as required.
\end{proof}

\begin{remark}
\label{rem:ell-odd-valuative-remark}
In the case where $K^{\times \ell^n} \leq H$ and $\ell$ is \emph{odd}, the condition of Lemma \ref{lem:valuative} can be made simpler.
Using the notation of Lemma \ref{lem:valuative}, the following are equivalent in this case:
\begin{enumerate}
\item There exists a valuation $v$ of $K$ such that $\Oc_v^\times \leq H$.
\item For all $x \in K^\times \smallsetminus H$ one has $1+x \in H \cup x\cdot H$.
\end{enumerate}
Again, the non-trivial direction of this claim follows from \cite{Arason1987}; see our Theorem \ref{thm:rigid-summary}(1).
\end{remark}

\begin{defn}
\label{defn:valuative}
A subgroup $H \leq K^\times$ will be called {\bf valuative} if it satisfies the equivalent conditions of Lemma \ref{lem:valuative}.
Similarly, a subgroup $I$ of $\Gc_K^a(n)$ will be called {\bf valuative} provided that $I^\perp$ is a valuative subgroup of $K^\times$.
Namely, $I$ is valuative if and only if there exists a valuation $v$ of $K$ such that $I \leq I_v(n)$.
We also say that $f \in \Gc_K^a(n)$ is valuative provided that $\ker(f)$ is valuative.
Namely, $f$ is valuative if and only if there exists a valuation $v$ of $K$ such that $f \in I_v(n)$.
\end{defn}

\begin{lem}
\label{lem:v_H}
Let $K$ be a field and let $H$ be a valuative subgroup of $K^\times$.
Then there exists a unique coarsest valuation $v_H$ such that $\Oc_{v_H}^\times \leq H$.
More precisely, if $w$ is a valuation of $K$ such that $\Oc_w^\times \leq H$, then $v_H$ is a coarsening of $w$; moreover $w = v_H$ if and only if $w(H)$ contains no non-trivial convex subgroups.

In particular, similar statements hold concerning valuative subgroups of $\Gc_K^a(n)$.
Namely, let $I$ be a valuative subgroup of $\Gc_K^a(n)$.
Then there exists a unique coarsest valuation $v_I$, depending only on $I$, such that $I \leq I_{v_I}(n)$.
If $w$ is a valuation of $K$ such that $I \leq I_w(n)$, then $v_I$ is a coarsening of $w$; moreover, $v_I = w$ if and only if $w(I^\perp)$ contains no non-trivial convex subgroups.
\end{lem}
\begin{proof}
Let $w$ be any valuation such that $\Oc_w^\times \leq H$ and consider the coarsening $v$ of $w$ which corresponds to the quotient of $\Gamma_w$ by the maximal convex subgroup of $w(H)$.
Namely, $v$ is the coarsest coarsening of $w$ such that $\Oc_v^\times \leq H$.
Furthermore, we note that $v$ induces a canonical isomorphism $K^\times/H \cong \Gamma_v/v(H)$.

By construction, $v(H)$ contains no non-trivial convex subgroups.
Thus, we deduce the following: If $x,y \in K^\times$ are given such that $x/y \in H$ and $v(x) < v(y)$, then there exists a $z \in K^\times$ such that $x \cdot H, y \cdot H \neq z \cdot H$ and $v(x) < v(z) < v(y)$.

Suppose that $h \in H$ and $x \in K^\times \smallsetminus H$.
Then $v(h) \neq v(x)$.
Moreover $v(h) < v(x)$ if and only if $h+x \in H$.
Similarly, $v(h) > v(x)$ if and only if $h+x \in x \cdot H$.
The discussion above shows that an element $h \in H$ such that $(1+x)/(h+x) \in H$ for all $x \in K^\times \smallsetminus H$, must be contained in $\Oc_v^\times$.
We deduce that $\Oc_v^\times$ depends only on $H$ and $K$, but not at all on the original choice of $w$.
More precisely, $\Oc_v^\times$ is exactly the set of all $h \in H$ such that, for all $x \in K^\times \smallsetminus H$, one has $(1+x)/(h+x) \in H$.
\end{proof}

\begin{defn}
\label{defn:v_H}
Suppose that $H$ is a valuative subgroup of $K^\times$.
We denote by $v_H$ the {\bf canonical valuation associated to $H$}, as described in Lemma \ref{lem:v_H}.
I.e. $v_H$ is the \emph{unique} coarsest valuation such that $\Oc_{v_H}^\times \leq H$.

Similarly, suppose $I$ is a valuative subgroup of $\Gc_K^a(n)$.
We denote by $v_I$ the canonical valuation $v_H$ associated to $H = I^\perp$.
Namely, $v_I$ is the unique coarsest valuation such that $I \leq I_{v_I}(n)$.
If $f \in \Gc_K^a(n)$ is a given valuative element, we will also denote by $v_f$ the valuation $v_{\langle f \rangle} = v_{\ker f}$.
\end{defn}

\subsection{C-pairs and Comparability of Valuations}

\begin{prop}
\label{prop:comparability-of-valuatives-in-rigid}
Let $f,g \in \Gc_K^a(n)$ be given valuative elements.
Denote the map $(f,g) : K^\times \rightarrow R_n\times R_n$ by $\Psi$.
Then the following are equivalent:
\begin{enumerate}
\item The valuations $v_f$ and $v_g$ are comparable.
\item The subgroup $\langle f,g \rangle$ is valuative.
\item For all $x \in K^\times \smallsetminus\{0,1\}$, the subgroup $\langle \Psi(1-x),\Psi(x) \rangle$ is cyclic.
\end{enumerate}
\end{prop}
\begin{proof}
First we prove that (1) and (2) are equivalent.
Assume (1) and without loss assume that $v_f$ is coarser than $v_g$.
Then $I_{v_f}(n) \leq I_{v_g}(n)$; thus $f,g \in I_{v_g}(n)$ and $\langle f,g \rangle$ is valuative.

Conversely, assume (2) and let $I := \langle f,g \rangle$.
Since $f \in I_{v_I}(n)$, we deduce from Lemma \ref{lem:v_H} that $v_f$ is a coarsening of $v_I$.
Similarly, $v_g$ is a coarsening of $v_I$.
Thus, $v_f$ and $v_g$ are comparable as they are both coarsenings of the valuation $v_I$.
Therefore (1) and (2) are equivalent.

The implication $(2) \Rightarrow (3)$ follows from the definition of ``valuative;'' see e.g. Lemma \ref{lem:I-and-D-give-c-pairs}.
It remains to show that $(3) \Rightarrow (2)$.

Assume that condition (3) holds true. 
Namely, for all $x  \in K^\times \smallsetminus\{0,1\}$, the subgroup $\langle\Psi(1-x), \Psi(x) \rangle$ is cyclic.
Since $\Psi(-1) = 0$, one equivalently has: for all $x \in K^\times \smallsetminus\{0,-1\}$, the subgroup  $\langle \Psi(1+x),\Psi(x) \rangle$ is cyclic.
Since $R_n$ is a quotient of a discrete valuation ring, $\langle \Psi(1+x),\Psi(x) \rangle$ is cyclic if and only if there exists some $a \in R_n$ such that $\Psi(1+x) = a \cdot \Psi(x)$ or $\Psi(x) = a \cdot \Psi(1+x)$.

Consider $T := \ker\Psi = \ker f \cap \ker g$.
We will use Lemma \ref{lem:valuative} to prove that $T$ is a valuative subgroup of $K^\times$, and this will yield (2).
We will first prove that for all $x \in K^\times \smallsetminus T$, one has $1+x \in T \cup x \cdot T$; i.e. we will show that for such an $x$, one has $\Psi(1+x) \in \{ \Psi(1), \Psi(x)\}$.
Let $x \in K^\times \smallsetminus T$ be given. 
We have two cases to consider: either $\Psi(1+x) = a \cdot \Psi(x)$, or $\Psi(x) = a \cdot \Psi(1+x)$.
As $f,g$ are valuative, we recall that $f(1+x) \in \{ f(1), f(x)\}$ and similarly $g(1+x) \in \{g(1),g(x)\}$.

\vskip 5pt
\noindent{\bf Case:} $\Psi(1+x) = a \cdot \Psi(x)$.

Namely, $f(1+x) = a\cdot f(x)$ and $g(1+x) = a\cdot g(x)$.
If $g(x) = 0$ or $f(x) = 0$, we trivially have $\Psi(1+x) = \Psi(1)$ or $ \Psi(1+x) = \Psi(x)$.
Thus, we may assume without loss that $f(x),g(x) \neq 0$.

If $f(1+x) = f(x)$ and $g(1+x) = g(x)$, then we are done as $\Psi(1+x) = \Psi(x)$.
Thus, we may assume, for example, that $f(1+x) = f(x)$ and $g(1+x) = g(1) = 0$.
Therefore $f(x) = a\cdot f(x)$ and $a\cdot g(x) = 0$.
Since $f(x) \neq 0$ and $f(x) = a \cdot f(x)$, we see that $a$ must be a unit in $R_n$ (in fact $a \in 1 + \ell \cdot R_n$) and this implies that $g(x) = 0$ -- contradiction!

To summarize:  $f(1+x) = f(x)$ if and only if $g(1+x) = g(x)$, and $f(1+x) = 0$ if and only if $g(1+x) = 0$.
In particular, $\Psi(1+x) = \Psi(x)$ or $\Psi(1+x) = \Psi(1)$, as required.

\vskip 5pt
\noindent{\bf Case:} $\Psi(x) = a \cdot \Psi(1+x)$.

Since $\Psi(x) \neq 0$, we may assume, for example that $f(x) \neq 0$ (otherwise $g(x) \neq 0$).
As $\Psi(x) = a \cdot \Psi(1+x)$, we see that $f(x) = a \cdot f(1+x)$.
Therefore, $f(1+x) \neq 0$ and, since $f$ is valuative, we see that $f(x) = f(1+x) = a \cdot f(x)$.
Therefore, $a \in 1 + \ell \cdot R_n$ and so $a$ is a unit.
Thus, one has $\Psi(1+x) = a^{-1} \cdot \Psi(x)$, and we have reduced to the previous case.

We have just proved that for all $x \in K^\times \smallsetminus T$, one has $1+x \in T \cup x \cdot T$.
To complete the proof that $T$ is valuative using Lemma \ref{lem:valuative}, we must show that, for all $x,y \in K^\times \smallsetminus T$ such that $1+x,1+y \in T$, one has $1+x(1+y) \in T$.
Let $x,y \in K^\times \smallsetminus T$ be given such that $\Psi(1+x) = \Psi(1+y) = 0$.
We must show that $\Psi(1+x\cdot (1+y)) = 0$.

Observe that $\Psi(1+x\cdot (1+y)) = a \cdot \Psi(x)$ for some $a \in \{0,1\}$; this is because $\Psi(1+y) = 0$ and $\Psi(x) = \Psi(x\cdot (1+y)) \neq 0$.
Furthermore, $\Psi(1+x\cdot (1+y)) = b \cdot \Psi(xy)$ for some $b \in \{0,1\}$; this is because $1+x\cdot (1+y) = 1+x+xy = t + xy$ for some $t \in T$, as $\Psi(1+x) = 0$.
We now have two cases to consider: either $\Psi(y) = -\Psi(x)$, or $\Psi(y) \neq -\Psi(x)$.

\vskip 5pt
\noindent{\bf Case:} $\Psi(y) = -\Psi(x)$.

In this case, $f(x) = 0$ if and only if $f(y) = 0$, and $g(x) = 0$ if and only if $g(y) = 0$.
If $f(x) = 0$ then $f(1+x\cdot (1+y)) = 0$ as well, since $f(1+x\cdot (1+y)) = a \cdot f(x)$; similarly, if $g(x) = 0$ then $g(1+x\cdot (1+y)) = 0$.
If $f(x) \neq 0$ then $f(y) \neq 0$ and therefore $f(1+x\cdot (1+y)) = 0$ since $f$ is valuative; similarly, if $g(x) \neq 0$ then $g(1+x\cdot (1+y)) = 0$.
In any case, we have $\Psi(1+x\cdot (1+y)) = 0$.

\vskip 5pt
\noindent{\bf Case:} $\Psi(y) \neq -\Psi(x)$.

Recall that $\Psi(1+x\cdot (1+y)) = a \cdot \Psi(x) = b \cdot \Psi(xy)$ for some $a,b \in \{0,1\}$.
Furthermore, we recall that $\Psi(x),\Psi(y) \neq 0$ and by assumption $\Psi(xy) \neq 0$.
Therefore, the only possibility is that $a,b = 0$.
In other words, $\Psi(1+x\cdot (1+y)) = 0$.
This implies (2) using Lemma \ref{lem:valuative}.
\end{proof}

\begin{remark}
\label{rem:c-pair-vs-rigid}
In this remark we will compare the condition of Proposition \ref{prop:comparability-of-valuatives-in-rigid} with the C-pair property.
Let $f,g \in \Gc_K^a(n)$ be given and denote the map $(f,g) : K^\times \rightarrow R_n \times R_n$ by $\Psi$.
Assume that for all $x \in K^\times \smallsetminus \ker\Psi$, the subgroup $\langle \Psi(x),\Psi(1-x) \rangle$ is cyclic.
Clearly, in this case, $f,g$ form a C-pair.
The converse holds true if $n = 1$ or $n = \infty$ since $R_n$ is a domain in these cases.
For general $n \in \Nbar$, however, the converse is completely false.
Lemma \ref{lem:c-pairs-vs-rigid} describes a weak converse of the statement for a general $n$, which follows from our cancellation principle.
\end{remark}

\begin{lem}
\label{lem:c-pairs-vs-rigid}
Let $n \in \Nbar$ be given and let $M := \Mfr_1(n)$.
Suppose that $a,b,c,d \in R_M$ are given such that $ad = bc$.
Then $\langle (a,b)_n,(c,d)_n \rangle$ is a cyclic subgroup of $R_n \times R_n$.

In particular, the following holds.
Let $f,g \in \Gc_K^a(M)$ be a given and denote the map $(f_n,g_n) : K^\times \rightarrow R_n \times R_n$ by $\Psi$.
Assume furthermore that $f,g$ form a C-pair.
Then, for all $x \in K^\times \smallsetminus \ker\Psi$, the subgroup $\langle \Psi(1-x),\Psi(x) \rangle$ is cyclic.
\end{lem}
\begin{proof}
The $n = \infty$ case is trivial, as noted in Remark \ref{rem:c-pair-vs-rigid}.
Thus, we may assume that $n \in \Nb$.
Assume without loss that $a = ec$ for some $e \in R_M$ (otherwise $c = ea$ for some $e \in R_M$).
Thus $ad = bc = edc$.
If $c_n \neq 0$, then one has $(de)_n = b_n$ by the cancellation principle. 
Thus $(a,b)_n = e_n \cdot (c,d)_n$ and so $\langle (a,b)_n,(c,d)_n \rangle$ is cyclic.

On the other hand, if $c_n = 0$, then $a_n = 0$ as well.
Thus $\langle (a,b)_n,(c,d)_n \rangle = \langle (0,b_n),(0,d_n) \rangle$ is cyclic.
\end{proof}

Using the fact that, for any valuation $v$ of $K$, the canonical map $I_v(\Mfr_1(n)) \rightarrow I_v(n)$ is surjective (as $\Gamma_v = K^\times/\Oc_v^\times$ is torsion-free), along with Proposition \ref{prop:comparability-of-valuatives-in-rigid}, the discussion of Remark \ref{rem:c-pair-vs-rigid} and Lemma \ref{lem:c-pairs-vs-rigid}, we deduce the following lemma which summarizes the discussion:
\begin{lem}
\label{lem:comp-of-valuations-in-C-pairs}
Let $f_i \in \Gc_K^a(n)$ be a collection of valuative elements.
Then the following are equivalent:
\begin{enumerate}
\item The valuations $v_{f_i}$ are all comparable.
\item The subgroup $I := \langle f_i \rangle_i$ is valuative.
\item For all pairs of indices $i,j$, there exists a C-pair $g,h \in \Gc_K^a(\Mfr_1(n))$ such that $g_n = f_i$ and $h_n = f_j$.
\end{enumerate}
Moreover, if these equivalent conditions hold true, then $v_I$ is the \emph{valuation-theoretic supremum} of the (comparable) valuations $v_{f_i}$.
\end{lem}
\begin{proof}
Assume (1).
Since $v_{f_i}$ are comparable, their valuation-theoretic supremum is well-defined and we call it $w$.
Then for all $i$ one has $f_i \in I_w(n)$; therefore $\langle f_i \rangle_i$ is valuative and we obtain (2).

Assume (2) and let $v := v_I$.
Observe that the canonical map $I_v(\Mfr(n)) \rightarrow I_v(n)$ is surjective, since $\Gamma_v = K^\times / \Oc_v^\times$ is torsion-free.
By Lemma \ref{lem:I-and-D-give-c-pairs}, $I_v(\Mfr(n))$ is a C-group and, since $I \leq I_v(n)$, we obtain (3).

Assume (3).
Then (1) follows from Lemma \ref{prop:comparability-of-valuatives-in-rigid}.
This completes the proof that (1), (2) and (3) are equivalent.

Now assume that the three equivalent conditions hold.
Consider $v := v_I$ and $w$ the valuation-theoretic supremum of $v_{f_i}$.
Since $f_i \in I_v(n)$, we see that $v_{f_i}$ is a coarsening of $v$ for all $i$.
This implies that $w$ is a coarsening of $v$ by the definition of $w$.

On the other hand, $v_{f_i}$ is a coarsening of $w$, and thus $f_i \in I_w(n)$ for all $i$.
Therefore, $v$ is a coarsening of $w$ by Lemma \ref{lem:v_H}.
We deduce that $v = w$, and this completes the proof of the lemma.
\end{proof}

\begin{lem}
\label{lem:c-pair-with-val-gives-decomp}
Let $n \in \Nbar$ be given and let $M := \Mfr_1(n)$.
Let $K$ be a field and let $f \in \Gc_K^a(M)$ be a valuative element.
Suppose that $g \in \Gc_K^a(M)$, and that $f,g$ form a C-pair.
Then $g_n \in D_{v}(n)$, where $v = v_{f_n}$ is the canonical valuation associated to $f_n$.
\end{lem}
\begin{proof}
Let $v = v_{f_n}$ denote the valuation associated to $f_n$; we must show that $g_n(1+\mf_v) = 0$.
Let $x \in K^\times$ be given such that $v(x) > 0$.
We will show that $g_n(1-x) = 0$ and this will complete the proof.

\vskip 5pt
\noindent{\bf Case:} $f_n(x) \neq 0$.

One has $f_n(1-x) = 0$ since $f_n \in I_v(n) \leq D_v(n)$.
Since $f$ is valuative, one has $f(1-x) \in \{f(1),f(x)\}$.
Since $f_n(x) \neq 0$ and thus $f(x) \neq 0$, we see that $f(1-x) = 0$.
Because $f,g$ form a C-pair, we deduce that $f(x)g(1-x) = 0$.
Finally, since $f_n(x) \neq 0$, we deduce from the cancellation principle that $g_n(1-x) = 0$.

\vskip 5pt
\noindent{\bf Case:} $f_n(x) = 0$.

By Lemma \ref{lem:v_H}, there exists a $y$ such that $0 < v(y) < v(x)$ and $f_n(y) \neq 0$.
Now, by the first case, we deduce that $g_n(1-y) = 0$.
Moreover, $v(y+x\cdot(1-y)) = v(y)$ and so $f_n(y+x\cdot(1-y)) = f_n(y) \neq 0$.
By the first case again, we have $g_n((1-y)\cdot(1-x)) = g_n(1-(y+x\cdot (1-y))) = 0$.
But this implies that $g_n(1-x) = 0$ as well since $g_n(1-y) = 0$.
This concludes the proof of the lemma.
\end{proof}

\subsection{Detecting subgroups of $I_v(n)$ and $D_v(n)$}

We are now ready to state and prove the main theorem of the paper which deals with C-groups.
This theorem, along with Theorem \ref{thm:cl-to-c-pairs}, gives a direct generalization of the main theorem of \cite{Bogomolov2007}.

\begin{thm}
\label{thm:c-groups-to-valuative}
Let $n \in \Nbar$ be given and let $N := \Nfr(\Mfr_1(n))$.
Let $D'' \leq \Gc_K^a(N)$ be given and assume that $D''$ is a C-group.
Then $D := D''_n$ contains a valuative subgroup $I \leq D$ such that:
\begin{itemize}
\item The quotient $D/I$ is cyclic.
\item One has $D \leq D_{v_I}(n)$.
\end{itemize}
\end{thm}
\begin{proof}
Let $M := \Mfr_1(n)$ and $D' := D''_M$.
Let $I'$ denote the subgroup of $D'$ which is generated by all $h \in D'$ such that $h$ is a \emph{valuative element}.
For all $f,g \in D'$, the quotient $\langle f,g \rangle/(\langle f,g \rangle \cap I')$ is cyclic by Theorem \ref{thm:main-c-pairs-thm}.
Thus, $D'/I'$ must be cyclic.

Moreover, by Lemma \ref{lem:comp-of-valuations-in-C-pairs}, $I'$ is valuative; thus $I := I'_n$ is valuative as well.
We must now prove that $D \leq D_{v_I}(n)$.

By Lemma \ref{lem:c-pair-with-val-gives-decomp}, for all $d \in D := D'_n$ and $h \in I$, one has $d \in D_{v_h}(n)$.
Since $v_I$ is the \emph{valuation-theoretic supremum} of the valuations $v_h$ for $h \in I$ (Lemma \ref{lem:comp-of-valuations-in-C-pairs}), one has $D \leq D_{v_I}(n)$, as required.
\end{proof}

Next we will prove a theorem which detects $I_v(n)$ within subgroups of $D_v(n)$ in a more precise way.
We will first need a technical lemma which is a reformulation of the approximation theorem for independent valuations.

\begin{lem}
\label{lem:non-valuative-comparable}
Let $v_1,v_2$ be two valuations of a field $K$ and assume that $f$ is a \emph{non-valuative} element of $\Gc_K^a(n)$ such that $f \in D_{v_1}(n) \cap D_{v_2}(n)$.
Then $v_1,v_2$ are comparable.
\end{lem}
\begin{proof}
Denote by $w$ the valuation associated to the finest common coarsening of $v_1,v_2$; i.e. $\Oc_w = \Oc_{v_1} \cdot \Oc_{v_2}$.
Also, let $H := \ker f$.
Further, denote by $H_w$ the kernel of the canonical surjection $k(w)^\times \rightarrow (\Oc_w^\times \cdot H) / H$.
For $i = 1,2$, consider $w_i = v_i/w$ the valuations of $k(w)$ induced by $v_i$.

By construction of $w_1,w_2$, if $w_1,w_2$ are both non-trivial, then they must be independent; we claim that this doesn't happen.
If $w_1,w_2$ are indeed independent, then we would have $(1+\mf_{w_1}) \cdot (1+\mf_{w_2}) = k(w)^\times$ by the approximation theorem for independent valuations.
However, we note that $H_w \neq k(w)^\times$, since $(\Oc_w^\times \cdot H) / H \cong k(w)^\times/H_w$ and $\Oc_w^\times$ is not contained in $H$ by our assumption on $f$.

On the other hand, $f \in D_{v_1}(n) \cap D_{v_2}(n)$; thus $1+\mf_{v_1} \leq H$ and $1+\mf_{v_2} \leq H$.
This implies that $1+\mf_{w_1} \leq H_w$ and $1+\mf_{w_2} \leq H_w$.
Therefore, $(1+\mf_{w_1}) \cdot (1+\mf_{w_2}) \leq H_w$, while $H_w$ is properly contained in $k(w)^\times$.
In particular, we see that $w_1,w_2$ cannot be independent.
Therefore, at least one of $w_1$ or $w_2$ must be trivial and thus $v_1,v_2$ are comparable.
\end{proof}

\begin{thm}
\label{thm:main-c-groups}
Let $n \in \Nbar$ be given and let $N := \Nfr(\Mfr_2(\Mfr_1(n)))$.
Let $I'' \leq D'' \leq \Gc_K^a(N)$ be given and consider $I := I''_n$ and $D := D''_n$.
Assume that $I'' \leq \Ibc(D'')$, and that $D$ is not a C-group.
Then $I$ is valuative and $D \leq D_{v_I}(n)$.
\end{thm}
\begin{proof}
Let $M := \Mfr_1(n)$, and consider $I' := I''_M$ and $D' := D''_M$.
Since $D$ is not a C-group and $D = D'_n$, we deduce that $D'$ is not a C-group.
Arguing similarly to Theorem \ref{thm:c-groups-to-valuative}, it suffices to prove that every element $f \in I'$ is valuative.
Assume for a contradiction that $f \in I'$ is a \emph{non-valuative} element.

Let $g_1,g_2 \in D'$ be given such that, for $i = 1,2$, the group $\langle f,g_i \rangle$ is non-cyclic.
We will show that $\langle f,g_1,g_2 \rangle$ is a C-group.
Then, as we vary over all such $g_1,g_2$, we would deduce that $D'$ (and thus $D$) is a C-group as well; this will provide the required contradiction and complete the proof of the theorem.

For the rest of the proof, let $M' := \Mfr_2(M) = \Mfr_2(\Mfr_1(n))$.
Choose lifts $f' \in I''_{M'}$ resp. $g_i' \in D''_{M'}$ for $f$ resp. $g_i$; note that $f'$ is non-valuative.
Then by Theorem \ref{thm:main-c-pairs-thm}, there exist valuations $v_1,v_2$ of $K$ such that:
\begin{itemize}
\item One has $\langle f',g_i' \rangle \leq D_{v_i}(M')$ for $i = 1,2$.
\item The quotient $\langle f',g_i' \rangle / (\langle f',g'_i \rangle \cap I_{v_i}(M'))$ is cyclic for $i = 1,2$.
\end{itemize}
For $i = 1,2$, we deduce that there exist $a_i,b_i \in R_{M'}$ such that (1) $a_i f' + b_i g'_i \in I_{v_i}(M')$ and (2) at least one of $a_i,b_i$ is a unit.
Indeed, otherwise $\langle f',g'_i \rangle \cap I_{v_i}(M')$ would be contained in $\langle \ell \cdot f',\ell \cdot g'_i \rangle = \ell \cdot \langle f',g'_i \rangle$ but $\langle f',g'_i \rangle/\ell$ is non-cyclic since $\langle f,g_i \rangle$ is non-cyclic.

We will now show that $(b_i g'_i)_M \neq 0$, and, in particular, $(b_1)_M,(b_2)_M \neq 0$.
If $a_i$ is a unit and $(b_i g'_i)_M = 0$, this would imply that $f$ is valuative, contradicting our original assumption; thus $a_i \in R_{M'}^\times$ implies $(b_i g'_i)_M \neq 0$.
On the other hand, if $b_i$ is a unit, then $(b_i g'_i)_M \neq 0$ since $g_i \neq 0$.
In particular, we deduce that $(b_1)_M,(b_2)_M \neq 0$.

Since $f'$ is non-valuative and $f' \in D_{v_1}(M') \cap D_{v_2}(M')$, the valuations $v_1,v_2$ must be comparable by Lemma \ref{lem:non-valuative-comparable}.
In particular, $\langle f',a_1f' + b_1g'_1, a_2 f' + b_2g'_2 \rangle = \langle f',b_1g'_1,b_2g'_2 \rangle$ forms a C-group by Lemma \ref{lem:I-and-D-give-c-pairs} and Proposition \ref{prop:comparability-of-valuatives-in-rigid}.

By the cancellation principle, $\langle f,g_1,g_2 \rangle$ form a C-group as well, as follows.
The pairs $f,g_i$, for $i = 1,2$ are C-pairs by assumption.
As for $g_1,g_2$, for all $x \in K^\times \smallsetminus\{0,1\}$ one has:
\[ b_1  b_2  \cdot g'_1(1-x) g'_2(x) = b_1 b_2 \cdot g'_1(x) g'_2(1-x). \]
Since $(b_1)_M,(b_2)_M \neq 0$ and $M' = \Mfr_2(M)$, the cancellation principle implies that 
\[g_1(1-x)g_2(x) = g_1(x)g_2(1-x). \]
This shows that $g_1,g_2$ form a C-pair, as required.
\end{proof}

\section{Detecting $D_v(n)$ and $I_v(n)$}
\label{sec:detecting-d_vn-i_vn}

In this section, we show how to detect the subgroups $D_v(n)$ and $I_v(n)$ precisely for certain ``maximal'' valuations $v$.
We also show that, in the case of function fields, these ``maximal'' valuations include the Parshin chains of divisors.

Let $(K,v)$ be a valued field and let $f \in D_v(n)$ be given.
Then the restriction $f|_{\Oc_v^\times}$ defines a homomorphism $f_v : k(v)^\times \rightarrow R_n$ such that $f_v(-1) = 0$.
In particular this provides a canonical map $D_v(n) \rightarrow \Gc^a_{k(v)}(n)$, which we denote by $f \mapsto f_v$.
This map, in some sense, forces the C-pair property as we see in the following lemma.

\begin{lem}
\label{lem:C-pair-compat-in-res-fields}
Let $(K,v)$ be a valued field and let $n \in \Nbar$ be given.
Let $w$ be a refinement of $v$ and consider $w/v$ the valuation of $k(v)$ induced by $w$.
Then the following hold:
\begin{enumerate}
\item The map $D_v(n) \rightarrow \Gc_{k(v)}^a(n)$ given by $f \mapsto f_v$ induces the following compatible isomorphisms:
\begin{enumerate}
\item $D_v(n)/I_v(n) \cong \Gc_{k(v)}^a(n)$.
\item $D_w(n)/I_v(n) \cong D_{w/v}(n)$.
\item $I_w(n)/I_v(n) \cong I_{w/v}(n)$.
\end{enumerate}
\item Let $f,g \in D_v(n)$ be given. Then $f,g$ form a C-pair if and only if their images $f_v,g_v$ in $\Gc^a_{k(v)}(n)$ form a C-pair.
\end{enumerate}
\end{lem}
\begin{proof}
\noindent{\bf Proof of (1):}

Assume with no loss that $n \in \Nb$ as the $n=\infty$ case follows in the limit.
Consider the short exact sequence:
\[ 1 \rightarrow k(v)^\times/\pm1 \rightarrow K^\times/\pm (1+\mf_v) \rightarrow \Gamma_v \rightarrow 1. \]
Tensoring this with $\Z/\ell^n$ and noting that $\Gamma_v$ is torsion-free, we obtain the following short exact sequence:
\[ 1 \rightarrow (k(v)^\times /\ell^n)/\pm 1 \rightarrow (K^\times/\ell^n)/\pm (1+\mf_v) \rightarrow \Gamma_v/\ell^n \rightarrow 1.\]
Applying the functor $\Hom(\bullet,\Z/\ell^n)$, we deduce that the following short sequence is exact by Pontryagin Duality:
\[ 1 \rightarrow I_v(n) \rightarrow D_v(n) \rightarrow \Gc_{k(v)}^a(n) \rightarrow 1. \]
This shows isomorphism (a).
Isomorphism (b) is obtained similarly by starting with the following exact sequence:
\[ 1 \rightarrow (k(v)^\times)/\pm(1+\mf_{w/v}) \rightarrow K^\times/\pm(1+\mf_w) \rightarrow \Gamma_v \rightarrow 1.\]
Isomorphism (c) is obtained similarly by starting with the following exact sequence:
\[ 1 \rightarrow \Gamma_{w/v} \rightarrow \Gamma_w \rightarrow \Gamma_v \rightarrow 1.\]

\vskip 5pt
\noindent{\bf Proof of (2):}

If $f,g$ form a C-pair then clearly $f_v,g_v$ form a C-pair as well.
Conversely, assume that $f_v,g_v$ are a C-pair.
Let $x \in K\smallsetminus\{0,1\}$ be given.

\vskip 5pt
\noindent{\bf Case:} $v(x) > 0$.

In this case, $1-x \in 1+\mf_v \leq \ker f \cap \ker g$.
Thus, $f(1-x)g(x) = 0  = f(x)g(1-x)$.

\vskip 5pt
\noindent{\bf Case:} $v(x) < 0$.

In this case, $x^{-1}(1-x) = x^{-1}-1 \in -(1+\mf_v)$ so that $(1-x) \in -x\cdot(1+\mf_v)$.
Thus, $f(1-x)g(x) = f(-x)g(x) = f(x)g(x) = f(x)g(-x) = f(x)g(1-x)$.

\vskip 5pt
\noindent{\bf Case:} $v(x) = 0$ and $v(1-x) > 0$.

In this case, we apply one of the previous cases replacing $x$ with $1-x$.

\vskip 5pt
\noindent{\bf Case:} $v(x) = v(1-x) = 0$.

We note that for all $z \in \Oc_v^\times$, one has $f(z) = f_v(\bar z)$ and $g(z) = g_v(\bar z)$, where $\bar z = z + \mf_v$ denotes the image of $z$ in $k(v)^\times$.
Since $f_v,g_v$ form a C-pair and $x,1-x \in \Oc_v^\times$, we deduce that $f(x)g(1-x) = f(1-x)g(x)$.
\end{proof}

\subsection{The set $\Vc_{K,n}$}
\label{sec:set-vc_k-n}

We will now show how to detect $I_v(n)$ and $D_v(n)$ precisely for certain natural collection of valuations $v$; we will call this collection $\Vc_{K,n}$.

In order to motivate our definition of $\Vc_{K,n}$, we begin with few remarks.
First, if $\Gamma_v$ contains a non-trivial $\ell$-divisible convex subgroup and $v'$ is the coarsening associated to this convex subgroup, then $I_v(n) = I_{v'}(n)$ by Lemma \ref{lem:v_H}.
Because of this, we must assume condition (V1) that $\Gamma_v$ contains no such subgroup.

Second, we will require a ``maximality'' condition on $v$ which ensures, in particular, that $k(v)$ has no $\ell$-Henselian valuations with non-$\ell$-divisible value group.
If $k(v)$ has such a valuation $w$, the compatibility in taking compositions of valuations essentially forces us to replace $v$ by $w \circ v$.
This will become condition (V2).

Lastly, we note that when $\Gc_K(n)$ is cyclic, one cannot expect to detect anything.
For example consider $K = \C((t))$.
Then $\Gc_K^a(n)$ is cyclic by Kummer theory, the whole $\Gc_K^a(n)$ is valuative and its corresponding valuation is the $t$-adic one.
On the other hand, we can consider $K = \F_p(\mu_{2\ell^n})$ (with $p \neq \ell$).
By Kummer theory, $\Gc_K^a(n)$ is again cyclic, but $K$ has no non-trivial valuations.
Because of this observation and the compatibility in taking residue fields (see Lemma \ref{lem:C-pair-compat-in-res-fields}), one cannot expect to detect $I_v(n)$ within $D_v(n)$ when $\Gc_{k(v)}^a(n)$ is cyclic.
Thus, we will need to assume condition (V3), that $\Gc_{k(v)}^a(n)$ is non-cyclic

\begin{defn}
\label{defn:Dc-Vc}
Let $n \in \Nbar$ be given.
Let $K$ be a field.

We will denote by $\Vc_{K,n}$ the collection of valuations $v$ of $K$ which satisfy the following three conditions:
\begin{enumerate}
\item[(V1)] The value group $\Gamma_v$ contains no non-trivial $\ell$-divisible convex subgroups. Equivalently by Lemma \ref{lem:v_H}, one has $v = v_I$ for $I = I_v(n)$.
\item[(V2)] The valuation $v$ is maximal among all valuations $w$ such that $D_v^n = D_w^n$ and $\Gamma_w$ contains no non-trivial $\ell$-divisible convex subgroups.
I.e. for all refinements $w$ of $v$ such that $D_w^n = D_v^n$ as subgroups of $\Gc_K^{a,n}$, one has $I_w^n = I_v^n$.
\item[(V3)] The group $\Gc_{k(v)}^a(n)$ is non-cyclic.
\end{enumerate}
For the sake of Example \ref{ex:prime-divs}, we further denote by $\Wc_{K,n}$ the collection of valuations $v$ of $K$ which only satisfy (V1) and (V2), although we will not use $\Wc_{K,n}$ in the statement of any theorem.

We also introduce the group-theoretical analogue of $\Vc_{K,n}$, which will make the statements in Remarks \ref{remark:main-detect-thm-mu} and \ref{remark:main-detect-thm-infty} much more elegant and intuitive.
We will need to use $N := \Nfr(\Mfr_2(\Mfr_1(n)))$ in this definition, although we omit it from the notation.
We denote by $\Dc_{K,n}$ the collection of subgroups $D \leq \Gc_K^a(n)$ endowed with $I \leq D$ which satisfy the following three conditions:
\begin{enumerate}
\item[(D1)] There exists $D' \leq \Gc_K^a(N)$ such that $(\Ibc(D'))_n = I$ and $D'_n = D$.
\item[(D2)] The subgroups $I \leq D \leq \Gc_K^a(n)$ are maximal with property (D1). 
Namely, if $D \leq E \leq \Gc_K^{a}(n)$ and $E' \leq \Gc_K^{a}(N)$ is given such that $E'_n = E$ and $I \leq (\Ibc(E'))_n$, then $D = E$ and $I = (\Ibc(E'))_n$.
\item[(D3)] One has $\Ibc(D) \neq D$; i.e. $D$ is not a C-group.
\end{enumerate}

To further make the notation easier in Remarks \ref{remark:main-detect-thm-mu} and \ref{remark:main-detect-thm-infty}, we will introduce notation for certain natural subsets of $\Vc_{K,n}$ and $\Dc_{K,n}$, relative to a fixed valuation $v_0$ of $K$.
\begin{enumerate}
\item We denote by $\Dc_{v_0,n}$ the subset of all $(I \leq D) \in \Dc_{K,n}$ such that $I_{v_0}(n) \leq I \leq D \leq D_{v_0}(n)$.
\item We denote by $\Vc_{v_0,n}$ the subset of all $v \in \Vc_{K,n}$ such that $v_0$ is a coarsening of $v$.
\end{enumerate}
\end{defn}

The set of valuation $\Vc_{K,n}$ contains many valuations of arithmetic/geometric interest.
The main motivating example of such valuations arise from prime divisors, as will be shown in the following example.

To keep the discussion as general as possible, we introduce some terminology.
We will say that a field $k$ is {\bf strongly $\ell$-closed} provided that any finite extension $k'|k$ satisfies $(k')^\times = (k')^{\times \ell}$.
For example, algebraically closed fields of any characteristic, and perfect fields of characteristic $\ell$ are strongly $\ell$-closed.
Observe that, if $v_0$ is a valuation of a strongly $\ell$-closed field $k$, then $k(v_0)$ is also strongly $\ell$-closed.

In the following example, we will show that geometric Parshin chains (i.e. compositions of valuations associated to Weil prime divisors) are elements of $\Wc_{K,n}$, if $K$ is a function field over a strongly $\ell$-closed field $k$.
In particular, the non-degenerate Parshin chains of \emph{non-maximal} length will lie in $\Vc_{K,n}$ while the non-degenerate Parshin chains of maximal length will lie in $\Wc_{K,n} \smallsetminus \Vc_{K,n}$.
Although the argument in Example \ref{ex:prime-divs} uses some results from sections \ref{sec:suff-many-roots} and \ref{sec:milnor-k-theory}, the proofs of these results do not depend on the argument given in this example.
We present this example here for the sake of continuity in exposition.

\begin{example}
\label{ex:prime-divs}
Our first claim will, in particular, imply that valuations associated to prime divisors (and more generally quasi-prime divisors) are elements of $\Wc_{K,n}$, and in most cases they are elements of $\Vc_{K,n}$.
The second claim concerns the valuation-theoretic composition of valuations in $\Wc_{\bullet,n}$.
Together, these two claims imply that Parshin-chains of (quasi-)prime divisors or non-maximal length are elements of $\Vc_{K,n}$ while the chains of maximal length are elements of $\Wc_{K,n}$.

\vskip 5pt
\noindent{\bf Prime Divisors:}

Suppose $K$ is an arbitrary field in which the polynomial $X^{2\ell^n}-1$ splits completely.
Let $v$ be a valuation of $K$ such that $\Gamma_v$ contains no non-trivial $\ell$-divisible convex subgroups.
Assume further that $k(v)$ is a function field over a strongly $\ell$-closed field $k$.
We claim that $v \in \Wc_{K,n}$.
Namely, we must prove that $v$ satisfies condition (V2).

If the transcendence degree of $k(v)|k$ is $0$, we observe that $k(v)^\times$ is $\ell$-divisible since $k$ is strongly $\ell$-closed.
Thus, it follows from the definitions that $v \in \Wc_{K,n} \smallsetminus \Vc_{K,n}$ in this case.

Now let us assume that $k(v)|k$ has transcendence degree $\geq 1$.
Assume that $w$ is a refinement of $v$ and that $D_w(n) = D_v(n)$.
Thus, we have the following inclusion of subgroups: $I_v(n) \leq I_w(n) \leq D_w(n) = D_v(n)$.
We must show that $I_v(n) = I_w(n)$.

Let $F := k(v)$ and consider the valuation $w/v$ of $F$ induced by $w$. 
Observe that $I_v(n) = I_w(n)$ if and only if $I_{w/v}(n) = 1$ as a subgroup of $\Gc_F^a(n)$, since we have a canonical isomorphism $I_w(n)/I_v(n) \cong I_{w/v}(n)$.
Thus, we can assume without loss of generality that $n = 1$ as $I_{w/v}(n)$ has the same rank as $I_{w/v}(1)$ (see Lemma \ref{lem:image-under-mod-N-to-mod-n} and/or the proof of Lemma \ref{lem:alt-defn-of-Vc}).

We must therefore prove that $I_{w/v}(1) = 0$.
Assume, for a contradiction, that $0 \neq f \in I_{w/v}(1)$ and let $T := \ker f$; note that $F^{\times \ell} \leq T$.
Clearly, $F^\times / T$ is cyclic; say, e.g. that $F^\times/T$ is generated by the image of $x \in F^\times$.
Namely, $F^\times/T = \langle x \cdot T \rangle \cong \Z/\ell$.

For all $g \in \Gc_F^a(1)$, the pair $f,g$ is a C-pair by Lemma \ref{lem:I-and-D-give-c-pairs}.
In particular, for all $H \leq F^\times$, such that $F^{\times \ell} \leq H$ and $F^\times/H \cong \Z/\ell$, the group $\Hom(F^\times/(H\cap T),\Z/\ell)$ is a C-group (considered as a subgroup of $\Gc_F^a(1)$).

Now assume that $y \in F^\times$ is \emph{any} element such that the images of $x,y$ are $\Z/\ell$ independent in $F^\times/\ell$.
In this case, we can choose $T_0$ such that $F^{\times \ell} \leq T_0 \leq T \leq F^\times$ and 
\[F^\times/T_0 = \langle x \cdot T_0,y \cdot T_0 \rangle \cong \Z/\ell \times \Z/\ell. \]
By the discussion above, for such a $T_0$, the subgroup $\Hom(F^\times/T_0,\Z/\ell)$ is a C-group.
By the K-theoretic criterion for C-pairs (Proposition \ref{prop:k-thy-c-pair}) we deduce, in particular, that $\{x,y\}_{T_0} \neq 0$ as an element of $K_2^M(F)/T_0$; the mod-$T_0$ Milnor K-theory groups are defined in \S\ref{sec:milnor-k-theory}.
In particular, $\{x,y\} \neq 0$ as an element of $K_2^M(F)/\ell$. 
We will show that this provides a contradiction by producing an element $y \in F^\times$, such that (1) $x,y$ have independent images in $F^\times/\ell$ and (2) $\{x,y\} = 0$ in $K_2(F)/\ell$.

First, since $x \in F^\times \smallsetminus F^{\times \ell}$ and $k$ is strongly $\ell$-closed, we deduce that $x$ is transcendental over $k$.
Consider the subfield $L := \overline{k(x)} \cap F$, the relative algebraic closure of $k(x)$ inside $F$.
Our aim will be to find $y \in k(x)^\times$ so that the images of $x,y$ in $L^\times/\ell$ are independent.
We have two cases to consider: $\Char k \neq \ell$ and $\Char k = \ell$.

\vskip 5pt
\noindent {\bf Case:} $\Char k \neq \ell$.

In this case, the existence of such a $y \in k(x)^\times$ is trivial since the image of the canonical map $k(x)^\times/\ell \rightarrow L^\times/\ell$ is infinite.
In fact, the image has finite index in $L^\times/\ell$ by Kummer theory since $L|k(x)$ is a finite extension and $\mu_\ell \subset k$.

\vskip 5pt
\noindent {\bf Case:} $\Char k \neq \ell$.

In this case, we see that $k$ is perfect and, since $x \notin L^{\times \ell}$, the extension $L|k(x)$ must be separable.
Consider the unique complete normal model $C$ for $L|k$ together with the (possibly branched) cover of curves $C \rightarrow \Pbb^1_k$ induced by $k(x) \rightarrow L$.
By the approximation theorem, there exists a prime divisor $P$ of $\Pbb^1_k$ and a function $y \in k(x)^\times$ such that $P$ is unramified in the cover $C \rightarrow \Pbb^1_k$, $P \neq 0,\infty$, and $v_P(y) = 1$; as usual, $v_P$ denotes the valuation of $k(x)$ associated to the prime divisor $P$.
Since $P$ is unramified in $C$, for any prolongation $P'$ of $P$ to $C$, one also has $v_{P'}(y) = 1$.
Moreover, as $P \neq 0,\infty$ and the divisor associated to $x$ on $\Pbb^1_k$ is precisely $0-\infty$, we deduce that the images of $x,y$ in $L^\times/\ell$ must be $\Z/\ell$-independent.

To summarize, in either case we have produced an element $y \in k(x)^\times$ such that the images of $x,y$ in $L^\times/\ell$ are independent.
Now we recall a theorem of Milnor stating that the following sequence is exact:
\[ 0 \rightarrow K_2^M(k) \rightarrow K_2^M(k(x)) \rightarrow \bigoplus_{P \in \Abb^1_k} K_1^M(k(P)) \rightarrow 0 \]
where the last map is the sum of the tame symbols associated to $v_P$, as $P$ ranges over the prime divisors of $\Pbb^1_k$ with support in $\Abb^1_k = \Spec k[x]$.
However, the extension $k(P)|k$ is finite and thus $k(P)^{\times \ell} = k(P)^\times$ since $k$ is strongly $\ell$-closed; in particular $K_1^M(k(P))/\ell = 0$.
Also, since $k^\times = k^{\times \ell}$, we see that $K_2^M(k)/\ell = 0$.
By tensoring the above exact sequence with $\Z/\ell$, we deduce that $K_2^M(k(x))/\ell = 0$.
In particular, $\{x,y\} = 0$ in $K_2^M(k(x))/\ell$.
By the functoriality of Milnor K-theory, we see that $\{x,y\} = 0$ in $K_2^M(F)/\ell$.

Since $L$ is relatively algebraically closed in $F$, the map $L^\times/\ell \rightarrow F^\times/\ell$ must be injective.
As the images of $x,y$ are independent in $L^\times/\ell$, their images must also be independent in $F^\times/\ell$.
This provides the desired contradiction, as we've produced an element $y \in F^\times$ such that the images of $x,y$ are independent in $F^\times/\ell$ while $\{x,y\} = 0$ in $K_2(F)/\ell$.

\vskip 5pt
\noindent{\bf Compositions of Valuations:}

We now show that compositions of valuations from $\Wc_{\bullet,n}$ lie in $\Wc_{\bullet,n}$.
Suppose that $v \in \Wc_{K,n}$ is given and $w' \in \Wc_{k(v),n}$.
Consider $w := w' \circ v$, the valuation theoretic composition of $w'$ and $v$.
By considering the canonical short exact sequence of value groups
\[ 1 \rightarrow \Gamma_{w'} \rightarrow \Gamma_{w} \rightarrow \Gamma_v \rightarrow 1 \]
we see immediately that $\Gamma_{w}$ contains no non-trivial $\ell$-divisible convex subgroups; thus condition (V1) holds true for $w$.

We must now show that (V2) holds true for $w = w' \circ v$. 
Suppose that $w_1$ is a refinement of $w$ such that $D_{w}(n) = D_{w_1}(n)$.
Since $v$ is a coarsening of $w$, it is also a coarsening of $w_1$.
This implies that $w'$ is a coarsening of $w_1/v$, as valuations of $k(v)$.
Because $D_{w}(n) = D_{w_1}(n)$, we see that $D_{w'}(n) = D_{w_1/v}(n)$ as subgroups of $\Gc_{k(v)}^a(n)$.
Since $w' \in \Wc_{k(v),n}$, we see that $I_{w'}(n) = I_{w_1/v}(n)$ by condition (V2).
Thus $I_{w}(n) = I_{w_1}(n)$ as subgroups of $\Gc_K^a(n)$, and condition (V2) holds true for $w = w' \circ v$.
\end{example}

In light of Theorem \ref{thm:main-c-groups}, in order to detect $I_v(n)$ and $D_v(n)$, we need a ``plethora'' of C-pairs in $\Gc_K^a(N)$ which arise from valuation theory.
We handle this by ensuring that the canonical maps $I_v(N) \rightarrow I_v(n)$ and $D_v(N) \rightarrow D_v(n)$ are surjective.
The map $I_v(N) \rightarrow I_v(n)$ is always surjective as $\Gamma_v$ is torsion-free; however, the map $D_v(N) \rightarrow D_v(n)$ may not be surjective in general.
However, this map is surjective in two important cases which we consider below.
The first case is when $K$ contains sufficiently many roots of unity (and thus the same is true for $k(v)$).
The second case is when $N = n$; denoting $N = \Nfr(\Mfr_2(\Mfr_1(n)))$ as in Theorem \ref{thm:main-c-groups}, we see that $N=n$ iff $n = 1$ or $n = \infty$.

\subsection{Sufficiently Many Roots of Unity}
\label{sec:suff-many-roots}

In this subsection, we will not restrict to fields $K$ whose characteristic is different from $\ell$.
Thus, instead of saying that $K$ contains roots of unity (for some readers this implicitly restricts the characteristic), we will say that the polynomial $X^{2\ell^N}-1$ splits completely in $K$.
If $N = \infty$, we take this to mean that $X^{2\ell^m}-1$ splits completely for all $m \in \Nb$.

Since it will be used in the proof of the following lemma, we take note of the following trivial fact: If $v$ is a valuation of $K$ and $X^{2\ell^N}-1$ splits completely in $K$, then the same polynomial splits completely in $k(v)$.

\begin{lem}
\label{lem:image-under-mod-N-to-mod-n}
Let $(K,v)$ be a valued field.
Let $N,n \in \Nbar$ be given with $N \geq n$, and assume that the polynomial $X^{2\ell^{N}}-1$ splits completely in $K$.
Then the following hold:
\begin{enumerate}
\item The following canonical maps are surjective:
\begin{itemize}
\item $\Gc_K^a(N) \rightarrow \Gc_K^a(n)$.
\item $I_v(N) \rightarrow I_v(n)$.
\item $D_v(N) \rightarrow D_v(n)$.
\end{itemize}
\item The two abelian pro-$\ell$ groups $\Gc_K^a(N)$ and $\Gc_K^a(n)$ have the same rank.
\item Let $w$ be a refinement of $v$ and consider the following canonical inclusion of subgroups of $\Gc_K^a(m)$ for $m = n,N$:
\[ I_v(m) \leq I_w(m) \leq D_w(m) \leq D_v(m).\]
Then $I_v(N) = I_w(N)$ if and only if $I_v(n) = I_w(n)$; and $D_w(N) = D_v(N)$ if and only if $D_w(n) = D_v(n)$.
\end{enumerate}
\end{lem}
\begin{proof}
\noindent{\bf Proof of (1):}

This is trivial if $n=\infty$ (since this forces $N = \infty = n$), and thus we can assume that $n \in \Nb$.
We will also assume that $N \in \Nb$ as the case where $N=\infty$ would follow from this in the limit.

Our assumption that $X^{2\ell^{N}}-1$ splits completely ensures that $-1 \in K^{\times \ell^N}$, and thus $K^\times/\ell^m$ is Pontryagin dual to $\Gc_K^a(m)$ for all $m \leq N$.
Therefore, the Pontryagin dual of the canonical map $\Gc_K^a(N) \rightarrow \Gc_K^a(n)$ is precisely the map:
\[ K^\times/\ell^n \xrightarrow{\ell^{N-n}} K^\times/\ell^N. \]
By Pontryagin duality, it suffices to prove that the map $K^\times/\ell^n \rightarrow K^\times/\ell^N$ is injective.

Suppose $x \in K^\times$ is given such that $x^{\ell^{N-n}} =  y^{\ell^N}$ for some $y \in K^\times$.
Then $x = y^{\ell^n} \cdot \zeta$ for some $\zeta$ with $\zeta^{\ell^{N-n}} = 1$.
But our assumptions ensure that $\zeta \in K^{\times \ell^n}$ and thus $x \in K^{\times \ell^n}$.
Thus the map $K^\times/\ell^n \rightarrow K^\times/\ell^N$ is injective and, dually, the map $\Gc_K^a(N) \rightarrow \Gc_K^a(n)$ is surjective.

The claim concerning the surjectivity of the map $I_v(N) \rightarrow I_v(n)$ is trivial as $\Gamma_v = K^\times/\Oc_v^\times$ is torsion-free.
Finally, the claim concerning surjectivity of $D_v(N) \rightarrow D_v(n)$ follows from the fact that $\Gc_{k(v)}^a(N) \rightarrow \Gc_{k(v)}^a(n)$ is surjective, along with the facts that $D_v(N)/I_v(N) = \Gc_{k(v)}^a(N)$ and $D_v(n)/I_v(n) = \Gc_{k(v)}^a(n)$ (Lemma \ref{lem:C-pair-compat-in-res-fields}). 

\vskip 5pt
\noindent{\bf Proof of (2):}

As above, we can assume with no loss that $N,n\in\Nb$.
Arguing as in the proof of claim (1), one has:
\[ \ell^{n} \cdot \Gc_K^a(N) = \Hom(K^\times/\pm 1, \ell^{n} \cdot R_N).\]
Thus the surjective map $\Gc_K^a(N) \rightarrow \Gc_K^a(n)$ corresponds precisely to the quotient $\Gc_K^a(N) \rightarrow \Gc_K^a(N)/\ell^{n} = \Gc_K^a(n)$ and this proves that $\Gc_K^a(N)$ and $\Gc_K^a(n)$ have the same rank as abelian pro-$\ell$ groups.

\vskip 5pt
\noindent{\bf Proof of (3):}

By claim (1), $I_v(N) = I_w(N)$ implies that $I_v(n) = I_w(n)$ and similarly $D_v(N) = D_w(N)$ implies that $D_v(n) = D_w(n)$.
To prove the converse it suffices to assume that $v$ is the trivial valuation by replacing $K$ with $k(v)$ and $w$ by $w/v$.
Indeed, by Lemma \ref{lem:C-pair-compat-in-res-fields}, one has $I_w(n)/I_v(n) = I_{w/v}(n)$ and $D_w(n)/I_v(n) = D_{w/v}(n)$.
Making these assumptions, we have $I_v(n) = 1$ and $D_v(n) = \Gc_K^a(n)$.

Assume that $I_w(n) = I_v(n) = 1$.
Then $\Gamma_w = \ell^n \cdot \Gamma_w$ and so $\Gamma_w = \ell^N \cdot \Gamma_w$ since $\Gamma_w$ is torsion-free; this implies that $I_w(N) = 1$.

Now assume that $D_w(n) = D_v(n) = \Gc_K^a(n)$.
Then $1+\mf_w \leq K^{\times \ell^n}$; we must show that $1+\mf_w \leq K^{\times \ell^N}$.
Let $x \in 1+\mf_w$ be given and let $y \in K^\times$ be such that $x = y^{\ell^n}$.
Applying $w$ to both sides we deduce that $y \in \Oc_w^\times$.

As usual, we denote by $t \mapsto \bar t$ the map $\Oc_w^\times \rightarrow k(w)^\times$.
The above implies that $\bar y ^{\ell^n} = \bar 1$.
Since the polynomial $X^{2\ell^N}-1$ splits in $k(w)$, there exists some $z \in \Oc_w^\times$ such that $\bar z^{\ell^{N-n}} = \bar y$.
Thus, $y = z^{\ell^{N-n}} \cdot a$ for some $a \in 1+\mf_w$.
Therefore, $x = z^{\ell^N} a^{\ell^n}$.
But, as $a \in K^{\times \ell^n}$, we see that $a^{\ell^n} \in K^{\times \ell^{2n}}$.
Continuing inductively in this way, we deduce that $x \in K^{\times \ell^N}$.
Thus, $1+\mf_w \leq K^{\times \ell^n}$, and this implies that indeed $D_w(N) = \Gc_K^a(N)$, as required.
\end{proof}

\begin{prop}
\label{prop:maximal-C-group-D}
Let $n \in \Nbar$ be given and let $N \geq \Nfr(\Mfr_1(n)))$.
Let $K$ be a field and assume that $X^{2\ell^{N}}-1$ splits completely in $K$.
Let $D \leq \Gc_K^a(n)$ be given.
Then the following are equivalent:
\begin{enumerate}
\item There exists a valuation $v$ of $K$ such that $D \leq D_v(n)$ and $D/(D \cap I_v(n))$ is cyclic.
\item There exists a subgroup $D' \leq \Gc_K^a(N)$ such that $D'$ is a C-group and $D'_n = D$.
\end{enumerate}
\end{prop}
\begin{proof}
The implication $(2) \Rightarrow (1)$ is Theorem \ref{thm:c-groups-to-valuative}.

We now prove $(1) \Rightarrow (2)$.
Assume (1).
Let $I := D \cap I_v(n)$ and choose $f \in D$ such that $\langle I,f \rangle = D$.
Choose $f' \in D_v(N)$ a lifting of $f$, via Lemma \ref{lem:image-under-mod-N-to-mod-n}.
Furthermore, consider the pre-image $I' \leq I_v(N)$ of $I \leq I_v(n)$ under the surjective map $I_v(N) \rightarrow I_v(n)$.
Then $I'_n = I$ and $f'_n = f$.
Moreover, by Lemma \ref{lem:I-and-D-give-c-pairs}, we see that $\langle I',f' \rangle$ is a C-group.
Thus $D' = \langle I',f' \rangle$ satisfies the requirements of (2).
\end{proof}

\begin{prop}
\label{prop:center-mu}
Let $n \in \Nbar$ be given and let $N \geq \Nfr(\Mfr_2(\Mfr_1(n)))$.
Let $K$ be a field and assume that $X^{2\ell^{N}}-1$ splits completely in $K$.
Assume that $\Ibc(\Gc_K^a(n)) \neq \Gc_K^a(n)$.
Consider $I' := \Ibc(\Gc_K^a(N))$ and let $I := I'_n$.
Then the following hold:
\begin{enumerate}
\item The subgroup $I$ is valuative and $v := v_I \in \Vc_{K,n}$.
\item One has $I = I_v(n)$ and $D_v(n) = \Gc_K^a(n)$.
\end{enumerate}
\end{prop}
\begin{proof}
We know that $I$ is valuative and, denoting $v := v_I$, one has $D_v(n) = \Gc_K^a(n)$ from Theorem \ref{thm:main-c-groups}.
Thus $D_v(N) = \Gc_K^a(N)$ by Lemma \ref{lem:image-under-mod-N-to-mod-n}.
In particular, $I_v(N) \leq I'$ by Lemma \ref{lem:I-and-D-give-c-pairs}.
Namely, $I_v(n) \leq I$ since $I'_n = I$ and $(I_v(N))_n = I_v(n)$.
Since $I \leq I_v(n)$, we deduce that $I = I_v(n)$.

We must prove that $v \in \Vc_{K,n}$.
Condition (V1) follows from the fact that $v = v_I$ is the canonical valuation associated to a valuative subgroup $I$ of $\Gc_K^a(n)$ (see Lemma \ref{lem:v_H}).

Concerning condition (V2), suppose that $w$ is a refinement of $v$ such that $D_v(n) = \Gc_K^a(n) = D_w(n)$.
Similarly to above, by Lemma \ref{lem:I-and-D-give-c-pairs} we have $I_w(N) \leq I'$; therefore $I_w(n) \leq I_v(n)$.
Since $I_v(n) \leq I_w(n)$ as well, we see that $I_w(n) = I_v(n)$.
Lastly, $\Gc_K^a(n)/I$ is non-cyclic since $\Gc_K^a(n)$ is not a C-group; thus condition (V3) holds and we see that $v \in \Vc_{K,n}$.
\end{proof}

\begin{thm}
\label{thm:maximal-I-D-mu}
Let $n \in \Nbar$ be given and let $N \geq \Nfr(\Mfr_2(\Mfr_1(n)))$.
Let $K$ be a field and assume that $X^{2\ell^{N}}-1$ splits completely in $K$.
Let $I \leq D \leq \Gc_K^a(n)$ be given.
Then there exists a valuation $v \in \Vc_{K,n}$ such that $I = I_v(n)$ and $D = D_v(n)$ if and only if the following conditions hold:
\begin{enumerate}
\item There exist $D' \leq \Gc_K^a(N)$ such that $(\Ibc(D'))_n = I$ and $D'_n = D$.
\item The subgroups $I \leq D \leq \Gc_K^a(n)$ are maximal with property (1). Namely, if $D \leq E \leq \Gc_K^{a}(n)$ and $E' \leq \Gc_K^{a}(N)$ is given such that $E'_n = E$ and $I \leq (\Ibc(E'))_n$, then $D = E$ and $I = (\Ibc(E'))_n$.
\item One has $\Ibc(D) \neq D$; i.e. $D$ is not a C-group.
\end{enumerate}
\end{thm}
\begin{proof}
Let $I \leq D$ be given which satisfy conditions (1),(2),(3) above.
By Theorem \ref{thm:main-c-groups} and conditions (1),(3), we see that $I$ is valuative and, denoting $v := v_I$, one has $D \leq D_v(n)$.

We first show that $I_v(n) = I$ and $D_v(n) = D$.
Consider $I'' := I_v(N) \leq D_v(N) =: D''$.
By Lemma \ref{lem:image-under-mod-N-to-mod-n}, one has $I''_n = I_v(n)$ and $D''_n = D_v(n)$.
Furthermore, by Lemma \ref{lem:I-and-D-give-c-pairs}, one has $I'' \leq \Ibc(D'')$.
Thus, $I \leq I_v(n) = I''_n \leq (\Ibc(D''))_n =: J$ and $D \leq D_v(n) = D''_n$.
By condition (2) on $I \leq D$ we deduce that $I = J$ and $D = D_v(n)$.
Also, $I \leq I_v(n) \leq J$ and $I = J$ implies that $I = I_v(n)$, as required.

We now show that $v = v_I$ is an element of $\Vc_{K,n}$.
Since $v = v_I$, condition (V1) holds true for $v$ by Lemma \ref{lem:v_H}.
Concerning condition (V2), assume that $w$ is a refinement of $v$ such that $D_v(n) = D_w(n)$.
Then $I_v(n) \leq I_w(n) \leq D_w(n) = D_v(n)$.
By Lemma \ref{lem:I-and-D-give-c-pairs}, we see that:
\[ I_v(n) \leq I_w(n) \leq (\Ibc(D_w(N)))_n \leq (D_w(N))_n = D_w(n) = D_v(n). \]
This implies that $I_v(n) = I_w(n)$ by condition (2) on $I \leq D$, and thus condition (V2) holds true for $v$.
Lastly, $\Gc^a_{k(v)}(n) = D_v(n)/I_v(n)$ is non-cyclic as $D_v(n)$ is not a C-group by condition (3) (Lemma \ref{lem:I-and-D-give-c-pairs}); therefore condition (V3) holds true for $v$.

Conversely, we assume that $v \in \Vc_{K,n}$ is given and consider $I := I_v(n) \leq D_v(n) =: D$.
We must show that $I \leq D$ satisfy conditions (1),(2),(3) of the theorem.

We first show condition (2).
Suppose that $D \leq E \leq \Gc_K^a(n)$ and $E' \leq \Gc_K^a(N)$ are given with $E = E'_n$ and $I \leq (\Ibcl(E'))_n =: J$.
By Theorem \ref{thm:main-c-groups}, the subgroup $J$ is valuative and $E \leq D_{v_J}(n)$.
Since $v = v_I$ by condition (V1) and $I \leq J$, we deduce that $v_J$ is a refinement of $v$.
Thus, $D_{v_J}(n) \leq D$; since $D \leq D_{v_J}(n)$ as well, we deduce that $D = D_{v_J}(n)$.
By condition (V2), this implies that $I = I_{v_J}(n)$.
Since $I \leq J \leq I_{v_J}(n)$, we deduce that $I = J$.
Thus condition (2) holds.

Now for condition (1).
By Lemma \ref{lem:I-and-D-give-c-pairs}, we have $I_v(N) \leq \Ibc(D_v(N)) \leq D_v(N)$ and by Lemma \ref{lem:image-under-mod-N-to-mod-n} we obtain:
\[ I =  I_v(n) \leq (\Ibc(D_v(N)))_n \leq D_v(n) = D. \]
We obtain condition (1) by using condition (2) with $E' = D_v(N)$ and $E = D$.

Lastly, we must show condition (3), that $D$ is not a C-group.
Assume for a contradiction that $D$ is a C-group; equivalently, $\Gc_{k(v)}^a(n)$ is a C-group by Lemma \ref{lem:C-pair-compat-in-res-fields}.
However, $\Gc_{k(v)}^a(n)$ is non-cyclic by (V3) and thus $\Gc_{k(v)}^a(1)$ is non-cyclic by Lemma \ref{lem:image-under-mod-N-to-mod-n}.
But, $\Gc_{k(v)}^a(n)$ being a C-group implies that $\Gc_{k(v)}^a(1)$ is a C-group as well.
Thus, applying Theorem \ref{thm:c-groups-to-valuative} with $n = 1$, there exists a valuative subgroup $J \leq \Gc_{k(v)}^a(1)$ such that, denoting $w' = v_J$, one has $\Gc_{k(v)}^a(1) = D_{w'}(1)$ and $D_{w'}(1)/I_{w'}(1)$ is cyclic.
By Lemma \ref{lem:image-under-mod-N-to-mod-n}, this implies that $D_{w'}(n) = \Gc_{k(v)}^a(n)$ and $D_{w'}(n)/I_{w'}(n)$ is cyclic as well.

Consider $w := v \circ w'$.
One has $I_v(n) \leq I_w(n) \leq D_w(n) = D_v(n)$, and $D_w(n)/I_w(n)$ is cyclic.
Since $D_v(n)/I_v(n) = \Gc_{k(v)}^a(n)$ is non-cyclic by condition (V3), we see that $I_v(n) \neq I_w(n)$; this contradicts condition (V2).
Having obtained our contradiction, we deduce that $D$ is not a C-group.
\end{proof}

\begin{remark}
\label{remark:main-detect-thm-mu}
Let $n \in \Nbar$ be given and let $N = \Nfr(\Mfr_2(\Mfr_1(n)))$.
Let $K$ be a field in which the polynomial $X^{2\ell^{N}}-1$ splits completely.
It follows from Lemma \ref{lem:v_H} that the comparability of two valuations $v,w$ in $\Vc_{K,n}$ is captured by the comparability of $I_v(n)$ and $I_w(n)$.
More precisely, if $v,w$ are two arbitrary valuations of $K$ which satisfy condition (V1), then Lemma \ref{lem:v_H} implies the following: $v \leq w$ if and only if $I_v(n) \leq I_w(n)$.
In particular, Theorem \ref{thm:maximal-I-D-mu} implies that the map $v \mapsto (I_v(n) \leq D_v(n))$ defines a bijection $\Vc_{K,n} \rightarrow \Dc_{K,n}$; the inverse, $\Dc_{K,n} \rightarrow \Vc_{K,n}$ is given by $(I \leq D) \mapsto v_I$ (note Theorem \ref{thm:maximal-I-D-mu} implies that this $I$ is valuative and thus $v_I$ makes sense).
Thus, the partially ordered structure of $\Vc_{K,n}$ can be recovered using the following data: (1) the canonical map of pro-$\ell$ groups $\Gc_K^a(N) \rightarrow \Gc_K^a(n)$, and (2) the collections of C-pairs in $\Gc_K^a(N)$ and $\Gc_K^a(n)$.

The bijection $\Vc_{K,n} \rightarrow \Dc_{K,n}$ is also compatible with passing to residue fields, as follows.
Let $v \in \Vc_{K,n}$ be given.
It follows from Lemma \ref{lem:C-pair-compat-in-res-fields} that the following four bijections are compatible in the obvious sense:
\begin{enumerate}
\item $\Vc_{k(v),n} \rightarrow \Dc_{k(v),n}$, arising from Theorem \ref{thm:maximal-I-D-mu} applied to $k(v)$.
\item $\Vc_{k(v),n} \rightarrow \Vc_{v,n}$, defined by $w \mapsto  w\circ v$.
\item $\Vc_{v,n} \rightarrow \Dc_{v,n}$, defined by restricting $\Vc_{K,n} \rightarrow \Dc_{K,n}$ to the subset $\Vc_{v,n} \subset \Vc_{K,n}$.
\item $\Dc_{v,n} \rightarrow \Dc_{k(v),n}$, defined by $(I \leq D) \mapsto (I/I_v(n) \leq D/I_v(n))$.
\end{enumerate}
\end{remark}

We conclude this subsection by providing an alternative definition of $\Vc_{K,n}$ which is much more concise than Definition \ref{defn:Dc-Vc}, although perhaps less intuitive.
\begin{lem}
\label{lem:alt-defn-of-Vc}
Let $n \in \Nbar$ be given and let $K$ be a field in which $X^{2\ell^n}-1$ splits completely.
Then $\Vc_{K,n}$ is precisely the collection of valuations $v$ of $K$ such that:
\begin{enumerate}
\item The value group $\Gamma_v$ contains no non-trivial $\ell$-divisible convex subgroups.
\item One has $I_v(1) = \Ibc(D_v(1)) \neq D_v(1)$.
\end{enumerate}
In particular, $\Vc_{K,n} = \Vc_{K,m}$ for all $m \leq n$.
\end{lem}
\begin{proof}
The argument of this lemma is similar to that of Theorem \ref{thm:maximal-I-D-mu}.
Denote by $\Vc$ the collection of valuations satisfying the two conditions (1),(2) above.

First, let us show that $\Vc \subset \Vc_{K,n}$.
Let $v \in \Vc$ be given; we need show that $v$ satisfies conditions (V1), (V2), and (V3).
Condition (1) for $v \in \Vc$ is precisely condition (V1).
As $\Ibc(D_v(1)) \neq D_v(1)$, we see that $\Gc_{k(v)}^a(n) = D_v(n)/I_v(n)$ is non-cyclic, since $I_v(n) \leq \Ibc(D_v(n))$ by Lemma \ref{lem:I-and-D-give-c-pairs}; thus condition (V3) holds true.

Suppose that $w$ is a refinement of $v$ such that $D_w(n) = D_v(n)$.
Consider the inclusion of subgroups $I_v(1) \leq I_w(1) \leq D_w(1) \leq D_v(1)$.
By Lemma \ref{lem:I-and-D-give-c-pairs} and condition (2), we see that:
\[ \Ibc(D_v(1)) = I_v(1) \leq I_w(1) \leq \Ibc(D_v(1)) \leq D_w(1) = D_v(1). \]
Thus, $I_w(1) = I_v(1)$.
By Lemma \ref{lem:image-under-mod-N-to-mod-n}, we see that $I_w(n) = I_v(n)$ as well; thus condition (V2) holds true.

Conversely we show that $\Vc_{K,n} \subset \Vc$.
Let $v \in \Vc_{K,n}$ be given.
Then condition (1) of the lemma holds trivially for $v$ by (V1).

We must show that $I_v(1) = \Ibc(D_v(1)) \neq D_v(1)$.
Clearly, $I_v(1) \leq \Ibc(D_v(1))$ by Lemma \ref{lem:I-and-D-give-c-pairs}.
Let $I := \Ibc(D_v(1))$.
By Theorem \ref{thm:main-c-groups}, $I$ is valuative and, denoting $w := v_I$, one has $D_v(1) \leq D_w(1)$.
Condition (V1) and Lemma \ref{lem:v_H} show that $w$ is a refinement of $v$.
Therefore, $D_w(1) \leq D_v(1)$.
Since $D_v(1) \leq D_w(1)$ also, we deduce that $D_w(1) = D_v(1)$ and thus $D_w(n) = D_v(n)$ by Lemma \ref{lem:image-under-mod-N-to-mod-n}.
By condition (V2), we have $I_w(n) = I_v(n)$ and thus $I_w(1) = I_v(1)$.
Since $I \leq I_w(1)$ and $I_v(1) \leq I$, we deduce that $I=I_v(1)$.

Lastly, (V3) says that $\Gc_{k(v)}^a(n)$ is non-cyclic and thus $\Gc_{k(v)}^a(1)$ is non-cyclic by Lemma \ref{lem:image-under-mod-N-to-mod-n}.
In particular, $D_v(1)/I$ cannot be cyclic.
This proves that $v$ satisfies condition (2) of the lemma, as required.
\end{proof}

\subsection{The $n = 1$ or $n = \infty$ Case}
\label{sec:n-=-infty}

Throughout this subsection, $n$ will denote either $1$ or $\infty$.
The key property to notice in these cases is that $R_n$ is a domain and that $\Nfr(n) = \Mfr_r(n) = n$.
In fact, $1$ and $\infty$ are the only fixed points of $\Nfr$ and of $\Mfr_r$.
The proofs of the results below are virtually identical (and in fact much easier since $n = N$) to those in \S\ref{sec:suff-many-roots}, using this observation.
Indeed, in \S\ref{sec:suff-many-roots}, the added assumption that $X^{2\ell^N}-1$ splits in $K$ was only used to ensure that the maps $D_v(N) \rightarrow D_v(n)$ and $\Gc_K^a(N) \rightarrow \Gc_K^a(n)$ are surjective.
In this case, $N=n$ so that these are trivially satisfied.
We therefore omit the proofs in this subsection.

\begin{prop}
\label{prop:maximal-C-group-D-infty}
Let $n = 1$ or $n =\infty$ and let $K$ be an arbitrary field.
Let $D \leq \Gc_K^a(n)$ be given.
Then the following are equivalent:
\begin{enumerate}
\item There exists a valuation $v$ of $K$ such that $D \leq D_v(n)$ and $D/(D \cap I_v(n))$ is cyclic.
\item The subgroup $D$ is a C-group.
\end{enumerate}
\end{prop}

\begin{prop}
\label{prop:center-infty}
Let $n = 1$ or $n =\infty$ and let $K$ be an arbitrary field.
Assume that $\Ibc(\Gc_K^a(n)) \neq \Gc_K^a(n)$ and consider $I := \Ibc(\Gc_K^a(n))$.
Then the following hold:
\begin{enumerate}
\item The subgroup $I$ is valuative and $v := v_I \in \Vc_{K,n}$.
\item One has $I = I_v(n)$ and $D_v(n) = \Gc_K^a(n)$.
\end{enumerate}
\end{prop}

\begin{thm}
\label{thm:maximal-I-D-infty}
Let $n =1$ or $n = \infty$.
Let $K$ be an arbitrary field and let $I \leq D \leq \Gc_K^a(n)$ be given.
Then there exists a valuation $v \in \Vc_{K,n}$ such that $I = I_v(n)$ and $D = D_v(n)$ if and only if the following hold:
\begin{enumerate}
\item One has $I = \Ibc(D)$.
\item The subgroups $I \leq D \leq \Gc_K^a(n)$ are maximal with property (1). Namely, if $D \leq E \leq \Gc_K^a(n)$ and $I \leq \Ibc(E)$, then $D = E$ and $I = \Ibc(E)$.
\item One has $\Ibc(D) \neq D$; i.e. $D$ is not a C-group.
\end{enumerate}
\end{thm}

\begin{remark}
\label{remark:main-detect-thm-infty}
Suppose that $K$ is an arbitrary field and $n = 1$ or $n = \infty$.
Similarly to Remark \ref{remark:main-detect-thm-mu}, the map $v \mapsto (I_v(n) \leq D_v(n))$ defines a bijection $\Vc_{K,n} \rightarrow \Dc_{K,n}$.
This bijection respects the ordered structure of $\Vc_{K,n}$ via the fact: $v \leq w$ if and only if $I_v(n) \leq I_w(n)$.
Furthermore, this bijection is compatible with passing to residue fields of a valuation, as discussed in Remark \ref{remark:main-detect-thm-mu}.
\end{remark}

\section{Restricting the Characteristic}
\label{sec:restr-char}

In this section we use C-pairs to force certain valuations to have residue characteristic $\neq \ell$.
We then prove three theorems, which are analogous to Theorems \ref{thm:main-c-pairs-thm}, \ref{thm:c-groups-to-valuative} and \ref{thm:main-c-groups}, which restrict the residue characteristics of valuations to be $\neq \ell$.

Throughout this section we work with a fixed $n \in \Nbar$.
Let $L|K$ be an extension of fields.
We recall that the canonical restriction map $\Gc^a_L(n) \rightarrow \Gc_K^a(n)$ is denoted by $f \mapsto f_K$.

For a subgroup $H \leq K^\times$, we write $K_H := K(\sqrt[\ell^n]{H})$.
As usual, if $n = \infty$ this notation stands for $K(\sqrt[\ell^\infty]{H}):= \bigcup_{m \in \Nb} K(\sqrt[\ell^m]{H})$.
Similarly, for a subgroup $A \leq \Gc_K^a(n)$ we write $K_A := K_{A^\perp}$.

\begin{lem}
\label{lem:pops-lemma}
Let $n \in \Nbar$ be given.
Let $(K,v)$ be a valued field such that $\Char K \neq \ell$.
Let $L$ be an extension of $K$ such that $1+\mf_v \subset L^{\times\ell^n}$, and let $w$ be a chosen prolongation of $v$ to $L$.
Let $\Delta$ denote the (possibly trivial) convex subgroup of $\Gamma_v$ which is generated by $v(\ell)$.
Then $\Delta \leq \ell^n \cdot \Gamma_w$.
\end{lem}
\begin{proof}
We can assume with no loss that $n \in \Nb$ as the $n=\infty$ case follows from this immediately.
If $\Char k(v) \neq \ell$ then $v(\ell) = 0$ and $\Delta$ is trivial so the lemma is trivially true.

Therefore, we may further assume that $\Char k(v) = \ell$.
Let $x \in K^\times$ be such that $0 < v(x) \leq v(\ell)$.
It suffices to prove that $w(x) \in \ell^n\cdot \Gamma_w$.

Since $v(x) > 0$, we see that $1+x \in L^{\times \ell^n}$.
Thus, there exists $y \in L$ such that $1+x = (1+y)^{\ell^n}$.
This forces $y \in \Oc_w$ and, since $1+x = (1+y)^{\ell^n} \in (1+y^{\ell^n})+ \mf_w$, we deduce that $y \in \mf_w$.

Expanding the equation $1+x = (1+y)^{\ell^n}$ using the binomial theorem, we see that $x = \ell \cdot y \cdot \epsilon + y^{\ell^n}$ for some $\epsilon \in \Oc_w$.
But $w(x) \leq w(\ell) < w(\ell \cdot y \cdot \epsilon)$ since $w(y) > 0$ and $w(\epsilon) \geq 0$.
Thus, $w(x) = w(y^{\ell^n}) = \ell^n \cdot w(y)$ by the ultrametric inequality.
\end{proof}

\begin{prop}
\label{prop:main-char-prop}
Let $n \in \Nbar$ be given.
Let $K$ be a field such that $\Char K \neq \ell$.
Suppose that $I \leq \Gc_K^a(n)$ and $D \leq \Gc_K^a(n)$ are given. 
Consider $L := K_D$, and assume that there exists a subgroup $I' \leq \Gc^a_L(n)$ such that $I'$ is valuative and $I'_K = I$.
Let $w' := v_{I'}$ and $w := w'|_K$, and assume that $D \leq D_{w}(n)$.
Then $I$ is valuative, $D \leq D_{v_I}(n)$ and $\Char k(v_I) \neq \ell$.
\end{prop}
\begin{proof}
First, as $I'$ is valuative and $I = I'_K$, we see that $I \leq I_w(n)$.
Thus $I$ is indeed valuative.
Since it will be used multiple times in this proof, we recall that $v := v_I$ is the coarsening of $w$ associated to the maximal convex subgroup of $w(I^\perp)$ (Lemma \ref{lem:v_H}).

As $D \leq D_w(n)$ and $v$ is a coarsening of $w$, we see that $D \leq D_v(n)$ as well.
Thus it remains to show that $\Char k(v_I) \neq \ell$.

Since $D \leq D_w(n)$ we note that $1+\mf_w \subset L^{\times\ell^n}$.
With Lemma \ref{lem:pops-lemma} in mind, consider $\Delta$ the convex subgroup of $\Gamma_w$ generated by $w(\ell)$.

Assume first that $n \in \Nb$.
We consider the following canonical \emph{injective} map induced by taking the $R_n$-dual of the surjective map $I' \twoheadrightarrow I$:
\[ \Gamma_w/w(I^\perp) \hookrightarrow \Gamma_{w'}/w'((I')^\perp). \]
By Lemma \ref{lem:pops-lemma}, we deduce that $\Delta \leq \ell^n \cdot \Gamma_{w'} \leq w'((I')^\perp)$.
The injectivity of the map above implies that $\Delta \leq w(I^\perp)$.
Therefore, $\Delta$ is contained in the kernel of the canonical projection $\Gamma_w \rightarrow \Gamma_v$.
In particular, $v(\ell) = 0$, so that $\Char k(v) \neq \ell$.

Assume now that $n = \infty$.
In this case, the $\Z_\ell$-dual of the surjective map $I' \twoheadrightarrow I$ is the \emph{injective} map of $\Z_\ell$-modules:
\[ \widehat \Gamma_w / \widehat w(I^\perp) \hookrightarrow \widehat \Gamma_{w'}/\widehat w'((I')^\perp). \]
By Lemma \ref{lem:pops-lemma}, the image of $\Delta$ in $\widehat \Gamma_w / \widehat w(I^\perp)$ is contained in the kernel of this map; this image is therefore trivial.
Namely, the image of $\Delta$, under the $\ell$-adic completion map $\Gamma_w \rightarrow \widehat \Gamma_w$, is contained in $\widehat w(I^\perp)$.
Since the kernel of $\Gamma_w \rightarrow \widehat \Gamma_w$ is $\ell^\infty \cdot \Gamma_w$, and $\ell^\infty \cdot \Gamma_w \leq w(I_w(\infty)^\perp)$, we see that $\Delta$ is contained in the kernel of $\Gamma_w \twoheadrightarrow \Gamma_v$.
Thus $v(\ell) = 0$ and $\Char k(v) \neq \ell$.
\end{proof}

\subsection{Detecting valuations with residue characteristic $\neq \ell$}
\label{sec:detect-val-with-res-char-not-ell}

We now prove three theorems which are analogous to the main results of \S\ref{sec:comp-valu}, while ensuring that all valuations in sight have residue characteristic $\neq \ell$.

\begin{thm}
\label{thm:main-c-pairs-thm-char}
Let $n \in \Nbar$ be given and let $N := \Nfr(n)$.
Let $K$ be a field such that $\Char K \neq \ell$.
Let $f,g \in \Gc_K^a(n)$ be given, let $H := \ker f \cap \ker g$, and consider $L := K_H$.
Assume that there exist $f'',g'' \in \Gc_L^a(N)$ such that $f'',g''$ form a C-pair, $(f''_n)_K = f$, and $(g''_n)_K = g$.
Then there exists a valuation $v$ of $K$ such that
\begin{enumerate}
\item One has $f,g \in D_v(n)$.
\item The quotient $\langle f,g \rangle/(\langle f,g \rangle \cap I_v(n))$ is cyclic (possibly trivial).
\item One has $\Char k(v) \neq \ell$.
\end{enumerate}
\end{thm}
\begin{proof}
Consider $f' := f''_n$ and $g' := g''_n$, both are elements of $\Gc_L^a(n)$.
By Theorem \ref{thm:main-c-pairs-thm}, there exists a valuation $w'$ of $L$ such that $f',g' \in D_{w'}(n)$, and $\langle f',g' \rangle/(\langle f',g' \rangle \cap I_{w'}(n))$ is cyclic.
Consider $w := w'|_K$ the restriction of $w'$ to $K$, $I := (\langle f',g' \rangle \cap I_{w'}(n))_K$ and $D := \langle f,g \rangle$; observe that $D \leq D_w(n)$.
With this set-up, the theorem follows from Proposition \ref{prop:main-char-prop}.
\end{proof}

\begin{thm}
\label{thm:c-groups-to-valuative-char}
Let $n \in \Nbar$ be given and let $N := \Nfr(\Mfr_1(n))$.
Let $K$ be a field such that $\Char K \neq \ell$.
Let $D \leq \Gc_K^a(n)$ be given, and assume that there exists a subgroup $D'' \leq \Gc^a_{K_D}(N)$ such that $D''$ is a C-group and $D = (D''_n)_K$.
Then there exists a valuative subgroup $I \leq D$ such that:
\begin{enumerate}
\item The quotient $D/I$ is cyclic.
\item One has $D \leq D_{v_I}(n)$.
\item One has $\Char k(v_I) \neq \ell$.
\end{enumerate}
\end{thm}
\begin{proof}
Let $L := K_D$ and consider $D' := D''_n \leq \Gc^a_L(n)$.
By Theorem \ref{thm:c-groups-to-valuative}, there exists a valuative subgroup $I' \leq D'$ such that, denoting $w' := v_{I'}$, one has $D' \leq D_{w'}(n)$ and $D'/I'$ is cyclic.
Let $I := I'_K$.
Since $D'/I'$ is cyclic and $D'_K = D$, we see that $D/I$ is cyclic as well.
Moreover, we observe that $D \leq D_w(n)$ where $w = w'|_K$ is the restriction of $w'$ to $K$.
With this set-up, the theorem follows from Proposition \ref{prop:main-char-prop}.
\end{proof}

\begin{thm}
\label{thm:main-c-groups-char}
Let $n \in \Nbar$ be given and let $N := \Nfr(\Mfr_2(\Mfr_1(n)))$.
Let $K$ be a field such that $\Char K \neq \ell$.
Let $I \leq D \leq \Gc_K^a(n)$ be given and consider $L := K_D$.
Assume that there exists $I'' \leq D'' \leq \Gc^a_L(N)$ such that $I'' \leq \Ibc(D'')$, $(I''_n)_K = I$, and $(D''_n)_K = D$.
Assume also that $D \neq \Ibc(D)$.
Then $I$ is valuative, $D \leq D_{v_I}(n)$ and $\Char k(v_I) \neq \ell$.
\end{thm}
\begin{proof}
This theorem follows from Proposition \ref{prop:main-char-prop} and Theorem \ref{thm:main-c-groups}, similarly to the way in which Theorem \ref{thm:c-groups-to-valuative-char} follows from Proposition \ref{prop:main-char-prop} and Theorem \ref{thm:c-groups-to-valuative}.
\end{proof}

\begin{remark}
\label{rem:char-neq-ell-remark-for-theorems}
Using Theorem \ref{thm:main-c-pairs-thm-char} resp. \ref{thm:c-groups-to-valuative-char} resp. \ref{thm:main-c-groups-char} instead of Theorem \ref{thm:main-c-pairs-thm} resp. \ref{thm:c-groups-to-valuative} resp. \ref{thm:main-c-groups}, one can prove results analogous to those in \S \ref{sec:detecting-d_vn-i_vn} while considering only valuations whose residue characteristic is different from $\ell$.
We will not state these results explicitly, as their Galois-theoretical analogues comprise Theorem \ref{thm:main-thm-char-intro}.
\end{remark}

\part{Milnor K-theory and Galois Theory}
\label{part:galois-theory-milnor}

\section{Milnor K-Theory and C-Pairs}
\label{sec:milnor-k-theory}

Let $M$ be an $R_n$-module.
A collection of non-zero elements $(f_i)_i$, $f_i \in M$ will be called {\bf quasi-independent} provided the following condition holds: if $a_i \in R_n$ are given with all but finitely many $a_i = 0$ such that $\sum_i a_i f_i = 0$, then $a_i f_i = 0$ for all $i$.
A generating set which is quasi-independent will be called a {\bf quasi-basis}.
Observe that any finitely generated $R_n$ module $M$ has a quasi-basis of unique finite size equal to $\dim_{\Z/\ell}(M/\ell)$.
Namely, any finitely generated $R_n$-module $M$ can be written as a direct product of cyclic submodules $M = \langle \sigma_1 \rangle \times \cdots\times \langle \sigma_k \rangle$; in this case $(\sigma_i)_{i=1}^k$ forms a quasi-basis for $M$.

Let $K$ be any field.
The usual construction of the Milnor K-ring goes as follows:
\[ K_n^M(K) := \frac{(K^\times)^{\otimes n}}{\langle a_1 \otimes \cdots \otimes a_n \ : \ \exists  \ 1 \leq i < j \leq n, \ a_i + a_j = 1\rangle}. \]
The tensor product makes $K_*^M(K) := \bigoplus_n K_n^M(K)$ into a graded-commutative ring and we denote by $\{\bullet,\bullet\}$ the product $K_1^M(K) \times K_1^M(K) \rightarrow K_2^M(K)$.

More generally, let $T \leq K^\times$ be given.
We abuse the notation and write $K_*^M(K)/T$ for the quotient of $K_*^M(K)$ by the graded ideal generated by $T \leq K^\times = K_1^M(K)$.
This is again a graded ring whose graded terms can be defined individually as follows:
\[ K_n^M(K)/T := \frac{(K^\times/T)^{\otimes n}}{\langle a_1 \cdot T \otimes \cdots \otimes a_n \cdot T \ : \ \exists \  1 \leq i < j \leq n, \  1 \in a_i \cdot T + a_j \cdot T \rangle}. \]
Again, the tensor product makes $K_*^M(K)/T = \bigoplus_n K_n^M(K)/T$ into a graded-commutative ring and we denote by $\{\bullet,\bullet\}_T$ the product in this ring.

Clearly, one has a surjective map of graded-commutative rings $K_*^M(K) \twoheadrightarrow K_*^M(K)/T$, and this restricts to a surjective homomorphism $K_n^M(K) \twoheadrightarrow K_n^M(K)/T$ for all $n$.
We recall that, for all $x \in K^\times$, one has $\{x,-1\} = \{x,x\} \in  K_2^M(K)$.
Thus the same is true in $K_2^M(K)/T$: for all $x \in K^\times$, one has $\{x,-1\}_T = \{x,x\}_T$.
For more on the arithmetical properties of these canonical quotients of the Milnor K-ring, refer to Efrat \cite{Efrat2006}, \cite{Efrat2007} where they are systematically studied.

Suppose that $T \leq K^\times$ and $-1 \in T$.
Then the canonical map $(K^\times/T) \otimes (K^\times/T) \twoheadrightarrow K_2^M(K)/T$ factors through
\[ \wedge^2(K/T) := \frac{(K^\times/T) \otimes (K^\times/T)}{\langle x \otimes x \ : \ x \in K^\times/T \rangle}. \]
Moreover, the kernel of the canonical surjective map $\wedge^2(K^\times/T) \rightarrow K_2^M(K)/T$ is generated by $z \wedge (1-z)$ as $z$ varies over the elements of $K^\times \smallsetminus \{0,1\}$.

Suppose that $n \in \Nb$, and $T$ satisfies $\pm \Kln \leq T \leq K^\times$.
Assume furthermore that $K^\times/T$ has rank $2$ as an $R_n$-module.
Choose $x,y \in K^\times$ which induce a quasi-basis of $K^\times/T$.
Namely, for some $0 \leq a,b < n$, one has
\[K^\times/T = x^{\Z/\ell^{n-a}} \times y^{\Z/\ell^{n-b}} \cong \Z/\ell^{n-a} \times \Z/\ell^{n-b}. \]
Clearly, $\wedge^2(K^\times/T)$ is cyclic generated by $x \wedge y$, with order $\ell^{n-\max(a,b)}$.
Thus $K_2^M(K)/T = \langle \{x,y\}_T \rangle$ is cyclic of order $\ell^{n-c}$ for some $c$ with $\max(a,b) \leq c \leq n$.
This observation will help prove the following K-theoretic characterization of C-pairs.

\begin{prop}[K-theoretic characterization of C-pairs]
\label{prop:k-thy-c-pair}
Let $n \in \Nb$ be given.
Let $f,g \in \Gc_K^a(n)$ be given quasi-independent elements of order $\ell^{n-a}$ and $\ell^{n-b}$ respectively.
In particular, 
\[ \langle f,g \rangle = \langle f \rangle \oplus \langle g \rangle \cong (\Z/\ell^{n-a}) \cdot f \oplus (\Z/\ell^{n-b}) \cdot g. \]
Let $T := \ker f \cap \ker g$, and let $0 \leq c \leq n$ be such that $K_2^M(K)/T$ has order $\ell^{n-c}$.
Then $f,g$ form a C-pair if and only if $c \leq a+b$.
\end{prop}
\begin{proof}
The following elegant proof was graciously suggested by the referee.
Consider the map $f \wedge g : \wedge^2(K^\times/T) \rightarrow R_n$ defined, in the usual way, as:
\[ (f \wedge g)(z\wedge w) = f(z)\cdot g(w)-f(w)\cdot g(z). \]
Merely by the definitions of ``C-pair'' and $K_2^M(K)/T$, it follows that $f,g$ form a C-pair if and only if $f \wedge g$ factors through $K_2^M(K)/T$.

Since $f,g$ are quasi-independent, we see that $K^\times/T$ has quasi-independent generators, say  $x,y$, which are dual to $f,g$.
Namely, $(f,g)(x) = (\ell^a,0)$ and $(f,g)(y) = (0,\ell^b)$.
This implies that $\wedge^2(K^\times/T)$ is generated by $x \wedge y$ and thus $K_2^M(K)/T$ is generated by $\{x,y\}_T$.

Since $K_2^M(K)/T$ has order $\ell^{n-c}$, we deduce that $f \wedge g$ factors through $K_2^M(K)/T$ if and only if 
\[ 0 = (f\wedge g)(\ell^{n-c} \cdot x \wedge y) = \ell^{n-c} \cdot \ell^a \cdot \ell^b = \ell^{n+a+b-c} \]
as an element of $\Z/\ell^n$.
We deduce that $f,g$ form a C-pair if and only if $0 = \ell^{n+a+b-c}$ as an element of $\Z/\ell^n$.
Thus, $f,g$ form a C-pair if and only if $c \leq a+b$.
\end{proof}

\begin{remark}
\label{remark:rigidity-k-thy}
Let $n \in \Nb$ be given.
Let $A \leq \Gc_K^a(n)$ be given and let $T := A^\perp$.
Proposition \ref{prop:k-thy-c-pair} gives a precise recipe to decide whether or not $A$ is a C-group using the structure of $K_*^M(K)/T$.
More precisely, Proposition \ref{prop:k-thy-c-pair} immediately implies that the following conditions are equivalent:
\begin{enumerate}
\item $A$ is a C-group.
\item For all subgroups $A_0 \leq A$ of rank 2, $A_0$ is a C-group.
\item For all subgroups $T_0 \leq K^\times$ such that $T \leq T_0 \leq K^\times$ and $K^\times/T_0$ has rank 2, $K_*^M(K)/T_0$ satisfies the equivalent conditions of Proposition \ref{prop:k-thy-c-pair}.
\end{enumerate}

In the case where $n=1$, we can provide a \emph{direct} characterization of C-groups $A \leq \Gc_K^a(1)$ using $K_*^M(K)/T$ where $T = A^\perp$. 
Namely, we want a criterion which doesn't require the auxiliary subgroups $T_0$ as above.
In the notation above ($n = 1$ and $T = A^\perp$), the following are equivalent:
\begin{enumerate}
\item $A$ is a C-group.
\item For all subgroups $T_0$ with $T \leq T_0 \leq K^\times$ such that $K^\times/T_0$ has rank 2, one has $K_2^M(K)/T_0 \neq 1$.
\item For all $x,y \in K^\times$ such that $x \cdot T, \ y \cdot T$ are $\Z/\ell$ independent in $K^\times/T$, one has $\{x,y\}_T \neq 0$ as an element of $K_2^M(K)/T$.
\item For all $x \in K^\times \smallsetminus T$ the group $\langle (1-x) \cdot T, \ x \cdot T \rangle$ is a cyclic subgroup of $K^\times/T$.
\item For all subgroups $H$ with  $T \leq H \leq K^\times$, and $x \in K^\times \smallsetminus H$, the group $\langle (1-x) \cdot H, \ x \cdot H \rangle$ is a cyclic subgroup of $K^\times/H$.
\item The canonical map $\wedge^2(K^\times/T) \rightarrow K_2^M(K)/T$ is an isomorphism.
\item For all subgroups $H$ with $T \leq H \leq K^\times$, the canonical map $\wedge^2(K^\times/H) \rightarrow K_2^M(K)/H$ is an isomorphism.
\end{enumerate}
Indeed, $(1) \Leftrightarrow (2)$ is Proposition \ref{prop:k-thy-c-pair}, while $(2) \Rightarrow (3) \Rightarrow (4) \Rightarrow (5) \Rightarrow (6) \Rightarrow (3)$ and $(5) \Rightarrow (7) \Rightarrow (2)$ follow immediately from the definitions.
In particular, the equivalence of conditions (1) and (6) above yield the desired direct characterization of C-groups $A \leq \Gc_K^a(1)$ based on the structure of $K_*^M(K)/T$ for $T = A^\perp$.
\end{remark}

\begin{remark}
\label{rem:k-thy-to-c-pairs}
We can obtain a K-theoretic criterion which detects C-pairs in $\Gc_K^a(\infty)$ which is similar to Proposition \ref{prop:k-thy-c-pair}, by passing to the limit over all $n \in \Nb$.
First, we note that a pair of elements $f,g \in \Gc_K^a(\infty)$ form a C-pair if and only if, for all $n \in \Nb$, the induced pair $f_n,g_n \in \Gc_K^a(n)$ is a C-pair.

Denote by $\widehat K_i^M(K)$ the $\ell$-adic completion of $K_i^M(K)$.
By the universal property of $\ell$-adic completions one has:
\[ \Gc_K^a(\infty) = \Hom(\widehat{K^\times} / \pm 1, \Z_\ell) = \Hom(\widehat{K^\times}/{\rm torsion},\Z_\ell). \]
Moreover, $\widehat{K^\times}/ {\rm torsion}$ is in perfect $\Z_\ell$-duality with $\Gc_K^a(\infty)$.

Let $f,g \in \Gc_K^a(\infty)$ be given.
As above, we may consider $f,g$ as homomorphisms $\widehat{K^\times} \rightarrow \Z_\ell$.
We assume that $\langle f,g \rangle$ is non-cyclic, for otherwise $f,g$ trivially form a C-pair.
Furthermore, if $f = \ell^a \cdot f'$ and $g = \ell^b \cdot g'$, then $f,g$ form a C-pair if and only if $f',g'$ form a C-pair.
Thus, we may assume that (1) $\Gc_K^a(\infty)/\langle f,g \rangle$ is torsion-free and (2) that $f,g$ are $\Z_\ell$-independent.
These assumptions imply that $\langle f_1, g_1\rangle$ is a non-cyclic subgroup of $\Gc_K^a(1)$.
Therefore, for all $n \in \Nb$, $f_n,g_n$ are quasi-independent elements of $\Gc_K^a(n)$, both of order $\Z/\ell^n$.

As $f,g$ are continuous homomorphisms $\widehat{K^\times} \rightarrow \Z_\ell$, we consider $T = \ker f \cap \ker g$ as a closed $\Z_\ell$-submodule of $\widehat{K^\times}$.
Thus $\widehat{K^\times} / T \cong \Z_\ell \times \Z_\ell$ has independent generators $x,y$ which are dual to $f,g$; namely, $(f,g)(x) = (1,0)$ and $(f,g)(y) = (0,1)$.

Consider $T_n := \ker f_n \cap \ker g_n$ as a subgroup of $K^\times$ which contains $\pm K^{\times \ell^n}$; observe that $\widehat{K^\times} / T = \varprojlim_n K^\times/T_n$.
From this we see that $\widehat K_2^M(K)/T := \varprojlim_n K_2^M(K)/T_n$ is a cyclic $\Z_\ell$-module which is generated by $\{x,y\}_T$.

By Proposition \ref{prop:k-thy-c-pair}, we see that $f,g$ form a C-pair if and only if for all $n \in \Nb$, the group $K_2^M(K)/T_n$ has order $\ell^n$.
Since $K_2^M(K)/T_n$ is generated by $\{x,y\}_{T_n}$, we see that $f,g$ form a C-pair if and only if $\{x,y\}_T$ (which is a generator of $\widehat K_2^M(K)/T$) has infinite order.
In other words, $f,g$ form a C-pair if and only if the canonical map $\widehat \wedge^2(\widehat{K^\times}/T) \rightarrow \widehat K_2^M(K)/T$ is an isomorphism.
Compare this with condition (6) of Remark \ref{remark:rigidity-k-thy}.
\end{remark}

We conclude this section about Milnor K-theory with the observation that Proposition \ref{prop:k-thy-c-pair} and the main results of Part \ref{part:underlying-theory} show how to detect valuations of a field using its first and second Milnor K-theory groups, along with the product map.
Below we restrict our attention to $n \in \Nb$ but similar statements can be made for $n = \infty$ using Remark \ref{rem:k-thy-to-c-pairs}.

Let $K$ be an arbitrary field. 
If $\ell \neq 2$, one has $\Gc_K^a(n) = \Hom(K_1^M(K)/\ell^n,R_n)$; if $\ell = 2$, $\Gc_K^a(n)$ is the image of the canonical map 
\[ \Hom(K_1^M(K)/\ell^{n+1},R_{n+1}) \rightarrow \Hom(K_1^M(K)/\ell^{n+1},R_n) = \Hom(K_1^M(K)/\ell^n,R_n). \]
Moreover, Proposition \ref{prop:k-thy-c-pair} shows that the product $K_1^M(K)/\ell^n \times K_1^M(K)/\ell^n \rightarrow K_2^M(K)/\ell^n$ can be used to determine the collection of C-pairs in $\Gc_K^a(n)$.

For example, Theorem \ref{thm:maximal-I-D-mu} and the observation above show that, if $K$ has sufficiently many roots of unity and $N$ is sufficiently large (e.g. $N = \Nfr(\Mfr_2(\Mfr_1(n)))+1$ is sufficient), then one can construct $\Gc_K^a(n)$ along with the subgroups $I_v(n) \leq D_v(n) \leq \Gc_K^a(n)$ for all $v \in \Vc_{K,n}$, using the following data:
\begin{enumerate}
\item The groups $K_*^M(K)/\ell^N$ for $* = 1,2$.
\item The product map $K_1^M(K)/\ell^N \times K_1^M(K)/\ell^N \rightarrow K_2^M(K)/\ell^N$.
\end{enumerate}
Lastly, we note that the inclusion $I_v(n) \hookrightarrow \Gc_K^a(n)$ is dual to the map $K^\times/\ell^n \rightarrow \Gamma_v/\ell^n$ induced by $v$.
To summarize, using the K-theoretic data mentioned above, one can reconstruct the maps $K^\times/\ell^n \rightarrow \Gamma_v/\ell^n$ for all the valuations $v \in \Vc_{K,n}$.

\section{CL Subgroups of Galois Groups}
\label{sec:cl-subgroups-galois}

Let $n \in \Nbar$ be given and suppose that $K$ is a field whose characteristic is different from $\ell$.
In this context, we denote by $R_n(i) := R_n \otimes_{\Z_\ell} \Z_\ell(i)$ the $i$-th cyclotomic twist of $R_n$.
Let us further assume that $\mu_{2\ell^n} \subset K$, and thus there is a (non-canonical) isomorphism of $G_K$-modules $R_n \cong R_n(1)$; we fix such an isomorphism for the rest of the paper.

Our assumption that $\mu_{2\ell^n} \subset K$ ensures that $-1 \in K^{\times \ell^n}$.
Thus, Kummer theory and our choice of $R_n \cong R_n(1)$, yield an isomorphism of pro-$\ell$ groups $\Gc_K^a(n) \cong \Gc_K^{a,n}$.
In this section we will use the K-theoretic criterion for C-pairs (Proposition \ref{prop:k-thy-c-pair}), along with the Merkurjev-Suslin theorem \cite{Merkurjev1982} to prove that C-pairs correspond to CL-pairs via this isomorphism $\Gc_K^a(n) \cong \Gc_K^{a,n}$.

We first prove some cohomological results concerning more general pro-$\ell$ Galois groups.
Only then will we restrict to pro-$\ell$ Galois groups which will allow us to prove the equivalence of C-pairs and CL-pairs.

\subsection{Pro-$\ell$ Groups}
\label{sec:pro-ell-groups}

Throughout this subsection, we will work with a fixed $n \in \Nbar$.
Let $\Gc$ be an arbitrary pro-$\ell$ group.
We recall that the $\ell^n$-central descending series of $\Gc$ is defined inductively as follows:
\[ \Gc^{(1,n)} = \Gc, \ \ \Gc^{(m+1,n)} = [\Gc,\Gc^{(m,n)}] \cdot (\Gc^{(m,n)})^{\ell^n}. \]
For simplicity we define $\Gc^{a,n} := \Gc/\Gc^{(2,n)}$ and $\Gc^{c,n} := \Gc/\Gc^{(3,n)}$.
We will usually use additive notation for the abelian pro-$\ell$ groups $\Gc^{a,n}$ and $\Gc^{(2,n)}/\Gc^{(3,n)}$.

Throughout, we will denote by $H^*(\Gc) := H^*_{\rm cont}(\Gc,R_n)$ the continuous-cochain cohomology of $\Gc$ with values in $R_n$.
If $n$ is finite, we recall that the short exact sequence:
\[ 1 \rightarrow \Z/\ell^n \xrightarrow{\ell^n} \Z/\ell^{2n} \rightarrow \Z/\ell^{n} \rightarrow 1 \]
yields the Bockstein homomorphism:
\[ \beta : H^1(\Gc) \rightarrow H^2(\Gc) \]
which is the connecting homomorphism in the associated long exact sequence in cohomology.
If $n = \infty$, we define $\beta : H^1(\Gc) \rightarrow H^2(\Gc)$ to be the trivial homomorphism.

The following discussion uses some well-known results concerning commutators and $\ell^n$-th powers in central descending series; see \cite{Neukirch2008} Proposition 3.8.3 for a reference.
For $\sigma,\tau \in \Gc^{a,n}$, we define $[\sigma,\tau] := \tilde\sigma^{-1}\tilde\tau^{-1}\tilde\sigma\tilde\tau$ where $\tilde\sigma,\tilde\tau \in \Gc^{c,n}$ are lifts of $\sigma,\tau$.
Since $\Gc^{c,n} \rightarrow \Gc^{a,n}$ is a central extension, the element $[\sigma,\tau]$ doesn't depend on the choice of lifts of $\sigma,\tau$.
Thus, we obtain a well-defined map:
\[ [\bullet,\bullet] : \Gc^{a,n} \times \Gc^{a,n} \rightarrow \Gc^{(2,n)}/\Gc^{(3,n)} \]
which is known to be $R_n$-bilinear.

Similarly, for $\sigma \in \Gc^{a,n}$, we define $\sigma^\pi := \tilde\sigma^{\ell^n}$ where, again, $\tilde\sigma \in \Gc_K^{c,n}$ is a lift of $\sigma$.
Since $\Gc^{c,n} \rightarrow \Gc^{a,n}$ is a central extension with kernel killed by $\ell^n$, the element $\sigma^\pi$ doesn't depend on the choice of lift.
Thus we obtain a well-defined map:
\[ (\bullet)^{\pi} : \Gc^{a,n} \rightarrow \Gc^{(2,n)}/\Gc^{(3,n)}\]
which is known to be $R_n$-linear if $\ell \neq 2$; if $\ell = 2$, this map is generally not linear.
We will write $\sigma^\beta := 2 \cdot \sigma^\pi$ and note that, unlike $(\bullet)^\pi$, the map 
\[ (\bullet)^\beta : \Gc^{a,n} \rightarrow \Gc^{(2,n)}/\Gc^{(3,n)} \]
is \emph{always} $R_n$-linear, regardless of $\ell$.

\begin{lem}
\label{lem:cup-product}
Let $\Gc$ be a pro-$\ell$ group.
Then the following holds:
\[ \ker(H^2(\Gc^{a,n}) \rightarrow H^2(\Gc)) = \ker(H^2(\Gc^{a,n}) \rightarrow H^2(\Gc^{c,n})). \]
In particular, suppose that $f,g \in \Hom(\Gc,R_n) = H^1(\Gc^{a,n}) = H^1(\Gc^{c,n}) = H^1(\Gc)$ are arbitrary.
Then the following are equivalent:
\begin{enumerate}
\item $f \cup g = 0$ as an element in $H^2(\Gc)$.
\item $f \cup g = 0$ as an element in $H^2(\Gc^{c,n})$.
\end{enumerate}
\end{lem}
\begin{proof}
Observe that the following canonical maps are isomorphisms: $H^1(\Gc^{a,n}) \rightarrow H^1(\Gc^{c,n}) \rightarrow H^1(\Gc)$.
Now, using the cohomological spectral sequence associated to the group extensions $\Gc^{c,n} \rightarrow \Gc^{a,n}$ and $\Gc \rightarrow \Gc^{a,n}$, we obtain the following commutative diagram with exact rows,
\[
\xymatrix{
0 \ar[r] & H^1(\Gc^{(2,n)}/\Gc^{(3,n)})^{\Gc^{c,n}} \ar[d]\ar[r]^-{d_2} &  H^2(\Gc^{a,n}) \ar@{=}[d] \ar[r] & H^2(\Gc^{c,n}) \ar[d]\\
0 \ar[r] & H^1(\Gc^{(2,n)})^{\Gc} \ar[r]_{d_2} &  H^2(\Gc^{a,n}) \ar[r] & H^2(\Gc)
}
\]
in which the two maps under consideration, $H^2(\Gc^{a,n}) \rightarrow H^2(\Gc)$ and $H^2(\Gc^{a,n}) \rightarrow H^2(\Gc^{c,n})$, appear.
It therefore suffices to prove that the map $H^1(\Gc^{(2,n)}/\Gc^{(3,n)})^{\Gc^{c,n}}\rightarrow H^1(\Gc^{(2,n)})^{\Gc}$ is an isomorphism.

Recall that $H^1(\Gc^{(2,n)})^{\Gc} = \Hom_{\Gc}(\Gc^{(2,n)},R_n)$ is the set of $\Gc$-equivariant continuous homomorphisms $\Gc^{(2,n)} \rightarrow R_n$, where $\Gc$ acts on $\Gc^{(2,n)}$ by conjugation and trivially on $R_n$.
From this, we see that the left kernel of the canonical pairing:
\[ \Gc^{(2,n)} \times H^1(\Gc^{(2,n)})^{\Gc} \rightarrow R_n \]
is \emph{by definition} $\Gc^{(3,n)}$.
In other words, the map $H^1(\Gc^{(2,n)}/\Gc^{(3,n)})^{\Gc^{c,n}}\rightarrow H^1(\Gc^{(2,n)})^{\Gc}$ is an isomorphism, and this completes the proof of the lemma.
\end{proof}

\begin{defn}
\label{defn:CL-pairs}
Let $\Gc$ be a pro-$\ell$ group and let $\sigma,\tau \in \Gc^{a,n}$ be given.
We say that $\sigma,\tau$ form a {\bf CL-pair} provided that:
\[ [\sigma,\tau] \in \langle \sigma^\beta,\tau^\beta \rangle. \]
If $\ell \neq 2$ we note that $\sigma,\tau$ form a CL-pair if and only if $[\sigma,\tau] \in \langle \sigma^\pi,\tau^\pi \rangle$, as $2$ is invertible in $R_n$.
Furthermore, as $(\bullet)^\beta$ is linear and $[\bullet,\bullet]$ is bilinear, if $\langle \sigma',\tau' \rangle = \langle \sigma,\tau \rangle$ and $\sigma,\tau$ form a CL-pair, then $\sigma',\tau'$ form a CL-pair as well.

A subgroup $A \leq \Gc^{a,n}$ will be called a {\bf CL-group} provided that any pair of elements $\sigma,\tau \in A$ form a CL-pair.
For a subgroup $A \leq \Gc^{a,n}$, we denote by $\Ibcl(A)$ the subset:
\[ \Ibcl(A) := \{\sigma \in A \ : \ \forall \tau \in A, \ \sigma,\tau \  \text{ form a CL-pair}. \} \]
and call $\Ibcl(A)$ the CL-center of $A$. 
In particular, $A$ is a CL-group if and only if $A = \Ibcl(A)$.
\end{defn}

\begin{remark}
\label{rem:caution-CL-pairs}
Let $\Gc$ be a pro-$\ell$ group and let $A \leq \Gc^{a,n}$ be given.
Suppose $A = \langle \sigma_i \rangle_i$ is generated by $(\sigma_i)_i$.
In a general pro-$\ell$ group $\Gc$, the fact that $(\sigma_i)_i$ are pairwise CL \emph{does not} necessarily imply that $A$ is a CL-group.
Similarly, if $A$ is an arbitrary subgroup of $\Gc^{a,n}$, then $\Ibcl(A)$ is not a subgroup of $A$ but merely a \emph{subset} in general.

For example, let $n=1$, $\ell \neq 2$ and let $S$ be the free pro-$\ell$ group on three generators $\sigma_1,\sigma_2,\sigma_3$.
Consider $\Gc = S^{c,n}/R$ where $R \leq S^{(2,n)}/S^{(3,n)}$ is the subgroup generated by the following three relations:
\begin{enumerate}
\item $[\sigma_1,\sigma_2] = \sigma_1^\pi$.
\item $[\sigma_1,\sigma_3] = \sigma_3^\pi$.
\item $[\sigma_2,\sigma_3] = \sigma_2^\pi$.
\end{enumerate}
We can consider $S^{(2,n)}/S^{(3,n)}$ as a $\Z/\ell$-vector space in the obvious way.
The following list is a basis for this vector space:
\[ [\sigma_1,\sigma_2], \ [\sigma_1,\sigma_3], \ [\sigma_2,\sigma_3], \ \sigma_1^\pi,\ \sigma_2^\pi, \ \sigma_3^\pi. \]
From this it is easy to see that $R$ has $\Z/\ell$-dimension 3.

Clearly, $\Gc^{a,n}$ is generated by pairwise CL elements (namely $\sigma_1,\sigma_2,\sigma_3$).
However, simple dimension considerations (see above) show that $(\sigma_1+\sigma_2),\sigma_3$ do not form a CL-pair.
Thus, $\Gc^{a,n}$ is generated by pairwise CL elements, but it is not a CL-group.

Fortunately, in the case where $\Gc = \Gc_K$ for a field $K$ with $\Char K \neq \ell$ and $\mu_{2\ell^n} \subset K$, it will be a consequence of Theorem \ref{thm:cl-to-c-pairs} that a subgroup $A$ of $\Gc_K^{a,n}$ is CL if and only if it is generated by pairwise CL elements, and that $\Ibcl(A)$ is indeed a subgroup.
\end{remark}

\subsection{Minimal Free Pro-$\ell$ presentations}
\label{sec:minimal-free-pro}

In this subsection we recall some basic facts about minimal free presentations of pro-$\ell$ groups and the relationship between cup-products resp. Bockstein and commutators resp. $\ell^n$-th powers.
For a reference, see \cite{Neukirch2008} Chapter 3.9.

Let $\Gc$ be a pro-$\ell$ group and assume that $\Gc^{a,n}$ is isomorphic to a direct power of $R_n$.
Choose a (convergent) minimal generating set $(\sigma_i)_{i \in \Lambda}$ for $\Gc^{a,n}$.
Furthermore, continuously choose lifts $\tilde\sigma_i \in \Gc$ of $\sigma_i$; then $(\tilde\sigma_i)_{i \in \Lambda}$ is a minimal generating set for $\Gc$.
Let $S$ be the free pro-$\ell$ group on $(\tilde\gamma_i)_{i \in \Lambda}$ and consider the surjective homomorphism $S \rightarrow \Gc$ defined by $\tilde\gamma_i \mapsto \tilde\sigma_i$.
We will furthermore denote by $\gamma_i \in S^{a,n}$ the image of $\tilde\gamma_i$ under the map $S \rightarrow S^{a,n}$; thus $(\gamma_i)_{i \in \Lambda}$ is a minimal generating set for $S^{a,n}$.
It is easy to see that the induced map $S^{a,n} \rightarrow \Gc^{a,n}$ is an isomorphism.
We will call such a homomorphism $S \rightarrow \Gc$ a \emph{minimal} free presentation.

Denote by $(x_i)_{i \in \Lambda}$ the $R_n$-basis for $H^1(S^{a,n}) = H^1(S)$ which is dual to $(\gamma_i)_{i\in\Lambda}$, and choose a total ordering for the index set $\Lambda$.
With these choices made, every element $\rho$ of $S^{(2,n)}/S^{(3,n)}$ has a \emph{unique} representation as:
\[ \rho = \sum_{i < j} a_{ij}(\rho) \cdot [\gamma_i,\gamma_j] + \sum_r b_r(\rho) \cdot \gamma_r^\pi. \]
The coefficients $a_{ij}$ and $b_r$ can therefore be considered as homomorphisms $S^{(2,n)}/S^{(3,n)} \rightarrow R_n$.
We may therefore consider $a_{ij}$ and $b_r$ as elements of $H^1(S^{(2,n)})^{S^{a,n}} = \Hom(S^{(2,n)}/S^{(3,n)},R_n)$.

Let $T$ denote the kernel of our minimal free presentation $S \rightarrow \Gc$.
Since $S^{a,n} \rightarrow \Gc^{a,n}$ is an isomorphism, we note that $T \leq S^{(2,n)}$.
Thus, we may restrict $a_{ij}$ and $b_r$ to elements of $H^1(T)^{\Gc} = \Hom(T/([S,T] \cdot T^{\ell^n}), R_n)$.

The spectral sequence associated to the extension $S \rightarrow \Gc$ induces an isomorphism:
\[ d_2 : H^1(T)^\Gc \rightarrow H^2(\Gc) \]
since $S$ and $T$ have $\ell$-cohomological dimension $\leq 1$ and the inflation $H^1(\Gc) \rightarrow H^1(S)$ is an isomorphism. 
Thus, we obtain a canonical perfect pairing:
\[(\bullet,\bullet) : H^2(\Gc) \times \left(\frac{T}{[S,T] \cdot T^{\ell^n}}\right) \rightarrow R_n \]
defined by $(\xi,\rho) = (d_2^{-1}\xi)(\rho)$.
This pairing can be described explicitly using the cup product and Bockstein morphism (see \cite{Neukirch2008} Propositions 3.9.13 and 3.9.14):
\begin{itemize}
\item $(x_i \cup x_j,\bullet) = -a_{ij}(\bullet)$, $i < j$.
\item $(\beta x_r,\bullet) = -b_r(\bullet)$.
\end{itemize}

Since it will be mentioned later on, we conclude this subsection with the following observation.
If we take $\Gc = S^{a,n}$, the discussion above yields a canonical perfect pairing:
\begin{align}
\label{Free-h2-pairing}
(\bullet,\bullet) : H^2(S^{a,n}) \times S^{(2,n)}/S^{(3,n)} \rightarrow R_n
\end{align}
which satisfies $(x_i \cup x_j,\rho) = -a_{ij}(\rho)$ for $i<j$ and $(\beta x_r,\rho) = -b_r(\rho)$.

\subsection{Pro-$\ell$ Galois Groups}
\label{sec:pro-ell-galois}

In this subsection, we will deal with pro-$\ell$ Galois groups of a field $K$ with $\Char K \neq \ell$ and $\mu_{2\ell^n} \subset K$.
In such a situation, we will choose, once and for all, an isomorphism of $G_K$-modules $R_n(1) \cong R_n$.
This isomorphism $R_n(1) \cong R_n$ induces isomorphisms $R_n(i) \cong R_n$ for all $i$, and we use these isomorphisms tacitly throughout.
Recall that Kummer theory yields a canonical perfect pairing:
\begin{align*}
\text{If $n \neq \infty$: } \ &  \Gc_K^{a,n} \times K^\times/\ell^n \rightarrow \Z/\ell^n(1). \\
\text{If $n = \infty$: } \ & \Gc_K^{a,n} \times \widehat{K^\times} \rightarrow \Z_\ell(1).
\end{align*}
Since $-1 \in K^{\times \ell^n}$, our fixed isomorphism $R_n \cong R_n(1)$ yields an isomorphism $\Gc_K^a(n) \cong \Gc_K^{a,n}$.

The Merkurjev-Suslin theorem \cite{Merkurjev1982} states that the Galois symbol is an isomorphism:
\begin{align*}
\text{If $n \neq \infty$: } \ & K_2^M(K)/\ell^n \cong H^2(K,\Z/\ell^n(2)). \\
\text{If $n = \infty$: } \  & \widehat K_2^M(K) \cong H^2(K,\Z_\ell(2)).
\end{align*}
In particular, the cup-product map $\cup : H^1(K,R_n(1)) \otimes H^1(K,R_n(1)) \rightarrow H^2(K,R_n(2))$ is surjective.
In turn, this implies that the inflation map $H^2(\Gc_K^{a,n}) \rightarrow H^2(\Gc_K)$ is surjective; this observation will be used below.

\begin{prop}
\label{prop:h1-h2-pairings}
Let $K$ be a field such that $\Char K \neq \ell$ and $\mu_{\ell^n} \subset K$.
Choose a minimal free pro-$\ell$ presentation $S \twoheadrightarrow \Gc_K$; namely, $S$ is a free pro-$\ell$ group, $S\rightarrow \Gc_K$ is surjective and the induced map $S^{a,n} \rightarrow \Gc_K^{a,n}$ is an isomorphism.
Denote by $R$ the kernel of the induced surjective map $S^{c,n} \rightarrow \Gc_K^{c,n}$.
Then one has a canonical perfect pairing:
\[ H^2(\Gc_K) \times R \rightarrow R_n \]
induced by the free presentation.
Furthermore, this pairing is compatible with the perfect pairing $H^2(S^{a,n}) \times S^{(2,n)}/S^{(3,n)} \rightarrow R_n$ described in (\ref{Free-h2-pairing}), via the inflation map $H^2(S^{(a,n)}) = H^2(\Gc_K^{a,n}) \rightarrow H^2(\Gc_K)$ and the inclusion $R \hookrightarrow S^{(2,n)}/S^{(3,n)}$.
\end{prop}
\begin{proof}
Let $T$ denote the kernel of $S \rightarrow \Gc_K$.
As discussed above, the spectral sequence associated to this extension induces an isomorphism:
\[ d_2 : H^1(T)^S \rightarrow H^2(\Gc_K). \]
Furthermore, $H^1(T)^S$ is in perfect duality with $T/([S,T] \cdot T^{\ell^n})$.
Thus, it suffices to show that the canonical map:
\[ \frac{T}{[S,T]T^{\ell^n}} \rightarrow \frac{T \cdot S^{(3,n)}}{S^{(3,n)}} = R \]
is an isomorphism.
Clearly this is a surjective map.

Taking $R_n$-duals of the composition
\[ \frac{T}{[S,T]T^{\ell^n}} \rightarrow \frac{T \cdot S^{(3,n)}}{S^{(3,n)}} \hookrightarrow \frac{S^{(2,n)}}{S^{(3,n)}},\]
we obtain the inflation map $H^2(S^{a,n}) \cong H^2(\Gc_K^{a,n}) \rightarrow H^2(\Gc_K)$ which is surjective by the Merkurjev-Suslin theorem \cite{Merkurjev1982}, as discussed above.
Thus $T/([S,T]T^{\ell^n}) \rightarrow S^{(2,n)}/S^{(3,n)}$ is injective by Pontryagin duality.
In particular, $T/([S,T]T^{\ell^n}) \rightarrow (T \cdot S^{(3,n)})/S^{(3,n)} = R$ is injective as well.

The compatibility with the pairing described in (\ref{Free-h2-pairing}) is immediate by the functoriality of the situation, along with our requirement that $S^{a,n} \rightarrow \Gc_K^{a,n}$ is an isomorphism.
\end{proof}

For a field $K$ and $n \in \Nbar$ as above, our fixed isomorphism $R_n(1) \cong R_n$ allows us to explicitly express the Bockstein morphism $\beta : H^1(\Gc_K,R_n) \rightarrow H^2(\Gc_K,R_n)$ using Milnor K-theory, as follows.

If $n = \infty$ this map is trivial, so there is nothing to say.
Let us therefore temporarily assume that $n \in \Nb$.
The following fact seems to be well-known; see \cite{Efrat2011b} Proposition 2.6 for a precise reference.
Denote by $\delta$ the canonical Kummer map $K^\times \rightarrow H^1(K,\mu_{\ell^n})$.
Then the cup product ${\bf 1} \cup \delta : H^1(\Gc_K,\Z/\ell^n) \otimes \mu_{\ell^n} \rightarrow H^2(\Gc_K,\mu_{\ell^n})$ is precisely the same as the map $\beta \cup {\bf 1}$.

Denote by $\omega$ the fixed generator of $\mu_{\ell^n}$ which corresponds to $1 \in \Z/\ell^n$ under our isomorphism $R_n(1) \cong R_n$.
This induces isomorphisms $H^1(\Gc_K,\Z/\ell^n) \cong H^1(G_K,\mu_{\ell^n}) \cong K_1^M(K)/\ell^n$ (by Kummer theory) and $H^2(\Gc_K,\Z/\ell^n) \cong H^2(G_K,\mu_{\ell^n}^{\otimes 2}) \cong K_2^M(K)/\ell^n$ (by Merkurjev-Suslin \cite{Merkurjev1982}).
Under these induced isomorphisms, we deduce that the Bockstein morphism $H^1(\Gc_K,\Z/\ell^n) \rightarrow H^2(\Gc_K,\Z/\ell^n)$ corresponds to the map $K_1^M(K)/\ell^n \rightarrow K_2^M(K)/\ell^n$ defined by $x \mapsto \{x,\omega\}$.
Namely, the following diagram commutes:
\[
\xymatrix
{
K_1^M(K)/\ell^n \ar[r]^{\cong} \ar[d]_{x \mapsto \{x,\omega\}} & H^1(K,\mu_{\ell^n}) \ar[r]^{\cong} \ar[d]^{\text{induced}} & H^1(\Gc_K,\Z/\ell^n) \ar[d]^{\beta \cup \mu_{\ell^n}} \ar@{=}[r] & H^1(\Gc_K,\Z/\ell^n) \ar[d]^{\beta} \\ 
K_2^M(K)/\ell^n \ar[r]_{\cong}  & H^1(K,\mu_{\ell^n}^{\otimes 2}) \ar[r]_{\cong} & H^2(\Gc_K,\mu_{\ell^n}) \ar[r]_{\cong} & H^2(\Gc_K,\Z/\ell^n)
}
\]
where the isomorphisms on the left are canonical given by the Galois symbol, while the isomorphisms on the right are induced by our fixed isomorphism $\mu_{\ell^n} = \langle \omega \rangle \cong \Z/\ell^n$.
We will tacitly use this fact for the rest of this section.

\begin{thm}
\label{thm:cl-to-c-pairs}
Let $K$ be a field such that $\Char K \neq \ell$ and $\mu_{2\ell^n} \subset K$.
Let $\sigma,\tau \in \Gc_K^{a,n}$ be given.
Consider $\sigma,\tau$ as homomorphisms $\sigma,\tau : K^\times \rightarrow R_n$ via our chosen isomorphism of $G_K$-modules $R_n(1) \cong R_n$ and Kummer theory.
Then $\sigma,\tau$ form a CL-pair if and only if $\sigma,\tau$ form a C-pair.
\end{thm}
\begin{proof}
We first assume that $n \in \Nb$ is finite.
The $n=\infty$ case follows in the limit using Remark \ref{rem:k-thy-to-c-pairs}; see the comment at the end of this proof.
Also, we may assume that $\langle \sigma,\tau \rangle$ is non-cyclic for otherwise the claim is trivial.
Furthermore, we may assume without loss of generality that $\sigma,\tau$ are quasi-independent.
In particular, one has $\langle \sigma,\tau \rangle = \langle \sigma \rangle \times \langle \tau \rangle$.

Therefore, we may choose a minimal generating set $(\sigma_i)_{i \in \Lambda}$ for $\Gc_K^{a,n}$ which \emph{extends} $\sigma,\tau$ in the following sense: $1,2 \in \Lambda$, and are some $a,b$ with $0 \leq a,b < n$ such that $\ell^a \cdot \sigma_1 = \sigma$ and $\ell^b \cdot \sigma_2 = \tau$.
We also choose a total ordering on the indexing set $\Lambda$.

Next, we (continuously) choose lifts $\tilde\sigma_i \in \Gc_K$ of $\sigma_i \in \Gc_K^{a,n}$; we denote the image of $\tilde\sigma_i$ in $\Gc_K^{c,n}$ by $\sigma_i^c$.
Therefore, $(\tilde\sigma_i)_{i \in \Lambda}$ is a minimal generating set for $\Gc_K$ and $(\sigma_i^c)_{i \in \Lambda}$ is a minimal generating set for $\Gc_K^{c,n}$.
Finally, we will consider $(x_i)_{i \in \Lambda}$ the basis for $H^1(\Gc_K)$ which is dual to $(\sigma_i)_{i\in \Lambda}$.
Using our isomorphism $R_n(1) \cong R_n$, we will consider $(x_i)_{i \in \Lambda}$ also as a basis for $K^\times/\ell^n$.

Using the isomorphism $R_n(1) \cong R_n$, we consider $\sigma_i$ as homomorphisms $K^\times \rightarrow R_n$.
Consider $H_0 := \ker \sigma_1 \cap \ker \sigma_2$ and $H := \ker \sigma \cap \ker \tau$; clearly, one has $H_0 \leq H$.
Furthermore, $K^\times/H_0$ is a free $\Z/\ell^n$-module of rank 2 which is generated by $x_1,x_2$.
Since $\ell^a \cdot \sigma_1 = \sigma$ and $\ell^b \cdot \sigma_2 = \tau$, one has $H = \langle H_0, x_1^{\ell^{n-a}},x_2^{\ell^{n-b}} \rangle$.

\vskip 5pt
\noindent{\bf CL-pair implies C-pair:}

We first assume that $\sigma,\tau$ form a CL-pair.
Let $A := \langle \sigma_1,\sigma_2 \rangle$ and $A^c := \langle \sigma_1^c, \sigma_2^c \rangle$.
Note that the Kummer-dual of the inclusion $A \hookrightarrow \Gc_K^{a,n}$ is precisely the projection $K^\times/\ell^n \rightarrow K^\times/H_0$.
Thus, we obtain the following commutative diagram:
\begin{align}
\label{diagram:inf-cup}
\xymatrix{
K^\times/\ell^n \times K^\times/\ell^n \ar[r]^-{\cong} \ar[d]_{\rm canonical} & H^1(\Gc_K^{a,n})\times H^1(\Gc_K^{a,n})  \ar[d]_{\res} \ar[r]^-{{\rm inf} \circ \cup} &H^2(\Gc_K^{c,n})  \ar[d]^{\res}\\
K^\times/H_0 \times K^\times/H_0 \ar[r]_{\cong} & H^1(A) \times H^1(A) \ar[r]_-{{\rm inf} \circ \cup}&  H^2(A^c)
}
\end{align}
By Lemma \ref{lem:cup-product}, the top map of (\ref{diagram:inf-cup}) factors through $K_2^M(K)/\ell^n$ and therefore the bottom map of (\ref{diagram:inf-cup}) factors through $K_2^M(K)/H_0$.

Let $F$ be the free pro-$\ell$ group on generators $\tilde \gamma_1,\tilde \gamma_2$, and consider the surjective map $F \rightarrow A^c$ defined by $\tilde \gamma_i \mapsto \sigma_i^c$.
Denote by $T$ the kernel of this (minimal) presentation $F \rightarrow A^c$ and let $\gamma_i$ denote the image of $\tilde \gamma_i$ in $F^{a,n}$.
Since $\ell^a \cdot \sigma_1,\ell^b \cdot \sigma_2$ form a CL-pair, we see that $(T \cdot F^{(3,c)})/F^{(3,c)} = T / F^{(3,c)}$ contains an element $\rho$ of the form:
\[\rho = \ell^{a+b} \cdot [\gamma_1,\gamma_2] + (c_1 \cdot \ell^a) \cdot \gamma_1^\beta +  (c_2 \cdot \ell^b) \cdot \gamma_2^\beta \]
for some $c_1,c_2 \in R_n$.

We recall the pairing of \S\ref{sec:minimal-free-pro} associated to the presentation $F \rightarrow A^c$, 
\[ (\bullet,\bullet) : H^2(A^c) \times \left(\frac{T}{[F,T]\cdot T^{\ell^n}}\right) \rightarrow \Z/\ell^n, \]
satisfies $(x_1 \cup x_2,\rho) = -\ell^{a+b}$.
Thus, $K_2^M(K)/H_0 = \langle \{x_1,x_2\}_{H_0} \rangle$ has order $\ell^{n-c_0}$ for some $0 \leq c_0 \leq a+b$.
Since $H = \langle H_0, x_1^{\ell^{n-a}},x_2^{\ell^{n-b}} \rangle$, we see that 
\[K_2^M(K)/H = \langle \{x_1,x_2\}_{H_0} \rangle / \langle \{x_1^{\ell^{n-a}},x_2\}_{H_0},\{x_1,x_2^{\ell^{n-b}}\}_{H_0} \rangle. \]
Thus, $K_2^M(K)/H$ has order $\ell^{n-c}$ where $\max(a,b,c_0) = c \leq a+b$.
Therefore $\sigma,\tau$ form a C-pair by the K-theoretic criterion (Proposition \ref{prop:k-thy-c-pair}).

\vskip 5pt
\noindent{\bf C-pair implies CL-pair:}

Now assume that $\sigma,\tau$ form a C-pair.
Let $S \rightarrow \Gc_K$ be a minimal free presentation associated to the minimal generating set $(\tilde \sigma_i)_i$.
Namely, $S$ is the free pro-$\ell$ group on $(\tilde \gamma_i)_i$ and the map $S \rightarrow \Gc_K$ is defined by $\tilde \gamma_i \mapsto \tilde \sigma_i$.
Recall that $S^{a,n} \rightarrow \Gc_K^{a,n}$ is an isomorphism, and denote by $\gamma_i \in S^{a,n}$ the image of $\tilde \gamma_i$.
We denote by $R$ the kernel of the induced surjective map $S^{c,n} \rightarrow \Gc_K^{c,n}$, as in Proposition \ref{prop:h1-h2-pairings}.

Observe that $K_2^M(K)/H$ is a rank-1 quotient of $K_2^M(K)/\ell^n$.
Using the isomorphism $K_2^M(K)/\ell^n \cong H^2(\Gc_K)$ induced by $R_n(1) \cong R_n$, and the pairing of Proposition \ref{prop:h1-h2-pairings}, we obtain a perfect pairing 
\[K_2^M(K)/\ell^n \times R \rightarrow R_n. \]
Under this pairing, $K_2^M(K)/H$ is dual to a rank-1 subgroup of $R$ which is generated by an element $\rho \in R$.
In other words, we obtain a restricted perfect pairing:
\[ (\bullet,\bullet)_H : K_2^M(K)/H \times \langle \rho \rangle \rightarrow R_n. \]
Furthermore, we can find $a_{ij},b_r \in R_n$ uniquely so that $\rho$ takes the form:
\[ \rho = \sum_{i < j}a_{ij} \cdot [\gamma_i,\gamma_j] +  \sum_r  b_r \cdot \gamma_r^\pi. \]

Recall that $\omega$ denotes the generator of $\mu_{\ell^n}$ which corresponds to $1 \in \Z/\ell^n$.
With this $\omega$ fixed, we recall that the pairing above satisfies $(\{x_i,x_j\}_H,\rho)_H = -a_{ij}$ (if $i < j$) and $(\{x_r,\omega\}_H,\rho)_H = - b_r$.

For all $f \in \Lambda \smallsetminus\{1,2\}$, one has $x_f \in H$ and thus $(\{x_f,x_j\}_H,\rho)_H = 0 = \pm a_{fj}$ and $(\{x_f,\omega\}_H,\rho)_H = 0 = b_f$.
Therefore, $\rho$ has the following simple form
\[ \rho = a_{12} \cdot [\gamma_1,\gamma_2] + b_1 \cdot \gamma_1^\pi + b_2 \cdot \gamma_2^\pi. \]

Since $\mu_{2\ell^n} \subset K$, we further see that $\omega$ is a square in $K^\times$.
Thus, we can find $j,k \in R_n$ such that $\omega \in (x_1^{-2j}x_2^{2k}) \cdot H$.
This observation induces the following relationship between $a_{12}$ and $b_1,b_2$:
\begin{enumerate}
\item $(\{x_1,x_2\}_H, \rho)_H = -a_{12}$
\item $(\{x_1,\omega\}_H, \rho)_H = 2k(\{x_1,x_2\}_H, \rho) = -2ka_{12} = -b_1$.
\item $(\{x_2,\omega\}_H, \rho)_H = 2j(\{x_1,x_2\}_H, \rho) = -2ja_{12} = -b_2$.
\end{enumerate}
Therefore, we see that
\[ \rho = a_{12} \cdot ([\gamma_1,\gamma_2] + k \cdot \gamma_1^\beta + j \cdot \gamma_2^\beta). \]

By Proposition \ref{prop:k-thy-c-pair}, $K_2^M(K)/H$ has order $\ell^{n-c}$ for some $0 \leq c \leq a+b$.
Since $\langle \rho \rangle$ is in a perfect $R_n$-pairing with $K_2^M(K)/H$, we deduce that $a_{12} \in \Z/\ell^n$ has (additive) order $\ell^{n-c'}$ for some $0 \leq c' \leq a+b$.
In particular, $a_{12}$ divides $\ell^{a+b}$ in $R_n$; say $t \in R_n$ satisfies $\ell^{a+b} = a_{12} \cdot t$.
Now we consider $\rho^t$:
\begin{align*}
 t \cdot \rho &= a_{12} \cdot t \cdot ([\gamma_1,\gamma_2] + k \cdot \gamma_1^\beta + j \cdot \gamma_2^\beta) \\
&=  \ell^{a+b} \cdot [\gamma_1,\gamma_2] + (k \cdot \ell^{a+b}) \cdot \gamma_1^\beta + (j \cdot \ell^{a+b}) \cdot \gamma_2^\beta \\
 &= [\ell^a \cdot \gamma_1,\ell^b \cdot \gamma_2]  + (k \cdot \ell^b) \cdot (\ell^a \cdot \gamma_1)^\beta + (j \cdot \ell^a) \cdot (\ell^b \cdot \gamma_2)^\beta.
\end{align*}
Since $t \cdot \rho \in R$, we deduce that $[\sigma,\tau] \in \langle \sigma^\beta,\tau^\beta \rangle$; i.e. $\sigma,\tau$ form a CL-pair, as required.

\vskip 5pt
\noindent{\bf The $n = \infty$ case:}

In this case we are assuming that $\mu_{\ell^\infty} \subset K$.
First recall that $f,g \in \Gc_K^a(\infty)$ form a C-pair if and only if $f_n,g_n$ form a C-pair for all $n \in \Nb$.
The proof that ``CL-pair'' implies ``C-pair'' now follows immediately from the $n \in \Nb$ case.

To conclude the proof that ``C-pair'' implies ``CL-pair,'' we note that the $j,k$ above would have been zero since $\mu_{\ell^\infty} \subset K$.
Thus, if $\sigma,\tau \in \Gc_K^{a,\infty}$ correspond to a C-pair in $\Gc_K^a(\infty)$, then, arguing as in the $n \in \Nb$ case, for all $n \in \Nb$ one has $[\sigma_n,\tau_n] = 0$.
Therefore $[\sigma,\tau] = 0$ as well and thus $\sigma,\tau$ are a CL-pair.
\end{proof}

\begin{remark}
\label{rem:CL-pairs-to-CL-groups}
As an immediate corollary of Theorem \ref{thm:cl-to-c-pairs} we deduce the following.
Given $(\sigma_i)_i \in \Gc_K^{a,n}$ which are pairwise CL, then any pair $\sigma,\tau \in \langle \sigma_i \rangle_i$ form a CL-pair.
We recall Remark \ref{rem:caution-CL-pairs} which notes that this doesn't hold for an arbitrary pro-$\ell$ group.

We also deduce that, for a subgroup $A \leq \Gc_K^{a,n}$, the subset $\Ibcl(A) \subset A$ is indeed a \emph{subgroup} which corresponds to $\Ibc(A)$ as defined in Part \ref{part:underlying-theory} via the Kummer isomorphism $\Gc_K^{a,n} \cong \Gc_K^a(n)$.
\end{remark}

\begin{remark}
\label{remark:ibcl}
Let $K$ be a field such that $\Char K \neq \ell$ and $\mu_{2\ell} \subset K$ and let $A \leq \Gc_K^{a,1}$ be given.
Using Remark \ref{remark:rigidity-k-thy}, we can now give an alternative definition for $\Ibcl(A)$ which extends the situation of \cite{Efrat2011a}.
Namely, in this remark we will show that:
\[ \Ibcl(A) = \{ \sigma \in A \ : \ \forall \tau \in A, \ [\sigma,\tau] \in A^\beta  \} =: I. \]
Observe that $\Ibcl(A) \leq I$ by definition and so it suffices to prove that $I \leq \Ibcl(A)$.

We will identify $\Gc_K^a(1)$ and $\Gc_K^{a,1}$ via Kummer theory and $R_1(1) \cong R_1$.
By Theorem \ref{thm:cl-to-c-pairs}, this isomorphism identifies C-pairs with CL-pairs.
Furthermore, we will identify subgroups $\Gc_K^{a,1}$ with their images in $\Gc_K^a(1)$.

Consider $T := A^\perp$ and $H := I^\perp$, both are subgroups of $K^\times$.
Suppose that $G$ is a subgroup of $K^\times$ such that $T \leq G \leq H$ and $H/G$ is cyclic.
We claim that $\Hom(K^\times/G,\Z/\ell)$ is a C-group, therefore proving that $\langle I,f \rangle$ is a C-group for all $f \in A$.
By Theorem \ref{thm:cl-to-c-pairs}, this would immediately imply that $I \leq \Ibcl(A)$, thus proving our claim.

Let $x_1 \in K^\times \smallsetminus H$ and $x_2 \in K^\times \smallsetminus T$ be given such that $x_1,x_2$ have $\Z/\ell$-independent images in $K^\times/T$.
Thus, we may complete $x_1,x_2$ to a $\Z/\ell$-basis $(x_i)_i$ for $K^\times/T$, with dual basis $(\sigma_i)_i$ for $A$, in such a way so that $\sigma_1 \in I$.

We immediately see that $[\sigma_1,\sigma_2] \in \langle \sigma_i^\beta \rangle_i$ by the definition of $I$.
Choose lifts (continuously) $\sigma_i^c \in \Gc_K^{c,1}$ for $\sigma_i$ and let  $A^c := \langle \sigma_i^c \rangle_i$.
Let $F$ the free pro-$\ell$-group on $(\tilde\gamma_i)_i$, and $F \twoheadrightarrow A^c$ the free presentation sending $\tilde\gamma_i$ to $\sigma_i^c$.
Let $\gamma_i$ denote the image of $\tilde\gamma_i$ in $F^{a,1}$.
Denote by $R$ the kernel of $F \rightarrow A^c$.
Since $[\sigma_1,\sigma_2] \in \langle \sigma_i^\beta \rangle$, we see that $R / F^{(3,1)}$ contains an element of the form:
\[ \rho = [\gamma_1,\gamma_2] + \sum_r b_r \cdot \gamma_r^{\beta}. \]
Arguing as in the proof of Theorem \ref{thm:cl-to-c-pairs}, we see that $(\{x_1,x_2\}_T,\rho) = 1 \neq 0$, where $(\bullet,\bullet)$ is the pairing arising from Proposition \ref{prop:h1-h2-pairings}.
In particular, $\{x_1,x_2\}_T \neq 0$.

Now suppose that $x \in K^\times \smallsetminus G$ is given and consider $1-x \in K^\times$.
If $x \notin H$ or $1-x \notin H$, we deduce from the argument above that $\langle x \cdot T,(1-x) \cdot T \rangle$ is a cyclic subgroup of $K^\times/T$; indeed, $\{x,1-x\}_T = 0$ and thus $x \cdot T, (1-x) \cdot T$ cannot be $\Z/\ell$-independent in $K^\times/T$.
Therefore $\langle x \cdot G,(1-x) \cdot G \rangle$ is a cyclic subgroup of $K^\times/G$.
On the other hand, if both $x,(1-x)$ are element of $H$, then $\langle x \cdot G,(1-x) \cdot G \rangle$ is a cyclic subgroup of $K^\times/G$ since $H/G$ is cyclic.

In any case, we see that $\langle x \cdot G, (1-x) \cdot G \rangle$ is cyclic for all $x \in K^\times \smallsetminus G$. 
Thus, $G$ satisfies condition (4) of Remark \ref{remark:rigidity-k-thy} which proves that $\Hom(K^\times/G,\Z/\ell)$ is a C-group, as required.
\end{remark}

\begin{remark}
\label{rem:minimized-inertia-decomp}
Let $(K,v)$ be a valued field such that $\Char K \neq \ell$ and $\mu_{2\ell^n} \subset K$.
Recall that $K^{a,n} := K(\sqrt[\ell^n]{K})$ is the Galois extension of $K$ with $\Gal(K^{a,n}|K) = \Gc_K^{a,n}$.
Also recall that the minimized decomposition/inertia subgroups of $v$ are defined to be:
\[ D_v^n := \Gal(K^{a,n}|K(\sqrt[\ell^n]{1+\mf_v})), \ \text{ and } \ I_v^n := \Gal(K^{a,n}|K(\sqrt[\ell^n]{\Oc_v^\times})). \]

It immediately follows from Kummer theory, Theorem \ref{thm:cl-to-c-pairs} and Lemma \ref{lem:I-and-D-give-c-pairs} that 
\begin{align}
\label{eqn:ID-Ibcl-incl}
I_v^n \leq \Ibcl(D_v^n)
\end{align}
In this remark, we give a refinement of this fact using an argument similar to that of Theorem \ref{thm:cl-to-c-pairs}.
Namely, the following proposition shows that the structure of $I_v^n \leq D_v^n$ resembles inertia/decomposition of a valuation of residue characteristic different from $\ell$, regardless of $\Char k(v)$.

We first set some notation.
Let $\sigma \in I_v^n$ and $\tau \in D_v^n$ be given.
If $n = \infty$, then (\ref{eqn:ID-Ibcl-incl}) implies that $[\sigma,\tau] = 0$.
We therefore assume for the rest of this remark that $n \in \Nb$.

Let $\omega \in \mu_{\ell^n}$ be a primitive $\ell^n$-th root of unity which corresponds to $1 \in R_n$ via our isomorphism $R_n(1) \cong R_n$.
Using this isomorphism $R_n(1) \cong R_n$ and Kummer theory, we may consider $\sigma,\tau$ as homomorphisms $K^\times \rightarrow R_n$; in particular, $\tau(\omega)$ is an element of $R_n$.
Furthermore, since $\mu_{2\ell^n} \subset K$, $\omega$ is a square in $K^\times$.
Thus, there exists a coefficient $a \in R_n$ such that $\tau(\omega) = 2 \cdot a$.

\begin{prop}
\label{prop:decthy-analogue}
In the notation above, one has $[\sigma,\tau] = - a \cdot \sigma^\beta = (-\tau(\omega)) \cdot \sigma^\pi$.
\end{prop}
\begin{proof}
It suffices to assume that $\sigma,\tau$ are actually $R_n$-independent.
Choose a minimal generating set $(\sigma_i)_i$ for $\Gc_K^{a,n}$ such that $\sigma_1 = \sigma$ and $\sigma_2 = \tau$.
Let $(x_i)_i$ be the basis for $H^1(K,R_n(1)) \cong H^1(\Gc_K^{a,n},R_n)$ which is dual to $(\sigma_i)_i$.
Furthermore, choose a minimal free presentation $S \rightarrow \Gc_K$ and use the same notation as in the second part of the proof of Theorem \ref{thm:cl-to-c-pairs} -- in particular, $R$ denotes the kernel of $S^{c,n} \rightarrow \Gc_K^{c,n}$.

Let $H = \ker\sigma_1 \cap \ker \sigma_2$.
Since $\sigma_1(\omega) = 0$ and $\sigma_2(\omega) = 2a$, we deduce
\[\omega \cdot H = x_1^{\sigma_1(\omega)} \cdot x_2^{\sigma_2(\omega)} \cdot H = x_2^{2a} \cdot H. \]
By Lemma \ref{lem:I-and-D-give-c-pairs} and Theorem \ref{thm:cl-to-c-pairs}, we see that $R$ contains an element of the form 
\[ \rho = [\gamma_1,\gamma_2] + c_1 \cdot \gamma_1^\beta + c_2 \cdot \gamma_2^\beta. \]
Arguing as in the proof of Theorem \ref{thm:cl-to-c-pairs} we see that $\langle \rho \rangle$ is in perfect duality with $K_2^M(K)/H$ via the pairing induced by Proposition \ref{prop:h1-h2-pairings}.
Namely, we have a perfect pairing:
\[(\bullet,\bullet)_H : K_2^M(K)/H \times \langle \rho \rangle \rightarrow R_n \]
which satisfies: $(\{x_1,x_2\}_H, \rho)_H = -1$ and $(\{x_i,\omega \}_H, \rho)_H = -2\cdot c_i$.
Therefore, we see that 
\[ -2 c_1 = (\{x_1,\omega\}_H, \rho)_H = 2 a\cdot (\{x_1,x_2\}_H, \rho)_H = -2 a\]
and 
\[ -2c_2 = (\{x_2,\omega\}_H, \rho)_H = 2 a\cdot (\{x_2,x_2\}_H, \rho)_H = 0. \]
In particular, $R$ contains an element $\rho$ of the form $[\gamma_1,\gamma_2] + a \cdot \gamma_1^\beta$.
We deduce that $[\sigma_1,\sigma_2] = -a \cdot \sigma_1^\beta$, as required.
\end{proof}
\end{remark}

\part{Proofs of Theorems and Corollaries}
\label{part:proofs}

Having developed the necessary machinery in Parts \ref{part:underlying-theory} and \ref{part:galois-theory-milnor}, we are now ready to prove the main results of the paper.
The main goal of this part is to prove Theorems \ref{thm:main-thm-intro-maximal} and \ref{thm:main-thm-char-intro}.
We will also prove the Corollaries from \S\ref{sec:guide-thro-manuscr}.

In order to stay in line with the notation of \S\ref{sec:main-results-manuscr} where our main theorems are stated, for $n \in \Nbar$, we define:
\[ \Rbf(n) = \Nfr(\Mfr_2(\Mfr_1(n))). \]
We note that the definition of $\Nfr$ and $\Mfr_i$ ensures that the following three conditions, which were mentioned in \S\ref{sec:main-results-manuscr}, hold true:
\begin{enumerate}
\item If $n \in \Nb$ then $\mathbf{R}(n) \in \Nb$.
\item One has $\mathbf{R}(1) = 1$ and $\mathbf{R}(\infty) = \infty$.
\item One has $\mathbf{R}(n) \geq n$ for all $n \in \Nbar$.
\end{enumerate}

\section{Proof of Theorem \ref{thm:main-thm-intro-maximal}}
\label{sec:main-results}

We use the notation of Theorem \ref{thm:main-thm-intro-maximal}.
Namely, $n \in \Nbar$ and $N \geq \Rbf(n)$ are given.
Also, $K$ is a field such that $\Char K \neq \ell$ and $\mu_{2\ell^N} \subset K$.

Since $-1 \in K^{\times \ell^N}$, our choice of isomorphism $R_N(1) \cong R_N$ and Kummer theory yield isomorphisms
\[ \phi_m : \Gc_K^{a,m} \xrightarrow{\cong} \Gc_K^a(m) \]
for all $m \leq N$.
If $m \leq m' \leq N$, then the canonical projections $\Gc_K^{a,m'} \rightarrow \Gc_K^{a,m}$ resp. $\Gc_K^a(m') \rightarrow \Gc_K^a(m)$ are compatible via the isomorphisms $\phi_m$ and $\phi_{m'}$.

Compatibility with subgroups goes as follows.
For a subgroup $H \leq K^\times$, the subgroup $\Gal(K^{a,m}|K(\sqrt[\ell^m]{H}))$ of $\Gc_K^{a,m}$ is mapped isomorphically onto the subgroup $\Hom(K^\times/H,R_m)$ of $\Gc_K^a(m)$.
Similarly, a subgroup $A$ of $\Gc_K^{a,m}$ maps isomorphically onto $\Hom(K^\times/\phi_m(A)^\perp, R_m)$.
In particular, for any valuation $v$ of $K$ and $m \leq N$, we see that $I_v^m$ is mapped isomorphically onto $I_v(m)$ and $D_v^m$ is mapped isomorphically onto $D_v(m)$.

Finally, by Theorem \ref{thm:cl-to-c-pairs}, the isomorphisms $\phi_m$ identify CL-pairs and C-pairs.
In particular, for a subgroup $A \leq \Gc_K^{a,m}$, $\phi_m$ sends $\Ibcl(A)$ isomorphically onto $\Ibc(\phi_mA)$.
In light of all of these compatible identifications, we immediately see that claim (1) of Theorem \ref{thm:main-thm-intro-maximal} follows from Proposition \ref{prop:maximal-C-group-D}.
Similarly, claim (2) of Theorem \ref{thm:main-thm-intro-maximal} follows from Theorem \ref{thm:maximal-I-D-mu}.
This completes the proof of Theorem \ref{thm:main-thm-intro-maximal}.

\section{Decomposition Theory and the Proof of Theorem \ref{thm:main-thm-char-intro}}
\label{sec:hilb-decomp-theory}

In this section we will prove Theorem \ref{thm:main-thm-char-intro}.
We must first recall some well-known facts concerning decomposition theory of valuations in pro-$\ell$ Galois extensions.
We must also prove Proposition \ref{prop:decomp-thy-prop} which compares minimized inertia/decomposition groups with the usual inertia/decomposition groups.

We first recall the collection $\Vc_{K,n}'$ which was defined in the introduction.
We define $\Vc_{K,n}'$ to be the collection of valuations $v$ of $K$ which satisfy the following four conditions:
\begin{enumerate}
\item[(V0')] One has $\Char k(v) \neq \ell$.
\item[(V1')] The value group $\Gamma_v$ contains no non-trivial $\ell$-divisible convex subgroups.
\item[(V2')] The valuation $v$ is maximal among all valuations $w$ such that $\Char k(w) \neq \ell$, $D_v^n = D_w^n$ and $\Gamma_w$ contains no non-trivial $\ell$-divisible convex subgroups.
Namely, for all refinements $w$ of $v$ such that $\Char k(w) \neq \ell$ and $D_w^n = D_v^n$ as subgroups of $\Gc_K^{a,n}$, one has $I_w^n = I_v^n$.
\item[(V3')] The group $\Gc_{k(v)}^{a,n}$ is non-cyclic.
\end{enumerate}

We observe that $\Vc_{K,n} = \Vc_{K,n}'$ whenever $\ell \neq \Char K > 0$.
For an arbitrary field $K$, one has:
\[ \{v \in \Vc_{K,n} \ : \ \Char k(v) \neq \ell \} \subset \Vc_{K,n}'. \]

\subsection{Hilbert's Decomposition Theory}
\label{sec:hilb-decomp-theory-general}

Let $(K,v)$ be a valued field with $\Char K \neq \ell$ and $\mu_{\ell^n} \subset K$.
Assume furthermore that $\Char k(v) \neq \ell$.
Choose $v'$ a prolongation of $v$ to $K^{a,n} = K(\sqrt[\ell^n]{K})$.
Recall that $T_v^n  \leq Z_v^n$ denote the inertia/decomposition subgroups of $\Gc_K^{a,n}$ associated to $v'|v$; since $\Gc_K^{a,n}$ is abelian, these subgroups are independent of the choice of prolongation and so we omit $v'$ from the notation.

Choose $w$ a prolongation of $v'|v$ to $K(\ell)$ and recall that $T_{w|v} \leq Z_{w|v}$ denote the inertia/decomposition subgroups of $\Gc_K = \Gal(K(\ell)|K)$ associated to $w|v$.
Furthermore, recall that $k(w) = k(v)(\ell)$ is the maximal pro-$\ell$ Galois extension of $k(v)$, and that the inertia/decomposition subgroups of $w|v$ fit into the following short exact sequence:
\[ 1 \rightarrow T_{w|v} \rightarrow Z_{w|v} \rightarrow \Gc_{k(v)} \rightarrow 1. \]

Kummer theory can be used to describe the action of $\Gc_{k(v)}$ on $T_{w|v}$, as follows.
One has a perfect pairing which is compatible with the action of $\Gc_{k(v)}$ on $T_{w|v}$:
\[ T_{w|v} \times (\Gamma_w/\Gamma_v) \rightarrow \mu_{\ell^\infty} = \mu_{\ell^\infty}(k(v)) \]
which is defined by $(\sigma,w(x)) = \overline{\sigma x / x}$; as usual, here $\bar y$ denotes the image of $y \in \Oc_w^\times$ in $k(w)^\times$.

In particular, we see that $T_{w|v}$ is isomorphic to $\Z_\ell^B$ for some indexing set $B$, and that the action of $\Gc_{k(v)}$ on $T_{w|v}$ factors through the $\ell$-adic cyclotomic character $\chi_\ell : \Gc_{k(v)} \rightarrow \Z_\ell^\times$.
Since $\mu_{\ell^n} \subset K$, and thus $\mu_{\ell^n} \subset k(v)$, the image of $\chi_\ell$ is contained in $1+\ell^n \cdot \Z_\ell$.
To summarize, if $\sigma \in T_{w|v}$ and $\tau \in Z_{w|v}$, then the following holds (compare with Proposition \ref{prop:decthy-analogue}):
\begin{align}
\label{eqn:inert-dec-cl-pair}
\sigma^{-1} \tau^{-1} \sigma  \tau = \sigma^{(-1-\chi_\ell(\tau))} \in \langle \sigma^{\ell^n} \rangle.
\end{align}
In particular, if $\sigma \in T_v^n$ and $\tau \in Z_v^n$, then $\sigma,\tau$ form a CL-pair.
Thus, $T_v^n \leq \Ibcl(Z_v^n)$.

What's more important, however, is that this property is preserved in extensions.
More precisely, suppose that $L|K$ is a Galois sub-extension of $K(\ell)|K$, and that $w_0$ denotes the restriction of $w$ to $L$; observe that $\Gc_L = \Gal(K(\ell)|L)$.
Then the following hold:
\begin{enumerate}
\item $T_{w|v}|_L = T_{w_0|v}$ and $Z_{w|v}|_L = Z_{w_0|v}$.
\item $T_{w|v} \cap \Gc_L = T_{w|w_0}$ and $Z_{w|v} \cap \Gc_L = Z_{w|w_0}$.
\end{enumerate}

We summarize this discussion with the following proposition which will be used in the proof of Theorem \ref{thm:main-thm-char-intro}.
\begin{prop}
\label{prop:summary-decomp-thy}
Let $n, N\in \Nbar$ be given such that $N \geq n$.
Let $(K,v)$ be a valued field such that $\Char K \neq \ell$, $\mu_{\ell^N} \subset K$, and $\Char k(v) \neq \ell$.
Suppose that $Z \leq Z_v^n$ is a subgroup and $K_Z := (K^{a,n})^Z$ is the associated Galois extension of $K$.
Also, let $w$ denote a prolongation of $v$ to $K_Z$.
Then the following hold:
\begin{enumerate}
\item One has $Z = ((Z_{w}^N)_n)_K$.
\item One has $T_v^n \cap Z = ((T_{w}^N)_n)_K$. 
\item One has $T_w^N \leq \Ibcl(Z_w^N)$. In particular, $T_v^n \cap Z$ is contained in $(\Ibcl(Z_{w}^N)_n)_K$.
\end{enumerate}
\end{prop}
\begin{proof}
Claims (1) and (2) follows from the compatibility of decomposition groups in towers of field extensions, while claim (3) was discussed above (see equation (\ref{eqn:inert-dec-cl-pair})).
\end{proof}

\subsection{Minimized Inertia/Decomposition}

Let $(K,v)$ be a valued field such that $\Char K \neq \ell$ and $\mu_{\ell^n} \subset K$.
Recall that the {\bf minimized inertia and decomposition} groups of $v$, which are subgroups of $\Gc_K^{a,n}$, are defined as:
\[ I_v^n = \Gal(K^{a,n}|K(\sqrt[\ell^n]{\Oc_v^\times})), \ \text{ and } \ D_v^n = \Gal(K^{a,n}|K(\sqrt[\ell^n]{1+\mf_v})). \]

The proof of the following proposition can be found in \cite{Pop2010} Fact 2.1 in the $n=\infty$ case and in \cite{Pop2011a} in the $n=1$ case, but is explicitly stated for valuations $v$ such that $\Char k(v) \neq \ell$.
It turns out that one direction of the proof still works, even if $\Char k(v) = \ell$.
We summarize this in the proposition below.

\begin{prop}
\label{prop:decomp-thy-prop}
Let $(K,v)$ be a valued field such that $\Char K \neq \ell$ and $\mu_{\ell^n} \subset K$.
Then $D_v^n \leq Z_v^n$ and $I_v^n \leq T_v^n$.
If furthermore $\Char k(v) \neq \ell$, then $D_v^n = Z_v^n$ and $I_v^n = T_v^n$.
\end{prop}
\begin{proof}
The $n=\infty$ case follows from the $n \in \Nb$ case using a limit argument.
Thus, it suffices to prove the claim for $n \in \Nb$.

Suppose $a \in K^\times$ is such that $\sqrt[\ell^n]{a} \in (K^{a,n})^{Z_v^n}$.
Let $w$ be a prolongation of $v$ to $(K^{a,n})^{Z_v^n}$.
Since $\Gamma_w = \Gamma_v$, there exists $y \in K^\times$ such that $v(a) = \ell^n \cdot v(y)$.
Moreover, as $k(v) = k(w)$, there exists $z \in \Oc_v^\times$ such that $\sqrt[\ell^n]{a}/y \in z \cdot (1+\mf_w)$.
Namely, $a/(yz)^{\ell^n} \in (1+\mf_v)$ so that $\sqrt[\ell^n]{a} \in K(\sqrt[\ell^n]{1+\mf_v})$.
This implies that $(K^{a,n})^{Z_v^n} \subset K(\sqrt[\ell^n]{1+\mf_v})$; thus $D_v^n \leq Z_v^n$.
The proof that $(K^{a,n})^{T_v^n} \subset K(\sqrt[\ell^n]{\Oc_v^\times})$ is similar, and this implies that $I_v^n \leq T_v^n$.

Assume furthermore that $\Char k(v) \neq \ell$.
Let $(K^Z,v)$ be some Henselization of $(K,v)$; recall that $K^Z \cap K^{a,n} = (K^{a,n})^{Z_v^n}$.
Let $a \in 1+\mf_v$ be given.
The polynomial $X^{\ell^n} - a$ reduces mod $\mf_v$ to the polynomial $X^{\ell^n} - 1$.
Since $\Char k(v) \neq \ell$, one has $\mu_{\ell^n} \subset k(v)$ and the polynomial $X^{\ell^n}-1$ has $\ell^n$ unique roots in $k(v)$.
By Hensel's lemma, the polynomial $X^{\ell^n}-a$ has a root in $K^Z \cap K^{a,n} = (K^{a,n})^{Z_v^n}$.
This therefore implies that $K(\sqrt[\ell^n]{1+\mf_v}) \subset (K^{a,n})^{Z_v^n}$; thus $Z_v^n \leq D_v^n$.
The proof that $K(\sqrt[\ell^n]{\Oc_v^\times}) \subset (K^{a,n})^{T_v^n}$ is similar, and this implies that $T_v^n \leq I_v^n$.
\end{proof}

\begin{remark}
\label{rem:charkv-is-ell}
If $\Char k(v) = \ell$, the minimized inertia/decomposition are quite different from the usual inertia/decomposition.
We illustrate this phenomenon with the following fact: if $(K,v)$ is a valued field such that $\Char K \neq \ell$, $\mu_\ell \subset K$ and $\Char k(v) = \ell$, then $D_v^1 \leq T_v^1$.
The proof of this fact is very similar to \cite{Pop2010b} Lemma 2.3(2); we sketch the argument below.

Let $\omega \in \mu_\ell$ be a primitive $\ell$-th root of unity.
Let $\lambda := \omega - 1 \in K$, and recall that $v(\lambda) > 0$ since $\Char k(v) = \ell$.
Let $u \in \Oc_v^\times$ be given and set $u' := \lambda^\ell \cdot u + 1 \in 1+\mf_v$.
Then the extension of $K$ corresponding to the equation $X^\ell - u'$ is precisely the same as the extension of $K$ corresponding to the equation $Y^\ell - Y + \lambda \cdot f(Y) = u$ for some (explicit) polynomial $f(Y)$; the polynomial $f$ can be computed by making the change of variables $X = \lambda Y + 1$.
Clearly, this polynomial reduces mod $\mf_v$ to the polynomial $Y^\ell - Y = \bar u$, where $\bar u$ denotes the image of $u$ in $k(v)$.

On the other hand, Artin-Schreier theory says that the maximal $(\Z/\ell)$-elementary abelian Galois extension of $k(v)$ is the extension of $k(v)$ generated by roots of polynomials of the form $Y^\ell - Y = \bar u$ for $\bar u \in k(v)$.
Thus, the maximal $(\Z/\ell)$-elementary abelian Galois extension of $k(v)$ is a sub-extension of the residue extension corresponding to $K(\sqrt[\ell]{1+\mf_v})|K$.
This immediately implies that $D_v^1 \leq T_v^1$, as required.
\end{remark}

\subsection{Proof of Theorem \ref{thm:main-thm-char-intro}}
\label{sec:detect-decomp-inert}

We use the notation of  Theorem \ref{thm:main-thm-char-intro}.
I.e. $n \in \Nbar$ and $N \geq \Rbf(n)$ are given.
Also, $K$ is a field such that $\Char K \neq \ell$ and $\mu_{2\ell^N} \subset K$.

Using our chosen isomorphism $R_N \cong R_N(1)$, we obtain the same compatible isomorphisms $\Gc_K^{a,m} \cong \Gc_K^a(m)$ for all $m \leq N$, as in the proof of Theorem \ref{thm:main-thm-intro-maximal}.
We furthermore obtain similar isomorphisms $\Gc_F^{a,m} \cong \Gc_F^a(m)$ for all field extensions $F|K$, in a compatible way with the isomorphisms $\Gc_K^{a,m} \cong \Gc_K^a(m)$.
To simplify the exposition, we will abuse the notation and language by making the following compatible identifications for all field extensions $F|K$ and $m \leq N$:
\begin{enumerate}
\item Identify $\Gc_F^{a,m}$ with $\Gc_F^a(m)$ using Kummer theory.
\item For a subgroup $H$ of $F^\times$, identify $\Gal(F^{a,m}|F(\sqrt[\ell^n]{H})) \leq \Gc_F^{a,m}$ with $\Hom(F^\times/H, R_m)$. 
Similarly, for a subgroup $A$ of $\Gc_F^{a,m}$, identify $A$ with $\Hom(F^\times/A^\perp,R_m)$.
\item Identify ``CL-pairs'' with ``C-pairs'' using Theorem \ref{thm:cl-to-c-pairs}.
\item For $A \leq \Gc_F^{a,n}$, the fields $(F^{a,n})^A$ and $F_A$ (defined in \S\ref{sec:restr-char}) are identical.
\end{enumerate}
Finally, we note the following consequence of Proposition \ref{prop:decomp-thy-prop}: If $m \leq N$,  $F$ is an extension of $K$, and $w$ is a valuation of $F$ such that $\Char k(w) \neq \ell$, then $I_w^n = T_w^n$ and $D_w^n = Z_w^n$.

\vskip 5pt
\noindent{\bf Proof of (1):} 

Let $D \leq \Gc_K^{a,n}$ be given and let $L := (K^{a,n})^D$.
Assume first that there exists a CL-group $D' \leq \Gc_L^{a,N}$ such that $(D'_n)_K = D$.
By Theorem \ref{thm:c-groups-to-valuative-char}, there exists a valuative subgroup $I$ of $D$ such that $\Char k(v_I) \neq \ell$, $D \leq D_{v_I}^n$, and $D/I$ is cyclic.

Conversely, assume that there exists a valuation $v$ such that $\Char k(v) \neq \ell$, $D \leq Z_v^n$ and $D / (D \cap T_v^n)$ is cyclic.
Consider $I := D \cap T_v^n$ and choose $f \in D$ such that $D = \langle I,f \rangle$.
Choose a prolongation $w$ of $v$ to $L$.
By Proposition \ref{prop:summary-decomp-thy}, there exits $f' \in Z_{w}^N$ such that $(f'_n)_K = f$.
Moreover, $I$ is contained in the image of the canonical map $T_{w}^N \rightarrow T_v^n$.
Consider $I'$ the pre-image of $I$ in $T_{v'}^N$.
Also by Proposition \ref{prop:summary-decomp-thy}, we see that $D' := \langle I',f' \rangle$ is a CL-group and $(D'_n)_K = D$, as required.

\vskip 5pt
\noindent{\bf Proof of (2):} 

Let $I := (\Ibcl(\Gc_K^{a,N}))_n$.
By Proposition \ref{prop:center-mu}, one has $I = I_{v_I}^n$ and $\Gc_K^{a,n} = D_{v_I}^n$.
Moreover, by Theorem \ref{thm:main-c-groups-char}, one has $\Char k(v_I) \neq \ell$.
Thus, $I = T_v^n$ and $\Gc_K^{a,n} = Z_v^n$, as required.

\vskip 5pt
\noindent{\bf Proof of (3):} 

Let $v \in \Vc_{K,n}$ be given.
Consider $I := I_v^n$ and $D := D_v^n$, and let $L := (K^{a,n})^D$.
Assume first that there exists $D' \leq \Gc_L^{a,N}$ and $I' \leq \Ibcl(D') \leq \Gc_L^{a,N}$ such that $(I'_n)_K = I$ and $(D'_n)_K = D$.
By Theorem \ref{thm:main-c-groups-char}, we deduce that $I$ is valuative, $D \leq D_{v_I}^n$ and $\Char k(v_I) \neq \ell$.
On the other hand, $v = v_I$ by condition (V1) of $v$ and the fact that $I = I_v^n$.
Therefore, we see that $\Char k(v) \neq \ell$.

Conversely, assume that $\Char k(v) \neq \ell$.
Then $I = T_v^n$ and $D = Z_v^n$.
Choose a prolongation $w$ of $v$ to $L$ and consider
\[ I' := T_w^N = I_w^N \leq D_w^N = Z_w^N =: D'. \]
By Proposition \ref{prop:summary-decomp-thy}, one has $(I'_n)_K = I$, $(D'_n)_K = D$ and $I' \leq \Ibcl(D')$, as required.

\vskip 5pt
\noindent{\bf Proof of (4):} 

The proof of this statement is analogous to the proof of Theorem \ref{thm:maximal-I-D-mu}, using the results of \S\ref{sec:restr-char} and Proposition \ref{prop:summary-decomp-thy}, instead of the results of \S\ref{sec:comp-valu}.
We give the detailed argument below.
First, we recall conditions (a),(b),(c) for subgroups $I \leq D$ of $\Gc_K^{a,n}$ as in the statement of the theorem:
\begin{enumerate}
   \item[(a)] There exist $D' \leq \Gc_{K_D}^{a,N}$ such that $((\Ibcl(D'))_n)_K = I$ and $(D'_n)_K = D$.
   \item[(b)] The subgroups $I \leq D \leq \Gc_K^{a,n}$ are maximal with property (a); i.e. if $D \leq E \leq \Gc_K^{a,n}$ and, $E' \leq \Gc_{K_E}^{a,N}$ is given such that $(E'_n)_K = E$ and $I \leq ((\Ibcl(E'))_n)_K$, then $D = E$ and $I = ((\Ibcl(E'))_n)_K$.
   \item[(c)] One has $\Ibcl(D) \neq D$ (i.e. $D$ is not a CL-group).
\end{enumerate}
The claim is that there exists $v \in \Vc_{K,n}'$ such that $I = T_v^n$ and $D = Z_v^n$ if and only if (a),(b),(c) hold true.

Suppose first that $I \leq D$ satisfy conditions (a),(b),(c) above.
By conditions (a),(c) and Theorem \ref{thm:main-c-groups-char}, we see that $I$ is valuative, and, denoting $v := v_I$, one has  $\Char k(v_I) \neq \ell$, $I \leq T_v^n =: T$ and $D \leq Z_v^n =: Z$.

Let $L = (K^{a,n})^Z$ and choose $w$ a prolongation of $v$ to $L$.
By Proposition \ref{prop:summary-decomp-thy}, we see that $T_w^N \leq \Ibcl(Z_w^N)$.
Thus, $I \leq T \leq (\Ibcl(Z_w^N)_n)_K$ and $D \leq ((Z_w^N)_n)_K = Z$.
By condition (b), we see that $I = T$ and $D = Z$.

Now we show that $v$ is an element of $\Vc_{K,n}'$.
Condition (V0') was noted above, while condition (V1') follows from Lemma \ref{lem:v_H}.
Condition (V3') holds since $D = Z$ is not a CL-group by (c), $T = I \leq \Ibcl(D)$ by Proposition \ref{prop:summary-decomp-thy}, and $\Gc_{k(v)}^{a,n} = D/I$.

We must show condition (V2').
As such, suppose that $v_1$ is a refinement of $v$ such that $D = Z_{v_1}^n$.
Choose a prolongation $w_1$ of $v_1$ to $L$ in such a way so that $w_1$ is a refinement of $w$.
By Lemma \ref{sec:suff-many-roots} and Proposition \ref{prop:decomp-thy-prop}, one has $Z_w^N = Z_{w_1}^N$.
By Proposition \ref{prop:summary-decomp-thy}, we have 
\[ I \leq T_{v_1}^n \leq  (\Ibcl(T_{w_1}^N))_n \leq Z_{v_1}^n = Z. \]
By condition (b), we see that $I = T_{v_1}^n$.
Thus condition (V2') holds for $v$ and we see that $v$ is an element of $\Vc_{K,n}'$.

Conversely, let $v \in \Vc_{K,n}'$ be given and consider $I := T_v^n \leq Z_v^n =: D$.
We must show that $I \leq D$ satisfy conditions (a),(b),(c) of the theorem.

First we prove condition (b).
Suppose that $D \leq E \leq \Gc_K^{a,n}$, and that there exists $E' \leq \Gc_{K_E}^{a,N}$ with $I \leq (\Ibcl(E')_n)_K =: J$.
By Theorem \ref{thm:main-c-groups-char}, we see that $J$ is valuative, $\Char k(v_J) \neq \ell$ and $E \leq Z_{v_J}^n$.
Condition (V1') ensures that $v = v_I$, and thus $v_J$ is a refinement of $v_I = v$.
Namely, $Z_{v_J}^n \leq Z_v^n$; since $Z_v^n \leq Z_{v_J}^n$ as well, we deduce $Z_{v_J}^n = Z_v^n$.
Therefore, $T_{v_J}^n = I$ by condition (V2').
Since $I \leq J \leq T_{v_J}^n$, we deduce that $I = J$; thus condition (b) holds true.

Now for condition (a).
We let $L = (K^{a,n})^D$ and choose $w$ a prolongation of $v$ to $L$.
By Proposition \ref{prop:summary-decomp-thy}, one has $T_w^N \leq \Ibcl(Z_w^N) \leq Z_w^N$.
Thus $I \leq (\Ibcl(Z_w^N)_n)_K$ and $((Z_w^N)_n)_K = D$.
We deduce condition (a) by applying condition (b) with $E' = Z_w^N$ and $E = D$.

Lastly, we prove condition (c).
Suppose for a contradiction that $D$ is a CL-group.
Then $\Gc_{k(v)}^{a,n} = D/I$ is a CL-group as well (e.g. by Lemma \ref{lem:C-pair-compat-in-res-fields}).
By Claim (1) of this theorem applied to $\Gc_{k(v)}^{a,n}$, we see that there is a valuation $w'$ of $k(v)$ such that $\Char k(w') \neq \ell$, $\Gc_{k(v)}^{a,n} = Z_{w'}^n$, and $\Gc_{k(v)}^{a,n}/T_{w'}^n$ is cyclic.
Letting $w'' = w' \circ v$ denote the refinement of $v$ associated to $w'$, we deduce that $\Char k(w'') \neq \ell$, $Z_{w''}^n = Z_v^n$ and thus $T_v^n = T_{w''}^n$ by condition (V2').
However, we also know that $Z_{w''}^n / T_{w''}^n = Z_v^n/T_v^n = \Gc_{k(v)}^{a,n}$ is cyclic, which contradicts condition (V3').
Thus condition (c) holds true.

This completes the proof of Theorem \ref{thm:main-thm-char-intro}.

\section{Corollaries}
\label{sec:struct-maxim-pro-ell}

Let $n \in \Nbar$ be given and let $N := \Rbf(n)$.
As in the introduction, we denote by $\Gc_K^{M,n}$ the smallest quotient of $\Gc_K$ for which $\Gc_L^{c,N}$ is a subquotient for all $K \subset L \subset K^{a,n}$.

The group $\Gc_K^{M,n}$ can be described directly as a Galois group, as follows.
Denote by $L^{c,N}$ the extension of $L$ such that $\Gal(L^{c,N}|L) = \Gc_L^{c,N}$.
Take $K^{M,n}$ to be the compositum of the fields $L^{c,N}$ as $L$ varies over all fields such that $K \subset L \subset K^{a,n}$; then $\Gc_K^{M,n} = \Gal(K^{M,n}|K)$. 
It is easy to see that $\Gc_K^{M,n}$ is a characteristic quotient of $\Gc_K$ and the assignment $\Gc_K \mapsto \Gc_K^{M,n}$ is functorial in $\Gc_K$.

\begin{cor}
\label{cor:application-char}
Let $n \in \Nbar$ be given and let $N := \Rbf(n)$.
Let $K$ be a field such that $\Char K=0$ and $\mu_{2\ell^N} \subset K$.
Assume that there exists a field $F$ such that $\ell \neq \Char F > 0$, $\mu_{2\ell^N} \subset F$ and $\Gc_K^{M,n} \cong \Gc_F^{M,n}$.
Then for all $v \in \Vc_{K,n}$, one has $\Char k(v) \neq \ell$.
\end{cor}
\begin{proof}
Observe that for any valuation $v$ of $F$, one has $\Char F = \Char k(v)$.
The corollary therefore follows from Theorem \ref{thm:main-thm-char-intro} claim (3).
\end{proof}

We recall that $k$ is strongly $\ell$-closed provided that all finite extensions $k'|k$ satisfy $(k')^{\times} = (k')^{\times \ell}$.
For instance, any perfect field of characteristic $\ell$ is strongly $\ell$-closed, and all algebraically closed fields are strongly $\ell$-closed.

\begin{cor}
\label{cor:function-fields-galois-groups}
Suppose that $K$ is one of the following:
\begin{itemize}
\item a function field over a number field $k$ such that $\mu_{2\ell} \subset k$, and $\dim(K|k) \geq 1$, or
\item a function field over a strongly $\ell$-closed field $k$ of characteristic $0$ such that $\dim(K|k) \geq 2$.
\end{itemize}
Then there does not exist a field $F$ such that $\mu_{2\ell} \subset F$, $\Char F > 0$ and $\Gc_K \cong \Gc_F$.
\end{cor}
\begin{proof}
Using Corollary \ref{cor:application-char}, it suffices to find a valuation $v \in \Vc_{K,1}$ such that $\Char k(v) = \ell$.
Furthermore, using the argument of Example \ref{ex:prime-divs}, it suffices to find a valuation $v$ of $K$ such that (1) $\Gamma_v$ contains no non-trivial $\ell$-divisible convex subgroups, and (2) $k(v)$ is a function field over a perfect field of characteristic $\ell$.
In both cases, if $\dim(K|k) \geq 2$, there exists such a valuation, taking, for example, $v$ a quasi-prime divisor prolonging the $\ell$-adic valuation of $\Q \subset k$; see e.g. the Appendix of \cite{Pop2006} and in particular Facts 5.4-5.6 and Remark 5.7 of loc.cit. for more on quasi-prime divisors.

On the other hand, if $\dim(K|k) = 1$ in the first case, we can take $v$ to be a Gauss valuation prolonging the $\ell$-adic valuation.
For a geometric construction, choose a model for $K$, say $\mathcal{X} \rightarrow \Spec \Oc_\ell$, where $\Oc_\ell$ denotes some prolongation of the $\ell$-adic valuation to $k$; then take $v$ the valuation associated to some prime divisor in the special fiber of $\mathcal{X} \rightarrow \Spec \Oc_\ell$. 
\end{proof}

\part{The Main Theorem of C-Pairs}

\section{Proof of Theorem \ref{thm:main-c-pairs-thm}}
\label{sec:proof-theorem}

We will now prove Theorem \ref{thm:main-c-pairs-thm}.
The proof will proceed in two main steps:
First, we will prove the theorem for $n \in \Nb$ and then prove it for $n = \infty$ with a limit argument using the first case.

\subsection{Case $n \neq \infty$}
Let $n \in \Nb$ be given.
To simplify the notation, we denote $N := \Nfr(n) = \Mfr_1(\Nfr'(n))$, $N' := \Nfr'(n)$ and $M := \Mfr_1(n)$ as defined in \S\ref{sec:main-theorem-c}.
We let $f,g \in \Gc_K^a(n)$ be given and assume that there exist $f'',g'' \in \Gc_K^a(N)$ lifts of $f,g$ which form a C-pair.
We denote the map $(f,g) : K^\times \rightarrow R_n \times R_n$ by $\Psi$.
We denote $\ker(f) \cap \ker(g)$ by $T$ and note that $T = \ker\Psi$.

The goal of the theorem is to show that there exists a valuation $v$ of $K$ such that $f,g \in D_v(n)$ and $\langle f,g \rangle / (\langle f,g \rangle \cap I_v(n))$ is cyclic.
To achieve this goal, we will prove the that there exists a valuation $v$ of $K$ such that $1+\mf_v \leq T$ and $(\Oc_v^\times \cdot T) / T$ is cyclic.
If such a valuation $v$ exists, we would obtain:
\[ \langle f,g \rangle = \Hom(K^\times/T,R_n) \leq D_v(n) \]
since $1+\mf_v \leq T$.
Also, we would have
\[ \langle f,g \rangle \cap I_v(n) = \Hom(K^\times/(T \cdot \Oc_v^\times),R_n). \]
Since $(\Oc_v^\times \cdot T) / T$ is cyclic, Pontryagin duality would imply that $\langle f,g \rangle / (\langle f,g \rangle \cap I_v(n))$ is cyclic as well, thus proving Theorem \ref{thm:main-c-pairs-thm} for $n \in \Nb$.
To summarize, it suffices to prove the following claim.
\begin{claim}
\label{claim:exists-a-valuation}
In the notation above, there exists a valuation $v$ of $K$ such that $1+\mf_v \leq T$ and $(\Oc_v^\times \cdot T) / T$ is cyclic.
\end{claim}

We will use the theory of rigid elements, our summary Theorem \ref{thm:rigid-summary} in particular, in order to produce such a valuation $v$.
We denote by $H$ the subgroup of $K^\times$ which is generated by $T$ and all $x \in K^\times \smallsetminus T$ such that $1+x \notin T \cup x \cdot T$.
Equivalently, $H$ is generated by $T$ and all $x \in K^\times$ such that $\Psi(x) \neq 0$ and $\Psi(1+x) \neq \Psi(1),\Psi(x)$.
Claims (1) and (3) of Theorem \ref{thm:rigid-summary} immediately reduce the proof of Claim \ref{claim:exists-a-valuation} to proving the following {\bf Key Claim}.

\begin{claim}[Key Claim]
\label{claim:key-claim-H-mod-T-cyclic}
In the notation above, $H/T$ is cyclic.
\end{claim}

In the case where $n=1$, Claim \ref{claim:key-claim-H-mod-T-cyclic} can be deduced from Koenigsmann's ``$\ell$-rigid calculus;" see \cite{Koenigsmann1998} Lemma 3.3, a form of which also appears in \cite{Koenigsmann1995}, and/or \cite{Efrat1999} Proposition 3.2.
The proof of this claim for a general $n \in \Nb$ is much more technical and this is the main content in the rest of this subsection.

Before we dive in to the proof of Claim \ref{claim:key-claim-H-mod-T-cyclic} for an arbitrary $n \in \Nb$, we give a simplified proof, which works only for $n = 1$, in order to illustrate the geometric origins of our general proof.
This will also provide an overall summary of the general argument for Claim \ref{claim:key-claim-H-mod-T-cyclic}.

\begin{proof}[Proof of Claim \ref{claim:key-claim-H-mod-T-cyclic} for $n = 1$]
The main benefit of working with $n = 1$ is that $R_n = \Z/\ell$ is a field, and $K^\times/T$ can be considered as a vector space over $\Z/\ell$.
Let $x,y \in K^\times \smallsetminus T$ be given such that $\Psi(1+x) \neq \Psi(1),\Psi(x)$ and $\Psi(1+y) \neq \Psi(1),\Psi(y)$.
It suffices to show that $\langle \Psi(x),\Psi(y) \rangle$ has $\Z/\ell$-dimension $1$.
This will imply that $\dim_{\Z/\ell}(H/T) = 1$, and thus prove Claim \ref{claim:key-claim-H-mod-T-cyclic}.

We consider $\Psi$ as a map into $\Pbb^2(\Z/\ell)$ via the composition:
\[ \Psi : K^\times \rightarrow (\Z/\ell)^2 = \Abb^2(\Z/\ell) \subset \Pbb^2(\Z/\ell), \]
by identifying $(i,j) \in (\Z/\ell)^2 = \Abb^2(\Z/\ell)$ with $(1:i:j) \in \Pbb^2(\Z/\ell)$.
Further abusing the notation, we will write $(i,j)$ for $(1:i:j)$ and $(i:j)$ for $(0:i:j)$, considered as elements of $\Pbb^2(\Z/\ell)$.

We extend $\Psi$ to a function on all of $K$ by formally setting $\Psi(0) = (0,0) = (1:0:0)$. 
It follows from the fact that $f,g$ form a C-pair that, for all $z,w \in K$, $\Psi(z+w)$ lies on a projective line in $\Pbb^2(\Z/\ell)$ which contains both $\Psi(z)$ and $\Psi(w)$.
If $\Psi(z) \neq \Psi(w)$, then this line is unique and we denote it by $\mathfrak{L}(\Psi(z),\Psi(w))$.

Assume, for a contradiction, that $\langle \Psi(x),\Psi(y) \rangle$ has $\Z/\ell$-dimension $2$, and thus $\Psi(x),\Psi(y)$ are $\Z/\ell$-linearly independent.
We recall that:
\begin{enumerate}
\item $\Psi(1+x) \neq \Psi(1),\Psi(x)$, and $\Psi(1+x) \in \mathfrak{L}(\Psi(1),\Psi(x))$.
\item $\Psi(1+y) \neq \Psi(1),\Psi(y)$, and $\Psi(1+y) \in \mathfrak{L}(\Psi(1),\Psi(y))$.
\end{enumerate}
Because of these facts, and our ``for-a-contradiction'' assumption that $\Psi(x),\Psi(y)$ are linearly independent, we can compose $\Psi$ with a (unique) projective-linear automorphism $\delta$ of $\Pbb^2(\Z/\ell)$ to obtain $\Psi' := \delta \circ \Psi$ which satisfies the following conditions:
\begin{enumerate}
\item $\Psi'(1) = (0,0)$.
\item $\Psi'(x) = (1,0)$ and $\Psi'(y) = (0,1)$.
\item $\Psi'(1+x) = (1:0)$ and $\Psi'(1+y) = (0:1)$.
\end{enumerate}
Since $\Psi(z+w) \in \mathfrak{L}(\Psi(z),\Psi(w))$ (if $\Psi(z) \neq \Psi(w)$), and $\delta$ is a projective-linear automorphism of $\Pbb^2(\Z/\ell)$, we also have $\Psi'(z+w) \in \mathfrak{L}(\Psi'(z),\Psi'(w))$.

Our goal will be to show that $\Psi'((m-1)+mx) = (m,0)$ for all integers $m \geq 1$.
This will provide us with a contradiction, since this will mean that the whole projective line $\mathfrak{L}(\Psi'(1),\Psi'(x))$ is contained in the image of $\Psi'$, which implies that the whole projective line $\mathfrak{L}(\Psi(1),\Psi(x))$ lies in the image of $\Psi$.
But this is absurd: the image of $\Psi$ is contained in $\Abb^2(\Z/\ell)$ while $\mathfrak{L}(\Psi(1),\Psi(x))$ is not contained in $\Abb^2(\Z/\ell)$.

The idea is to write an element of the form $i+jx+ky$, where $i,j,k$ are certain specific integers, as a sum (or difference) in two different ways; then $\Psi'(i+jx+ky)$ must lie on the intersection of the corresponding projective lines.

For example, consider the fact that $1+x+y = (1+x)+y = (1+y)+x$.
Since $1+x+y = (1+x)+y$, the point $\Psi'(1+x+y)$ lies on the line $\mathfrak{L}(\Psi'(y),\Psi'(1+x)) = \mathfrak{L}((0,1),(1:0))$; since $1+x+y = (1+y)+x$, this point also lies on the line $\mathfrak{L}(\Psi'(x),\Psi'(1+y)) = \mathfrak{L}((1,0),(0:1))$.
This implies that $\Psi'(1+x+y) = (1,1)$, and this is precisely step (1) below.
In the steps below, we omit the discussion concerning intersections of lines, and only give the appropriate sum decompositions and the associated calculation of $\Psi'$.
\begin{enumerate}
\item $\Psi'(1+x+y) = (1,1)$ since 
\begin{align*}
1+x+y &= (1+x)+y \\ 
&=(1+y)+x.
\end{align*}
\item $\Psi'(2+x+y) = (1:1)$ since 
\begin{align*}
2+x+y &= (1+x)+(1+y) \\
&= 1+(1+x+y).
\end{align*}
\item By induction on $m$, we prove that $\Psi'((m-1)+mx) = (m,0)$ and $\Psi'(m+mx+y) = (m,1)$. Note that $\Psi'(x) = (1,0)$ and $\Psi'(1+x+y) = (1,1)$ and so the base case $m = 1$ is done.
\item By induction, if $\Psi'(m-1+mx) = (m,0)$ and $\Psi'(m+mx+y) = (m,1)$, then $\Psi'((m+1)+(m+1)x+y) = (m+1,1)$ since 
\begin{align*}
(m+1)+(m+1)x+y &= (1+x)+(m+mx+y) \\
&= ((m-1)+mx)+(2+x+y).
\end{align*}
\item If $\Psi'(m-1+mx) = (m,0)$ and $\Psi'((m+1)+(m+1)x+y) = (m+1,1)$, then $\Psi'(m+(m+1)x) = (m+1,0)$ since
\begin{align*}
m+(m+1)x &= (1+x)+(m-1+mx) \\
&= ((m+1)+(m+1)x+y)-(1+y).
\end{align*}
\end{enumerate}
Thus, we've shown that $\Psi'((m-1)+mx) = (m,0)$ for all integers $m \geq 0$ and this completes the proof by the discussion above.

To shed some more light on the general proof for $n \in \Nb$ which will follow, we observe that the contradiction occurs \emph{precisely} at $m = (1-a)$ where $a \in \Z/\ell$ is the (unique) coefficient with $\Psi(1+x) = a \cdot \Psi(x)$.
Thinking geometrically, this is because our projective automorphism $\delta$ of $\Pbb^2(\Z/\ell)$, which was used to define $\Psi' := \delta \circ \Psi$, satisfies: 
\[ \delta^{-1}(1-a,0) \in \mathfrak{L}((1:0),(0:1)) = \Pbb^2(\Z/\ell) \smallsetminus \Abb^2(\Z/\ell). \qedhere \]
\end{proof}

The idea for the proof of Claim \ref{claim:key-claim-H-mod-T-cyclic} for a general $n \in \Nb$ is to calculate explicit restrictions on $\Psi(i+jx+ky)$, for $i+jx+ky$ which appear in the proof for $n = 1$.
The main difference is that $\Psi(i+jx+ky)$ is no longer uniquely determined by looking at intersection of ``lines'' in $\Z/\ell^n \times \Z/\ell^n$, since these ``lines'' are defined by the vanishing of a certain $2 \times 2$ determinant.
Thus we are only able to prove weaker restrictions on $\Psi(i+jx+ky)$, whereas if $n = 1$, we were able to calculate this point precisely.
Because of this observation, we instead work in $R_{N'}$, where $N'$ was appropriately chosen so that, when projecting down to $R_n$, we obtain the same contradiction as we did in the $n=1$ case.

Another difficulty in the proof for a general $n \in \Nb$ is that we can no longer ``arrange'' the points $\Psi(1)$, $\Psi(x)$, $\Psi(y)$, $\Psi(1+x)$ and $\Psi(1+y)$ by composing with the auxiliary projective isomorphism $\delta$ as we did in the proof for $n = 1$.
The best we can do is a weak rearrangement so that the image of $x$ and $1+x$ lie on the horizontal axis and so that $y$ and $1+y$ lie on the vertical axis; this is what we call $\Phi$ in the proof below.

Because of these two main difficulties, we will need to carry out some very explicit calculations in $R_{N'}$, which correspond to ``intersecting lines'' in the $n = 1$ case. 
Fortunately, the main steps of the proof are essentially the same as the $n=1$ case, and are summarized as follows:
\begin{enumerate}
\item Work with $1+x+y$ using the fact that 
\begin{align*}
1+x+y &= (1+x)+y \\ 
&= (1+y)+x
\end{align*}
\item Work with $2+x+y$ using the fact that 
\begin{align*}
2+x+y &= 2+(1+x+y) \\
&= (1+x)+(1+y)
\end{align*}
\item Work with $(m-1)+mx$ and $m+mx+y$, inductively for integers $m \geq 1$, using the following sum-decompositions:
\begin{align*}
(m+1)+(m+1)x+y &= (m+mx+y)+(1+x) \\
&= ((m-1)+mx)+(2+x+y)
\end{align*}
and 
\begin{align*}
m+(m+1)x &= (1+x)+((m-1)+mx) \\
&= ((m+1)+(m+1)x+y) - (1+y).
\end{align*}
\item Assume that $\langle \Psi(x),\Psi(y) \rangle$ is non-cyclic, and obtain a contradiction precisely at the point $m = (1-a)$ where $\Psi(1+x) = a \cdot \Psi(x)$.
\end{enumerate}

We now turn to the detailed proof of Key Claim \ref{claim:key-claim-H-mod-T-cyclic} for a general $n \in \Nb$, and the rest of the subsection will be devoted to this.
We define $f' := f''_{N'}$ and $g' := g''_{N'}$.
Since $f'',g''$ form a C-pair in $\Gc_K^a(N)$, the pair $f',g'$ is a C-pair in $\Gc_K^a(N')$, and the pair $f,g$ is a C-pair in $\Gc_K^a(n)$.
We recall that the map $(f,g) : K^\times \rightarrow R_n \times R_n$ is denoted by $\Psi$.
We will also denote the map $(f',g') : K^\times \rightarrow R_{N'} \times R_{N'}$ by $\Theta$.
Since $f'',g''$ are a C-pair, one has $\Theta(-x) = \Theta(x)$, and $N = \Mfr(N')$, it follows from Lemma \ref{lem:c-pairs-vs-rigid} that, for all $x \in K^\times$ with $x \neq -1$, the group $\langle \Theta(1+x),\Theta(x) \rangle$ is cyclic.
In particular, the same is true for $\Psi$: for all $x \in K^\times$ with $x \neq -1$, the group $\langle \Psi(1+x),\Psi(x) \rangle$ is cyclic.

\begin{claim}
\label{claim:key-claim-red-to-x-y}
Let $\Sc$ denote the collection of all $x \in K^\times \smallsetminus T$ such that:
\begin{enumerate}
\item One has $\Psi(1+x) \neq \Psi(1),\Psi(x)$.
\item There exists some $a \in R_{N'}^\times$ such that $\Theta(1+x) = a \cdot \Theta(x)$.
\end{enumerate}
Assume that for all $x,y \in \Sc$, the subgroup $\langle \Psi(x),\Psi(y) \rangle$ is cyclic.
Then Claim \ref{claim:key-claim-H-mod-T-cyclic} holds true.
\end{claim}
\begin{proof}
Denote by $H'$ the subgroup of $K^\times$ which is generated by $T$ and $\Sc$.
We first show that $H' = H$.
It is clear from the definition of $H$ that $H' \leq H$; we must therefore show that $H \leq H'$.

Recall that $H$ is generated by $T$ and all $z \notin T$ such that $\Psi(1+z) \neq \Psi(1),\Psi(z)$.
We will show that each such $z$ is contained in $H'$.
For any such $z$, we know that $\langle \Theta(1+z),\Theta(z) \rangle$ is cyclic.
Therefore, there exists a $c \in R_{N'}$ such that either $\Theta(1+z) = c \cdot \Theta(z)$ or $\Theta(z) = c \cdot \Theta(1+z)$.
We will consider these in two different cases.

\vskip 5pt
\noindent{\bf Case:} $\Theta(1+z) = c \cdot \Theta(z)$:

If $\Theta(1+z) = c \cdot \Theta(z)$ and $c \in R_{N'}^\times$, then $z \in \Sc$ and so $z \in H'$.
On the other hand, if $\Theta(1+z) = c \cdot \Theta(z)$ and $c \in R_{N'} \smallsetminus R_{N'}^\times = \ell \cdot R_{N'}$, then 
\[ \Theta(1+z^{-1}) = \Theta(1+z)-\Theta(z) = (c-1) \cdot \Theta(z) = (1-c) \cdot \Theta(z^{-1}) \]
and $(1-c) \in R_{N'}^\times$.
Namely, $z^{-1} \in \Sc$ and thus $z \in H'$.

\vskip 5pt
\noindent{\bf Case:} $\Theta(z) = c \cdot \Theta(1+z)$:

Consider $z' = -(1+z)$.
Since $\Psi(1+z) = \Psi(z')$, and $0 \neq \Theta(z) = c \cdot \Theta(z')$, we see that $z' \notin T$.
Also, $\Theta(1+z') = \Theta(z)$ and thus $\Theta(1+z') = c \cdot \Theta(z')$.
Since $\Theta(z) = c \cdot \Theta(z')$ and $\ker\Theta \leq T$, it suffices to prove that $z' \in H'$ because this will imply that $z \in H'$ as well.

If $c \in R_{N'}^\times$ then $z' \in \Sc$ so that $z' \in H'$ and we're done.
Otherwise, $c  \in R_{N'} \smallsetminus R_{N'}^\times$, and, arguing as above, $(z')^{-1} \in \Sc$ so that $z' \in H'$ in this case as well. 
This proves that $H = H'$.

The assumption of the claim ensures that, for all $x,y \in \Sc$, the group $\langle \Psi(x),\Psi(y) \rangle$ is cyclic.
Since $T = \ker\Psi$, the subgroup $\langle x \cdot T,y \cdot T \rangle$ of $H/T$ is cyclic for all $x,y \in \Sc$.
Because $K^\times/T$ is finite, and the elements of $\Sc$ generate $H'/T = H/T$, this immediately implies Claim \ref{claim:key-claim-H-mod-T-cyclic}, as contended.
\end{proof}

We now fix a pair of elements $x,y \in K^\times \smallsetminus T$ which satisfy the two conditions of Claim \ref{claim:key-claim-red-to-x-y}; i.e. $x,y \in \Sc$.
We also fix $a,b \in R_{N'}^\times$ such that $\Theta(1+x) = a \cdot \Theta(x)$ and $\Theta(1+y) = b \cdot \Theta(y)$.
We will be working with the determinant $D := f'(x) g'(y)-f'(y)g'(x)$ as an element of $R_{N'}$.

By Lemma \ref{lem:c-pairs-vs-rigid}, if $D = 0 \mod \ell^M$, then the group $\langle \Psi(x),\Psi(y) \rangle$ is cyclic.
Using Claim \ref{claim:key-claim-red-to-x-y}, we have therefore reduced the proof of Key Claim \ref{claim:key-claim-H-mod-T-cyclic} -- and thus the proof of Theorem \ref{thm:main-c-pairs-thm} in the case $n \in \Nb$ -- to proving the following simpler and explicit claim:
\begin{claim}
\label{claim:D-mod-ellM-is-zero}
In the notation above, one has $D = 0 \mod \ell^M$.
\end{claim}
We will now focus on the proof of Claim \ref{claim:D-mod-ellM-is-zero}, which will complete the proof of Claim \ref{claim:key-claim-H-mod-T-cyclic} and thus prove Theorem \ref{thm:main-c-pairs-thm} in the case $n \in \Nb$.

We consider the following $R_{N'}$-linear combinations of $f',g'$ in $\Hom(K^\times,R_{N'})$:
\begin{itemize}
\item $p := g'(y) \cdot f' - f'(y) \cdot g'$
\item $q := f'(x) \cdot g' - g'(x) \cdot f'$
\end{itemize}
and we will denote the map $(p,q) : K^\times \rightarrow R_{N'} \times R_{N'}$ by $\Phi$.
To simplify the exposition, we formally extend $\Phi$ to a function on all of $K$ by defining $\Phi(0) = (p(0),q(0))$ to be $(0,0)$.
Since $p,q$ form a C-pair (as they are linear combinations of $f',g'$), we immediately see that, for all $z,w \in K$, the following $2 \times 2$ determinant vanishes:
\[ \left|\begin{array}{cc}
    p(z+w)-p(w) & p(z)-p(w) \\
    q(z+w)-q(w) & q(z)-q(w)
  \end{array}
 \right|  = 0. \]
Indeed, if $z+w,w,z \neq 0$, the vanishing of this determinant follows from the definition of a C-pair.
On the other hand, if $0 \in \{w+z, w, z\}$, this follows from the fact that $\Phi(0) := (0,0) = \Phi(-1)$.

By the definition of $p$ and $q$, one has:
\begin{itemize}
\item $\Phi(x) =  (p,q)(x) = (D,0)$.
\item $\Phi(y) =  (p,q)(y) = (0,D)$.
\end{itemize}
Also, since $\Theta(1+x) = a \cdot \Theta(x)$ and $\Theta(1+y) = b \cdot \Theta(y)$, we deduce:
\begin{itemize}
\item $\Phi(1+x) = (p,q)(1+x) = (aD,0)$.
\item $\Phi(1+y) = (p,q)(1+y) = (0,bD)$.
\end{itemize}

We will use the following notation: $a' := a-1$ and $b' := b-1$, both are elements of $R_{N'}$.
While $a,b$ are units in $R_{N'}$, the elements $a',b'$ need not be units.
However, we recall that $\Psi(1+x) \neq \Psi(1),\Psi(x)$ and that $\Psi(1+y) \neq \Psi(1),\Psi(y)$.
Since $\Psi(1+x) = a_n \cdot \Psi(x)$ and $\Psi(1+y) = b_n \cdot \Psi(y)$, this implies that $a_n \neq 0,1$ and $b_n \neq 0,1$.
In particular $(a')_n,(b')_n \neq 0$.
To summarize:
\begin{itemize}
\item The elements $a$ and $b$ are units in $R_{N'}$, with $\Phi(1+x) = a \cdot \Phi(x)$ and $\Phi(1+y) = b \cdot \Phi(y)$.
\item One has $a' := a-1$, $b' := b-1$ and $(a')_n,(b')_n \neq 0$.
\end{itemize}

To make the notation a bit less cumbersome, we will use the following notational conventions which are motivated by homogeneous coordinates in projective space over a field.
Let $i,j,\gamma, \gamma_1,\gamma_2,\gamma_3 \in \Z/\ell^{N'}$ be given.
We will write:
\[ \gamma : \ \gamma_2 = \gamma_3 \]
to mean that $\gamma\gamma_2 = \gamma\gamma_3$.
Also, we will write $(i,j) = (\gamma_1:\gamma_2:\gamma_3)$ to mean that $i \cdot \gamma_1 = \gamma_2$ and $j \cdot \gamma_1 = \gamma_3$.
Furthermore, we will use the notation $(i,j) = \gamma \cdot (\gamma_1:\gamma_2:\gamma_3)$ to mean that $(i,j) = (\gamma\gamma_1:\gamma\gamma_2:\gamma\gamma_3)$.
Assume that $\gamma$ divides $\gamma'$  in $\Z/\ell^{N'}$; we make the following trivial observations because they will be used later:
\begin{enumerate}
\item $\gamma :  \ \gamma_2  = \gamma_3$ implies $\gamma' : \  \gamma_2 = \gamma_3$.
\item $(i,j) = \gamma \cdot (\gamma_1: \gamma_2 : \gamma_3)$ implies $(i,j) = \gamma' \cdot (\gamma_1 : \gamma_2 : \gamma_3)$.
\end{enumerate}

\begin{claim}
\label{claim:1+x+y}
In the notation above, one has
\[ \Phi(1+x+y) = D\cdot (a'b'+a'+b':Dab':Da'b). \]
\end{claim}
\begin{proof}
For notational simplicity, set $(P_1,Q_1) := \Phi(1+x+y)$.
To prove the claim, we will use the fact that $1+x+y$ can be written as a sum in two ways: $1+x+y = (1+x)+y = (1+y)+x$.
First let us consider $1+x+y = (1+x)+y$. 
Therefore:
\[\left|\begin{array}{cc}
    p(1+x+y)-p(y) & p(1+x)-p(y) \\
    q(1+x+y)-q(y) & q(1+x)-q(y)
  \end{array}
 \right|  = 0. \\
\]
Making the appropriate substitutions:
\[
0 = \left|\begin{array}{cc}
    P_1 & aD \\
    Q_1-D & -D
  \end{array}
 \right|  = D \cdot  \left|\begin{array}{cc}
    P_1 & a \\
    Q_1-D & -1
  \end{array}
 \right|. \]
In other words we deduce that the following equation holds true:
\begin{align}
\label{1.1}
D : \ P_1 + aQ_1 = aD.
\end{align}
By symmetry, using the fact that $1+x+y = (1+y)+x$, the following equation holds true as well:
\begin{align}
\label{1.2}
D : \ bP_1 + Q_1 = bD.
\end{align}
Using equation (\ref{1.2}) to substitute for $Q_1$ in equation (\ref{1.1}), we obtain the following steps:
\begin{enumerate}
\item $D : \ P_1 + a\cdot(bD-bP_1) = aD$
\item $D : \ P_1 + ab\cdot(D-P_1) = aD$
\item $D : \ P_1\cdot(1-ab) = Da\cdot(1-b)$
\item $D : \ P_1\cdot(ab-1) = Dab'$
\end{enumerate}
Since $ab-1 = a'b'+a'+b'$, the following equation holds true:
\begin{align}
\label{1.3}
D : \ P_1\cdot (a'b'+a'+b') = Dab'.
\end{align}
By symmetry, the following equation holds true as well:
\begin{align}
\label{1.4}
D : \ Q_1\cdot (a'b'+a'+b') = Da'b.
\end{align}
Equations (\ref{1.3}) and (\ref{1.4}) together imply that:
\[ \Phi(1+x+y) = D\cdot (a'b'+a'+b':Dab':Da'b),\]
and our claim is proven.
\end{proof}

\begin{claim}
\label{claim:2+x+y}
In the notation above, one has
\[ \Phi(2+x+y) = D^2\cdot(a'+b':ab'D:ba'D). \]
\end{claim}
\begin{proof}
For notational simplicity, we set $(P_2,Q_2) := \Phi(2+x+y)$.
To prove the claim, we will use the fact that $2+x+y$ can be written as a sum in two ways: $2+x+y = 1+(1+x+y) = (1+x)+(1+y)$.

Since $2+x+y = 1+(1+x+y)$, one has:
\[ \left|\begin{array}{cc}
    p(2+x+y) & p(1+x+y) \\
    q(2+x+y) & q(1+x+y)
  \end{array}
\right| = 0 \]
In particular, the following determinant vanishes as well:
\[ D \cdot \left|\begin{array}{cc}
    p(2+x+y) & (a'b'+a'+b') \cdot  p(1+x+y) \\
    q(2+x+y) & (a'b'+a'+b') \cdot  q(1+x+y)
  \end{array}
\right| = 0. \]
By Claim \ref{claim:1+x+y}, we see that:
\[0 = D  \cdot\left|\begin{array}{cc}
    P_2 & Dab' \\
    Q_2 & Dba'
  \end{array}
\right| = D^2\cdot \left|\begin{array}{cc}
    P_2 & ab' \\
    Q_2 & ba'
  \end{array}
\right|. \]
Therefore the following equation holds true:
\begin{align}
\label{2.1}
D^2 : \ b a' P_2 = b' a Q_2.
\end{align}

On the other hand, $2+x+y = (1+x)+(1+y)$ so that:
\[ \left|\begin{array}{cc}
    P_2-p(1+y) & p(1+x)-p(1+y) \\
    Q_2-q(1+y) & q(1+x)-q(1+y)
  \end{array}
\right| = 0. \]
Making the appropriate substitutions:
\[0 = \left|\begin{array}{cc}
    P_2 & aD \\
    Q_2-bD & -bD
  \end{array}
\right| = D \cdot \left|\begin{array}{cc}
    P_2 & a \\
    Q_2-bD & -b
  \end{array}
\right|. \]
Thus, the following equation  holds true:
\begin{align}
\label{2.2}
D : \ b P_2 + a Q_2 = a b D.
\end{align}
Using equation (\ref{2.2}) to substitute for $a\cdot Q_2$ in equation (\ref{2.1}), we deduce the following steps:
\begin{enumerate}
\item $D^2 : \ ba'P_2 = b'\cdot(abD-bP_2)$
\item $D^2 : \ ba'P_2 = bb'\cdot(aD-P_2)$
\item $D^2 : \ P_2\cdot(ba'+bb') = bb'aD$
\end{enumerate}
Since $b$ is a unit in $R_{N'}$, the following equation holds true:
\begin{align}
\label{2.3}
D^2: \ P_2\cdot(a'+b') = b'aD.
\end{align}
By symmetry, we see that the following equation holds true as well:
\begin{align}
\label{2.4}
D^2: \ Q_2\cdot(a'+b') = a'bD.
\end{align}
Equations (\ref{2.3}) and (\ref{2.4}) together imply that:
\[ \Phi(2+x+y) = D^2\cdot(a'+b':ab'D:ba'D), \]
and our claim is proven.
\end{proof}

We now introduce a bit of notation which will help us carry out the inductive portion of our proof.
For a positive integer $m$, and $\gamma \in R_{N'}$, we will consider the following two statements, whose validity depends on $m$ and $\gamma$:
\begin{enumerate}
\item $\mathbf{(P1)}(m,\gamma): \ \ \Phi((m-1)+mx) = \gamma \cdot (a'b'+mb':mDab':0)$.
\item $\mathbf{(P2)}(m,\gamma): \ \ \Phi(m+mx+y) = \gamma \cdot(a'b'+mb'+a':mDab':Da'b)$.
\end{enumerate}

\begin{claim}
\label{claim:inductivestep1}
Let $e,f,g,h,i,j$ be non-negative integers and consider the following elements of $R_{N'}$:
\begin{enumerate}
\item $A := D^{e}(a')^{f}(b')^{g}$.
\item $B := D^{h}(a')^{i}(b')^{j}$.
\end{enumerate}
Let $m$ be a positive integer, and assume that the two statements $\mathbf{(P1)}(m,A)$ and $\mathbf{(P2)}(m,B)$ hold true.
Then the statement $\mathbf{(P2)}(m+1,E')$ holds true with 
\[E' := D^{\max(2,e,h)+1}(a')^{\max(f,i)+1}(b')^{\max(g,j)+1}. \]
\end{claim}
\begin{proof}
To simplify the notation, we will define:
\begin{itemize}
\item $\Delta_0 := a'+b'$.
\item $\Delta_1 := a'b'+mb'$.
\item $\Delta_2 := a'b'+mb'+a'$.
\end{itemize}
Here is a summary of what we know and our assumptions, using this notation:
\begin{enumerate}
\item  Claim \ref{claim:2+x+y} says that $\Phi(2+x+y) = D^2 \cdot (\Delta_0:ab'D:ba'D)$.
\item  $\mathbf{(P1)}(m,A)$ says that $\Phi((m-1)+mx) = A \cdot (\Delta_1:m Dab':0)$.
\item  $\mathbf{(P2)}(m,B)$ says that $\Phi(m+mx+y) = B \cdot (\Delta_2 : mDab':Da'b)$.
\end{enumerate}

We now consider $(P_3,Q_3) := \Phi((m+1)+(m+1)x+y)$.
As before, we will write $(m+1)+(m+1)x+y$ as a sum in two different ways: 
\begin{align*}
(m+1)+(m+1)x+y &= ((m-1)+mx) + (2+x+y) \\ &= (m+mx+y)+(1+x).
\end{align*}
Since $(m+1)+(m+1)x+y = ((m-1)+mx)+(2+x+y)$, we see that
\[ \left|\begin{array}{cc}
    P_3-p((m-1)+mx) & p(2+x+y)-p((m-1)+mx) \\
    Q_3-q((m-1)+mx) & q(2+x+y)-q((m-1)+mx)
  \end{array}
\right| = 0. \]
By multiplying this determinant by $A \cdot \Delta_1$ and using statement $\mathbf{(P1)}(m,A)$, we deduce that:
\[ A \cdot \left|\begin{array}{cc}
    \Delta_1\cdot P_3-mDab' & \Delta_1\cdot p(2+x+y)-mDab' \\
    Q_3 & q(2+x+y)
  \end{array}
\right| = 0. \]
Consider $A' := D^{\max(2,e)}(a')^{f}(b')^{g}$; in particular, one has $D^2 | A'$ and $A|A'$ in $R_{N'}$.
Now we multiply the right-hand column of the matrix above by $\Delta_0 \cdot D^{\max(2,e)-e}$, and use Claim \ref{claim:2+x+y}, to see that:
\[ A' \cdot \left|\begin{array}{cc}
    \Delta_1\cdot P_3-mDab' & \Delta_1\cdot Dab'-\Delta_0\cdot mDab' \\
    Q_3 & Da'b
  \end{array}
\right| = 0. \]
Rearranging some terms, we have
\[ A' \cdot \left|\begin{array}{cc}
    \Delta_1\cdot P_3-mDab' & Dab'\cdot (\Delta_1-\Delta_0m) \\
    Q_3 & Da'b
  \end{array}
\right| = 0. \]
Now we substitute for $\Delta_1$ and $\Delta_0$ to deduce:
\[ A' \cdot \left|\begin{array}{cc}
    \Delta_1\cdot P_3-mDab' & Dab'\cdot (a'b'+mb'-ma'-mb') \\
    Q_3 & Da'b
  \end{array}
\right| = 0. \]
Therefore,
\[ A' \cdot \left|\begin{array}{cc}
    \Delta_1\cdot P_3-mDab' & Dab'\cdot (a'b'-ma') \\
    Q_3 & Da'b
  \end{array}
\right| = 0 \]
and finally
\[ A'Da' \cdot \left|\begin{array}{cc}
    \Delta_1\cdot P_3-mDab' & ab'\cdot (b'-m) \\
    Q_3 & b
  \end{array}
\right| = 0. \]
Thus we obtain the following steps: 
\begin{enumerate}
\item $A'Da': \ b\cdot (a'b'+mb')\cdot P_3 = Q_3\cdot ab'\cdot (b'-m)+mDabb'$.
\item $A'Da': \ bb'\cdot (a'+m)\cdot P_3 = Q_3\cdot ab'\cdot (b'-m)+mDabb'$.
\end{enumerate}
Thus the following equation holds true:
\begin{align}
\label{3.1}
A'Da'b': \ P_3\cdot b\cdot (a'+m) = Q_3\cdot a\cdot (b'-m)+mDab.
\end{align}

Now we will use the fact that $(m+1)+(m+1)x+y = (m+mx+y)+(1+x)$; this implies that:
\[ \left|\begin{array}{cc}
    P_3-p(1+x) & p(m+mx+y)-p(1+x) \\
    Q_3-q(1+x) & q(m+mx+y)-q(1+x)
  \end{array}
\right| = 0. \]
Making the appropriate substitutions, we have
\[ \left|\begin{array}{cc}
    P_3-aD & p(m+mx+y)-aD \\
    Q_3 & q(m+mx+y)
  \end{array}
\right| = 0. \]
Now we multiply the right-hand column of the matrix above by $B \cdot \Delta_2$, and use statement $\mathbf{(P2)}(m,B)$, to deduce that:
\[ B \cdot \left|\begin{array}{cc}
    P_3-aD & mDab'-\Delta_2\cdot aD \\
    Q_3 & Da'b
  \end{array}
\right| = 0. \]
Rearranging some terms, we obtain:
\[ BD \cdot \left|\begin{array}{cc}
    P_3-aD & a\cdot (mb'-\Delta_2) \\
    Q_3 & a'b
  \end{array}
\right| = 0. \]
Now we substitute into $\Delta_2$ to obtain: 
\[ BD \cdot \left|\begin{array}{cc}
    P_3-aD & a\cdot (mb'-a'b'-mb'-a') \\
    Q_3 & a'b
  \end{array}
\right| = 0 \]
and thus
\[ BD \cdot \left|\begin{array}{cc}
    P_3-aD & -aa'\cdot (b'+1) \\
    Q_3 & a'b
  \end{array}
\right| = 0. \]
Recall that $b' = b-1$ and $b$ is a unit; therefore
\[ BD \cdot \left|\begin{array}{cc}
    P_3-aD & -aa'b \\
    Q_3 & a'b
  \end{array} 
\right| = 0 = BDa' \cdot \left|\begin{array}{cc}
    P_3-aD & -a \\
    Q_3 & 1
  \end{array}
\right|. \]
Namely, the following equation holds true:
\begin{align}
\label{3.2}
BDa': \ P_3+aQ_3 = aD.
\end{align}

Recall the definition of $E'$ based on the expression defining $A$ and $B$:
\[ E' = D^{\max(2,e,h)+1}(a')^{\max(f,i)+1}(b')^{\max(g,j)+1}. \]
Also note that $A'Da'b'|E'$ and $BDa'|E'$ in $R_{N'}$.
Now we use equation (\ref{3.2}) in order to substitute for $aQ_3$ in equation (\ref{3.1}); we thus obtain the following steps: 
\begin{enumerate}
\item $E': \ P_3\cdot b\cdot(a'+m) = (aD-P_3)\cdot(b'-m)+mDab$.
\item $E': \ P_3\cdot(b\cdot(a'+m)+b'-m) = aD\cdot(b'-m)+mDab$.
\item $E': \ P_3\cdot(ba'+bm+b'-m) = a\cdot(D\cdot(b'-m)+mDb)$.
\item $E': \ P_3\cdot(ba'+bm+b'-m) = aD\cdot(b'-m+mb)$.
\item $E': \ P_3\cdot(ba'+bm+b'-m) = aD\cdot(b-1-m+mb)$. 
\item $E': \ P_3\cdot(ba'+bm+b'-m) = aD\cdot((m+1)b-(m+1))$.
\end{enumerate}
And since $b' = b-1$, the following equation holds true:
\begin{align}
\label{3.25}
E': \ P_3\cdot (ba'+bm+b'-m) = (m+1)aDb'.
\end{align}

We now expand out $ba'+bm+b'-m$ and $a'b'+(m+1)b'+a'$ in order to verify that the two expressions are identical.
Firstly, we have:
\begin{eqnarray*}
ba'+mb+b'-m & = & b(a-1)+mb+b-1-m \\
&=& ab-b+mb+b-1-m \\
&=& ab+mb-(m+1).
\end{eqnarray*}
Secondly, we have:
\begin{eqnarray*}
a'b'+(m+1)b'+a' & = & (a-1)(b-1)+(m+1)(b-1)+a-1 \\
&=& ab-a-b+1+(m+1)b-(m+1)+a-1 \\
&=& ab+mb-(m+1).
\end{eqnarray*}
Therefore, we indeed have the following equality of elements in $R_{N'}$:
\begin{align}
\label{mbcalc}
 ba'+mb+b'-m = a'b'+(m+1)b'+a'.
\end{align}
Using equation (\ref{mbcalc}) along with equation (\ref{3.25}), we see that the following equation holds true: 
\begin{align}
\label{3.3}
E': \ P_3\cdot (a'b'+(m+1)b'+a') = (m+1)Dab'.
\end{align}
To conclude the proof of the claim, we multiply equation (\ref{3.2}) by $(a'b'+(m+1)b'+a')$, and use equation (\ref{3.3}) (while recalling that $BDa'$ divides $E'$ in $R_{N'}$ and that $a$ is a unit) to obtain the following steps:
\begin{enumerate}
\item $E': \ (m+1)Dab'+aQ_3\cdot (a'b'+(m+1)b'+a') = aD\cdot (a'b'+(m+1)b'+a')$.
\item $E': \ aQ_3\cdot (a'b'+(m+1)b'+a') = aD\cdot (a'b'+a')$.
\item $E': \ Q_3\cdot (a'b'+(m+1)b'+a') = Da'\cdot (b'+1)$.
\end{enumerate}
Since $b'+1 = b$, the following equation holds true:
\begin{align}
\label{3.4}
E': \ Q_3\cdot (a'b'+(m+1)b'+a') = Da'b.
\end{align}
Finally, equations (\ref{3.3}) and (\ref{3.4}) together imply that statement $\mathbf{(P2)}(m+1,E')$ holds true, thereby concluding the proof of the claim.
\end{proof}

\begin{claim}
\label{claim:inductivestep2}
Let $e,f,g,h,i,j$ be non-negative integers and consider the following elements of $R_{N'}$ as in Claim \ref{claim:inductivestep1}:
\begin{enumerate}
\item $A := D^{e}(a')^{f}(b')^{g}$.
\item $B := D^{h}(a')^{i}(b')^{j}$.
\end{enumerate}
Let $m$ be a positive integer, and assume that the two statements $\mathbf{(P1)}(m,A)$ and $\mathbf{(P2)}(m,B)$ hold true.
Then the two statements $\mathbf{(P1)}(m+1,E)$ and $\mathbf{(P2)}(m+1,E)$ holds true with 
\[ E = D^{\max(2,e,h)+2}(a')^{\max(f,i)+1}(b')^{\max(g,j)+1}. \]
\end{claim}
\begin{proof}
We will continue to use the following notation from Claim \ref{claim:inductivestep1}:
\[ E' = D^{\max(2,e,h)+1}(a')^{\max(f,i)+1}(b')^{\max(g,j)+1}. \]
Also, Claim \ref{claim:inductivestep1} says that the statement $\mathbf{(P2)}(m+1,E')$ holds true. 
Since $E = E' \cdot D$, we immediately see that the statement $\mathbf{(P2)}(m+1,E)$ holds true as well.
However, we will use the slightly stronger fact that $\mathbf{(P2)}(m+1,E')$ holds true in our calculations below.

We consider $(P_4,Q_4) := \Phi(m+(m+1)x)$.
Since we will also be working with $(m+1)+(m+1)x+y$ in the proof of this claim, we denote $\Phi(m+1+(m+1)x+y)$ by $(P_3,Q_3)$ as we did in the proof of Claim \ref{claim:inductivestep1}.
As before, the idea is to write $m+(m+1)x$ as a sum (or difference) in two different ways:
\begin{align*}
m+(m+1)x &= ((m-1)+mx)+(1+x) \\ &= ((m+1)+(m+1)x+y) - (1+y).
\end{align*}
 
Similarly to the proof of Claim \ref{claim:inductivestep1}, we will work with the following elements of $R_{N'}$:
\begin{itemize}
\item $\Delta_0 := a'+b'$.
\item $\Delta_1 := a'b'+mb'$.
\item $\Delta_2 := a'b'+mb'+a'$.
\item $\Delta_2' := a'b'+(m+1)b'+a'$.
\end{itemize}

First, we use the fact that $m+(m+1)x = ((m-1)+mx) + (1+x)$ to deduce:
\[ \left|\begin{array}{cc}
    P_4-p(1+x) & p((m-1)+mx)-p(1+x) \\
    Q_4-q(1+x) & q((m-1)+mx)-q(1+x)
  \end{array}
\right| = 0. \]
Making the appropriate substitutions, we have
\[ \left|\begin{array}{cc}
    P_4-aD & p((m-1)+mx)-aD \\
    Q_4 & q((m-1)+mx)
  \end{array}
\right| = 0. \]
Now we multiply this determinant by $A \cdot \Delta_1$, and use statement $\mathbf{(P1)}(m,A)$ to deduce:
\[ A \cdot \left|\begin{array}{cc}
    P_4-aD & mDab'-\Delta_1\cdot aD \\
    Q_4 &  0
  \end{array}
\right| = 0. \]
Since $a$ is a unit, we see that 
\[ A \cdot \left|\begin{array}{cc}
    P_4-aD & mDb'-\Delta_1\cdot D \\
    Q_4 &  0
  \end{array}
\right| = 0. \]
Factoring out a $D$ and substituting into $\Delta_1$ we obtain:
\[ 0 = AD \cdot \left|\begin{array}{cc}
    P_4-aD & mb'-(a'b'+mb') \\
    Q_4 & 0 
  \end{array}
\right| =  AD \cdot \left|\begin{array}{cc}
    P_4-aD & -a'b' \\
    Q_4 & 0
  \end{array}
\right|. \]
Thus, the following equation holds true:
\begin{align}
\label{3.5}
Q_4 \cdot (ADa'b') = 0.
\end{align}

Now we use the fact that $m+(m+1)x = ((m+1)+(m+1)x+y)-(1+y)$ and thus:
\[ \left|\begin{array}{cc}
    P_4-p(1+y) & P_3-p(1+y) \\
    Q_4-q(1+y) & Q_3-q(1+y)
  \end{array}
\right| = 0. \]
Making the appropriate substitutions, we see that
\[ \left|\begin{array}{cc}
    P_4 & P_3 \\
    Q_4-bD & Q_3-bD
  \end{array}
\right| = 0. \]
We multiply the right-hand column of this matrix by $E' \cdot \Delta_2'$, and use the fact that $\mathbf{(P2)}(m+1,E')$ holds true (Claim \ref{claim:inductivestep1}) to deduce that:
\[ 0 = E' \cdot \left|\begin{array}{cc}
    P_4 & (m+1)Dab' \\
    Q_4-bD & Da'b-\Delta'_2 \cdot bD
  \end{array}
\right| =  E'D \cdot \left|\begin{array}{cc}
    P_4 & (m+1)ab' \\
    Q_4-bD & a'b-\Delta'_2 \cdot b
  \end{array}
\right|. \]
Now, we recall that $ADa'b' | E'D$ in $R_{N'}$, and that $E'D = E$.
Therefore, equation (\ref{3.5}) above implies that:
\[ E \cdot \left|\begin{array}{cc}
    P_4 & (m+1)ab' \\
    -bD & a'b-\Delta'_2 \cdot b
  \end{array}
\right| = 0. \]
Moreover, since $b$ is a unit, we obtain
\[ E \cdot \left|\begin{array}{cc}
    P_4 & (m+1)ab' \\
    -D & a'-\Delta'_2
  \end{array}
\right| = 0. \]
Now we substitute into $\Delta_2'$ to obtain:
\[ E \cdot \left|\begin{array}{cc}
    P_4 & (m+1)ab' \\
    -D & a'-(a'b'+(m+1)b'+a')
  \end{array}
\right| = 0. \]
We therefore deduce that
\[ E \cdot \left|\begin{array}{cc}
    P_4 & (m+1)ab' \\
    D & a'b'+(m+1)b'
  \end{array}
\right| = 0. \]
In particular, the following equation holds true: 
\begin{align}
\label{3.6}
{E : \ P_4\cdot (a'b'+(m+1)b') = (m+1)Dab'}.
\end{align}
Equations (\ref{3.5}) and (\ref{3.6}) together imply that the statement $\mathbf{(P1)}(m+1,E)$ holds true, and this completes the proof of the claim.
\end{proof}

\begin{claim}
\label{claim:final-induction}
In the notation above, the following two statements hold true for all integers $m \geq 1$:
\begin{itemize}
\item $\mathbf{(P1)}(m,D^{2m}(a')^{m-1}(b')^{m-1})$.
\item $\mathbf{(P2)}(m,D^{2m}(a')^{m-1}(b')^{m-1})$.
\end{itemize}
In particular, for each integer $m \geq 1$, there exists $\Pi_m \in R_{N'}$ such that:
\begin{align}
\label{eq:Pm}
 D^{2m}(a')^{m-1}(b')^{m} \cdot Dma = D^{2m}(a')^{m-1}(b')^{m} \cdot (a'+m) \cdot \Pi_m. 
\end{align}
\end{claim}
\begin{proof}
We proceed by induction on $m \geq 1$.
For the base case, $m = 1$, we first observe that $\Phi(x) = (D,0) = (a'b'+b':Dab':0)$ (since $a' = a-1$).
Also, by Claim \ref{claim:1+x+y} we have:
\[ \Phi(1+x+y) = D\cdot (a'b'+a'+b':Dab':Da'b). \]
Namely, the statements $\mathbf{(P1)}(1,1)$ and $\mathbf{(P2)}(1,D)$ hold true; in particular, the statements $\mathbf{(P1)}(1,D^2)$ and $\mathbf{(P2)}(1,D^2)$ holds true as well.

The inductive step follows from Claim \ref{claim:inductivestep2}.
More precisely, assume that the statements $\mathbf{(P1)}(m,D^{2m}(a')^{m-1}(b')^{m-1})$ and $\mathbf{(P1)}(m,D^{2m}(a')^{m-1}(b')^{m-1})$ hold true.
By Claim \ref{claim:inductivestep2}, we obtain that $\mathbf{(P1)}(m+1,D^{2(m+1)}(a')^m(b')^m)$ and $\mathbf{(P2)}(m+1,D^{2(m+1)}(a')^m(b')^m)$ are true as well.

Now we prove the statement concerning the existence of $\Pi_m \in R_{N'}$.
This follows from the fact that $\mathbf{(P1)}(m,D^{2m}(a')^{m-1}(b')^{m-1})$ holds true for all $m \geq 1$.
Namely, for all $m \geq 1$, the following equation holds true:
\[ \Phi((m-1)+mx) =  D^{2m}(a')^{m-1}(b')^{m-1} \cdot (a'b'+mb':mDab':0). \]
Therefore, we have:
\[ D^{2m}(a')^{m-1}(b')^{m-1} : \  p((m-1)+mx) \cdot (a'b'+mb') = mDab'. \]
By defining $\Pi_m := p((m-1)+mx) \in R_{N'}$, we see that 
\[  D^{2m}(a')^{m-1}(b')^{m} \cdot Dma = D^{2m}(a')^{m-1}(b')^{m} \cdot (a'+m) \cdot \Pi_m, \]
as required.
\end{proof}

At last, we are prepared to prove our original Claim \ref{claim:D-mod-ellM-is-zero}, that $D = 0 \mod \ell^M$, and therefore complete the proof of Theorem \ref{thm:main-c-pairs-thm} in the case where $n \in \Nb$.
\begin{proof}[Proof of Claim \ref{claim:D-mod-ellM-is-zero}]
For non-zero elements $\eta \in R_{N'}$ we will define $\mathbf{o}(\eta) := {\rm ord}_\ell(\tilde\eta)$ where $\tilde\eta$ denotes some lift of $\eta$ to $\Z_\ell$ and ${\rm ord}_\ell$ denotes the $\ell$-adic valuation.
We observe that $\mathbf{o}(rt) = \mathbf{o}(r)+\mathbf{o}(t)$ if $rt \neq 0$ in $R_{N'}$.

Assume, for a contradiction, that $D \neq 0 \mod \ell^M$.
Thus $\mathbf{o}(D) \leq M-1 = 2(n-1)$.
Furthermore, observe that $\mathbf{o}(a') \leq n-1$.
Let $m \in \Z$ be a representative for $(-a' \mod \ell^{3n-2})$ such that $1 \leq m \leq \ell^{3n-2}-1$.
Observe that the following inequality holds merely by the definition of $N'$:
\[N' = (6\ell^{3n-2}-7)(n-1)+3n-2 \geq (6m-1)(n-1)+3n-2. \]

By abuse of notation, we will also write $m$ for the image of $m$ in $R_{N'} = \Z/\ell^{N'}$.
Since $(-a')_n \neq 0$, we see that $\mathbf{o}(m) \leq n-1$.
Let us now consider the orders of the elements in the left-hand-side of Equation (\ref{eq:Pm}).
Since $\mathbf{o}(D) \leq 2n-2$ and $\mathbf{o}(a'),\mathbf{o}(b'),\mathbf{o}(m) \leq n-1$ we deduce that:
\[ 2m \cdot\mathbf{o}(D) + (m-1) \cdot\mathbf{o}(a') + m \cdot\mathbf{o}(b') +\mathbf{o}(D)+\mathbf{o}(m) < (6m-1)(n-1)+3n-2 \leq N'. \]
Moreover, we recall that $\mathbf{o}(a) = 0$ as $a$ is a unit.
Thus left-hand-side of equation (\ref{eq:Pm}) is non-zero as an element of $R_{N'}$.
Thus, by Equation (\ref{eq:Pm}), we deduce that 
\[ \mathbf{o}(D)+\mathbf{o}(m) = \mathbf{o}(a'+m) + \mathbf{o}(\Pi_m).\]

Since $\mathbf{o}(D) \leq 2n-2$ and $\mathbf{o}(m) \leq n-1$, we deduce that $\mathbf{o}(D)+\mathbf{o}(m) \leq 3n-3$.
However, $a'+m = 0 \mod \ell^{3n-2}$ which means that $\mathbf{o}(a'+m) \geq 3n-2$.
This yields our contradiction since:
\[ 3n-3 \geq \mathbf{o}(D)+\mathbf{o}(m) = \mathbf{o}(a'+m)+\mathbf{o}(\Pi_m) \geq 3n-2. \]
We therefore deduce that $D = 0 \mod \ell^M$, as required.
\end{proof}

Using the discussion preceding Claim \ref{claim:D-mod-ellM-is-zero}, this completes the proof of Theorem \ref{thm:main-c-pairs-thm} for $n \in \Nb$.

\subsection{Case $n = \infty$} 
Let us now show how to deduce the theorem for $n = \infty$; this will follow from a limit argument using the $n \in \Nb$ case proved above.

Let $f,g \in \Gc_K^a(\infty)$ be a given C-pair.
Equivalently, $f_n,g_n$ form a C-pair for all $n \in \Nb$.
Consider the following subgroups of $K^\times$:
\[  T := \ker f \cap \ker g, \ \ T_n := \ker f_n \cap \ker g_n. \]
Then $T_n \geq T_{n+1}$ and $T = \bigcap_n T_n$.
Let $\Psi$ denote the map $(f,g) : K^\times \rightarrow R_\infty \times R_\infty$ and let $\Psi_n$ denote the map $(f_n,g_n) : K^\times \rightarrow R_n \times R_n$.
Denote by $H$ the subgroup generated by $T$ and all $x \notin T$ such that $\Psi(1+x) \neq \Psi(1),\Psi(x)$.
Arguing as in the previous case using Theorem \ref{thm:rigid-summary}, it suffices to prove that $\Hom(K^\times/T,\Z_\ell) / \Hom(K^\times/H,\Z_\ell)$ is cyclic.
In order to show this, it suffices to prove that $(T_n \cdot H) / T_n$ is cyclic for all $n \in \Nb$.

For each $n \in \Nb$ denote by $H_n$ the subgroup of $K^\times$ which is generated by $T_n$ and all $x \notin T_n$ such that $\Psi_n(1+x) \neq \Psi_n(1),\Psi_n(x)$.
If $\Psi_n(x) \neq 0$ and $\Psi_n(1+x) \neq \Psi_n(1),\Psi_n(x)$ then also $\Psi_{N}(x) \neq 0$ and $\Psi_{N}(1+x) \neq \Psi_{N}(1),\Psi_{N}(x)$ for all $N \geq n$.
Thus, $H_n \leq T_n \cdot H_{N}$ and thus $H_n/T_n \leq (T_n \cdot H_{N}) / T_n$.
Therefore $(T_n \cdot H) / T_n = \bigcup_{N \geq n} (T_n \cdot H_N) / T_n$ is an inductive union.
By the proof of the $n \in \Nb$ case (Claim \ref{claim:key-claim-H-mod-T-cyclic} in particular), the quotient $H_N/T_N$ is cyclic for all $N$.
Thus $(T_n\cdot  H_N)/T_n$ is cyclic for all $N \geq n$.
Therefore, $(T_n \cdot H) / T_n$ is cyclic, as required.



\end{document}